\documentclass{article}
\usepackage{amsmath}
\usepackage{amsfonts,amssymb,fontenc, hyperref}
\usepackage{theorem}
\usepackage{tikz}
\usepgflibrary{arrows}
\usetikzlibrary{matrix}

\usepackage[all]{xy}

\input amssym.def

\numberwithin{equation}{section}

\numberwithin{figure}{section}

\makeatletter
\newcommand\qedsymbol{\hbox{$\Box$}}
\newcommand\qed{\relax\ifmmode\Box\else
  {\unskip\nobreak\hfil\penalty50\hskip1em\null\nobreak\hfil\qedsymbol
  \parfillskip=\z@\finalhyphendemerits=0\endgraf}\fi}

\newenvironment{proof}[1][{}]{\par\noindent Proof{#1}. }{\qed}

\newcommand{\bfzero}{{\bf 0}}

\newcommand{\ed}{{\bullet\hspace{-0.05cm}-\hspace{-0.05cm}\bullet}}

\newcommand{\lp}{{\circlearrowleft}}
\newcommand{\nl}{\not{\circlearrowright}}

\newcommand{\dGra}{{\mathsf{dGra}}}

\newcommand{\GC}{\mathsf{GC}}

\newcommand{\fGC}{\mathsf{fGC}}
\newcommand{\dfGC}{\mathsf{dfGC}}
\newcommand{\dfGCo}{\mathsf{dfGC}^{\oplus}}
\newcommand{\fGCo}{\mathsf{fGC}^{\oplus}}
\newcommand{\GCo}{\mathsf{GC}^{\oplus}}
\newcommand{\QCo}{\mathsf{QC}^{\oplus}}

\newcommand{\dgra}{\mathsf{dgra}}

\newcommand{\gra}{\mathsf{gra}}

\newcommand{\coCom}{{\mathsf{coCom}}}

\newcommand{\Conv}{{\mathrm{Conv}}}

\newcommand{\grVect}{{\mathsf{grVect}}}

\renewcommand{\c}{{\circ}}

\newcommand{\End}{\mathsf{End}}

\renewcommand{\span}{\mathsf{span}}

\newcommand{\Hom}{\mathrm{Hom}}
\newcommand{\Home}{\mathsf{Home}}
\newcommand{\Aut}{{\mathrm {Aut}}}

\newcommand{\Av}{\mathrm {Av}}
\newcommand{\red}{\mathrm{red}}

\newcommand{\Gr}{{\mathrm {Gr}}}

\renewcommand{\Im}{{\mathrm{Im}}}

\newcommand{\id}{{\mathsf{ i d} }}

\newcommand{\Sh}{\mathrm{Sh}}
\newcommand{\conn}{\mathrm{conn}}

\newcommand{\wh}[1]{{\widehat{#1}}}
\newcommand{\ti}[1]{{\tilde{#1}}}

\newcommand{\und}[1]{{\underline{#1}}}

\newcommand{\dia}{\diamond}
\newcommand{\hs}{\heartsuit}

\newcommand{\al}{{\alpha}}
\newcommand{\la}{{\lambda}}

\newcommand{\bul}{{\bullet}}

\newcommand{\mb}{{\mathfrak{b}}}

\newcommand{\mmp}{\mathfrak{p}}

\newcommand{\grt}{{\mathfrak{grt}}}

\newcommand{\si}{{\sigma}}
\newcommand{\ga}{{\gamma}}

\newcommand{\vf}{{\varphi}}
\newcommand{\ve}{{\varepsilon}}
\newcommand{\ka}{{\kappa}}
\newcommand{\G}{{\Gamma}}
\newcommand{\cF}{{\cal F}}
\newcommand{\pa}{{\partial}}

\newcommand{\bsi}{{\bf s}^{-1}\,}

\newcommand{\bs}{{\bf s}}

\newcommand{\cC}{{\mathcal C}}

\newcommand{\cG}{\mathcal{G}}

\newcommand{\cW}{\mathcal{W}}

\newcommand{\cV}{{\cal V}}

\newcommand{\bbK}{{\mathbb K}}

\newcommand{\bbZ}{{\mathbb Z}}
\newcommand{\bbQ}{{\mathbb Q}}

\newcommand{\La}{{\Lambda}}
\newcommand{\Ups}{{\Upsilon}}

\newcommand{\te}{\theta}

\newcommand{\de}{{\delta}}

\newcommand{\sgn}{\mathrm{sgn}}

\newcommand{\Frame}{{\mathsf{Frame}}}
\newcommand{\tensor}{\otimes}
\newcommand{\bbS}{\mathbb{S}}

\newcommand{\maps}{\colon}

\newcommand{\gray}[1]{\textcolor{gray}{#1}}

\newcommand{\Gim}{{\gimel}}

\date{}

\newtheorem{thm}{Theorem}[section]
\newtheorem{defi}[thm]{Definition}

\newtheorem{lem}[thm]{Lemma}
\newtheorem{cor}[thm]{Corollary}

\newtheorem{prop}[thm]{Proposition}
\newtheorem{claim}[thm]{Claim}
\newtheorem{cond}[thm]{Condition}

\theorembodyfont{\rm}
\newtheorem{example}[thm]{Example}

\newtheorem{remark}[thm]{Remark}

\title{The cohomology of the full directed graph complex}

\author{Vasily A. Dolgushev and Christopher L. Rogers}
\date{}

\begin{document}


\maketitle

\begin{abstract}
In his seminal paper ``Formality conjecture'', M. Kontsevich introduced a graph complex $\GC_{1ve}$
closely connected with the problem of constructing a formality quasi-isomorphism for Hochschild cochains.
In this paper, we express the cohomology of the full directed graph complex $\dfGC$ 
explicitly in terms of the cohomology of $\GC_{1ve}$. Applications of our results include a recent work 
by the first author which completely characterizes homotopy classes of formality quasi-isomorphisms 
for Hochschild cochains in the stable setting.
\end{abstract}


\tableofcontents

\section{Introduction}
Graph complexes provide us with a large supply of intriguing questions and conjectures\footnote{This list of references is 
far from complete.} \cite{Bar-Natan-GC}, \cite{C-Vogtmann},  
\cite{hairy-stuff}, \cite{slides},  \cite{notes}, \cite{Ham-Lazar}, \cite{Ham-Lazar1}, 
\cite{KWZ}, \cite{KWZ11}, \cite{K-conj}, 
\cite{K-noncom-symp}, \cite{Thomas-Ger-knots},
\cite{Thomas-slides}, \cite{Thomas}, \cite{WZ}. One source of the motivation for 
working with graph complexes comes from the study of embedding spaces \cite{A-Turchin}, \cite{Cattaneo}, 
\cite{K-M-Volic}, \cite{Sinha}, \cite{Thomas-Ger-knots}.  Another source \cite{Ham}, \cite{Igusa}, \cite{Sergei-Thomas} 
comes from the study of moduli spaces of smooth complex algebraic curves.  

This paper is devoted to the full directed graph complex $\dfGC$ which generalizes 
Kontsevich's graph complex $\GC_{1ve}$ from \cite[Section 5.2]{K-conj} and its study
is motivated by the fact that it ``acts on'' homotopy classes of stable formality 
quasi-isomorphisms \cite{stable}. More precisely, using the full directed graph complex $\dfGC$, 
one can describe all homotopy classes of formality quasi-isomorphisms for 
Hochschild cochains in the ``stable setting''. 

The graded vector space of $\dfGC$ is ``assembled from'' directed graphs (possibly with loops) 
with some additional data. $\dfGC$ is naturally a graded Lie algebra and the 
graph $\G_{\ed}$ with a single edge connecting two vertices is a Maurer-Cartan element 
in $\dfGC$. So the differential on $\dfGC$ is defined as the adjoint action of $\G_{\ed}$.  

It is easy to see that the graph $\G_{\lp}$ which consists of the single loop and 
the polygon\footnote{To get a vector in $\dfGC$ from the undirected graph $\G^{\dia}_{4m+1}$, 
we have to choose a total order on the set of edges and take the sum of directed graphs which are 
obtained from $\G^{\dia}_{4m+1}$ by choosing all possible directions on edges.} 
$\G^{\dia}_{4m+1}$ with $4m+1$ edges (for $m \ge 1$) shown in figure \ref{fig:4m1-intro}
are non-trivial cocycles in $\dfGC$. 
\begin{figure}[htp]
\centering 
\begin{tikzpicture}[scale=0.5, >=stealth']
\tikzstyle{w}=[circle, draw, minimum size=4, inner sep=1]
\tikzstyle{b}=[circle, draw, fill, minimum size=4, inner sep=1]
\node [b] (b1) at (4,0) {};
\draw (4.4,0) node[anchor=center, color=gray] {{\small $5$}};
\node [b] (b2) at (3.54,-1.86) {};
\draw (4,-1.86) node[anchor=center, color=gray] {{\small $6$}};
\node [b] (b3) at (2.27,-3.29) {};
\draw (2.6,-3.7) node[anchor=center, color=gray] {{\small $7$}};
\node [b] (b4) at (0.48,-3.97) {};
\draw (0.48,-4.5) node[anchor=center, color=gray] {{\small $8$}};
\node [b] (b5) at (-1.42,-3.74) {};
\draw (-1.5, -4.3) node[anchor=center, color=gray] {{\small $9$}};
\node [b] (b6) at (-2.99,-2.65) {};
\draw (-3.3,-3.2) node[anchor=center, color=gray] {{\small $10$}};
\node [b] (b7) at (-3.88,-0.96) {};
\draw (-4.5,-0.96) node[anchor=center, color=gray] {{\small $11$}};
\node [b] (b9) at (-2.99,2.65) {};
\draw (-4,3.2) node[anchor=center, color=gray] {{\small $4m+1$}};
\node [b] (b10) at (-1.42,3.74) {};
\draw (-1.42,4.3) node[anchor=center, color=gray] {{\small $1$}};
\node [b] (b11) at (0.48,3.97) {};
\draw (0.48,4.6) node[anchor=center, color=gray] {{\small $2$}};
\node [b] (b12) at (2.27,3.29) {};
\draw (2.27,3.9) node[anchor=center, color=gray] {{\small $3$}};
\node [b] (b13) at (3.54,1.86) {};
\draw (4,1.86) node[anchor=center, color=gray] {{\small $4$}};
\draw (-4,0.2) node[anchor=center] {{\large $\vdots$}};
\draw  (b1) edge (b2);
\draw  (b2) edge (b3);
\draw  (b3) edge (b4);
\draw  (b4) edge (b5);
\draw  (b5) edge (b6);
\draw  (b6) edge (b7);
\draw  (b9) edge (b10);
\draw (b10) edge (b11);
\draw (b11) edge (b12);
\draw (b12) edge (b13);
\draw (b13) edge (b1);
\end{tikzpicture}
\caption{The graph $\G^{\dia}_{4m + 1}$} \label{fig:4m1-intro}
\end{figure}

There is an obvious embedding 
$$
\GC_{1ve} \hookrightarrow \dfGC
$$
which upgrades to the following map of cochain complexes:
\begin{equation}
\label{main-map}
\Psi :  \bs^{-2} \wh{S} \big( \,\bs^2 \GC_{1ve}  ~ \oplus ~ \bigoplus_{m \ge 0} \bs^2 \bbK v_{4m-1}   \, \big) 
~~\to~~ \dfGC,  
\end{equation}
where $\bs$ (resp. $\bs^{-1}$) denotes the operator which shifts the degree up by $1$
(resp. down by $1$), $\wh{S}$ is the completed (and truncated) symmetric algebra 
$$
\wh{S}(V) : = \prod_{n \ge 1} S^n (V),
$$
$v_{4m-1}$ is a vector of degree $4m-1$ which gets mapped to the cocycle $\G^{\dia}_{4m + 1}$ in $\dfGC$ 
if $m \ge 1$ and $v_{-1}$ is a vector of degree $-1$ which gets mapped to $\G_{\lp}$. 

It turns out that the map $\Psi$ is a quasi-isomorphism of cochain complexes and the 
goal of this paper is to give a careful proof of this statement\footnote{In this paper, we assume that $\bbK$ is 
a field of characteristic zero.} (see Theorem \ref{thm:main}). Using this statement, we 
deduce the following corollary
\begin{cor}
\label{cor:H-dfGC-intro}
For the full directed graph complex $\dfGC$, we have
$$
H^{\bul} (\dfGC) \cong  \bs^{-2} \wh{S} \big(\, \bs^2 H^{\bul}(\GC_{1ve}) ~ \oplus ~  
\bigoplus_{m \ge 0} \bs^{4m+1} \bbK  \,\big).
$$
\end{cor}

Combining Corollary \ref{cor:H-dfGC-intro} with \cite[Theorem 1.1]{Thomas}, we conclude that
\begin{cor}
\label{cor:dfGC-grt}
$$
H^0(\dfGC) \cong \grt_1\,,  \qquad H^{-1}(\dfGC) \cong \bbK [\G_{\lp}]\,, \qquad \textrm{and}
\qquad  H^{\le -2}(\dfGC) = \bfzero\,,
$$
where $\grt_1$ is the Grothendieck-Teichmueller Lie algebra 
introduced by V. Drinfeld in \cite[Section 6]{Drinfeld} and $[\G_{\lp}]$ is the cohomology 
class of the graph $\G_{\lp}$ which consists of a single loop. 
\end{cor}

\subsection{Organization of the paper}
The material of the paper is presented in a way which requires a minimal knowledge of 
prerequisites. For example, a reader, who is unfamiliar with algebraic operads \cite{LV-book}, can easily 
understand most of this paper. 

In Section \ref{sec:dfGC}, we introduce the full directed graph complex $\dfGC$, 
its ``uncompleted'' version $\dfGCo$, its undirected version $\fGC$ and 
the important subcomplexes $\GC_{1ve} \subset \GC \subset \fGC$. 
At the end of this section, we recall the interpretation of $\dfGC$ in terms of 
the convolution Lie algebra \cite[Section 4]{notes}. 

In Section \ref{sec:thm-proof}, we formulate the main result of this paper 
(see Theorem \ref{thm:main}) and its variant for the ``uncompleted'' version $\dfGCo$ of $\dfGC$
(see Theorem \ref{thm:dfGCo}). We deduce Theorem \ref{thm:main} from  Theorem \ref{thm:dfGCo} 
and prove the version of Theorem \ref{thm:main} for the subcomplex $\dfGC^{\nl} \subset \dfGC$
which is ``assembled from'' directed graphs without loops (see Proposition \ref{prop:dfGC-nl-dfGC}). 
It is Proposition \ref{prop:dfGC-nl-dfGC} which is used in paper \cite{stable} to describe homotopy 
classes of formality quasi-isomorphisms for Hochschild cochains in the ``stable setting''.  

In Section \ref{sec:thm-proof}, we also recall the version $\dfGC_d$ of the full directed graph complex 
for an arbitrary \emph{even} integer $d$ ($\dfGC : = \dfGC_2$) and generalize Theorem \ref{thm:main}
to the case of arbitrary even $d$. 

The proof of Theorem \ref{thm:dfGCo} is broken into several parts (see Subsection \ref{sec:outline} for 
the outline of the proof) and the two major parts are presented in Sections \ref{sec:dfGCo-conn-ge3} and 
\ref{sec:GC1ve-to-GC}, respectively. 

Section \ref{sec:dfGCo-conn-ge3} is devoted to the subcomplex $\dfGCo_{\conn, \ge 3}$ which is 
spanned by connected graphs with at least one vertex having valency $\ge 3$. In this section, we prove that 
the natural embedding $\GCo \hookrightarrow \dfGCo_{\conn, \ge 3}$ is a quasi-isomorphism of cochain complexes. 

In Section \ref{sec:GC1ve-to-GC}, we prove that the natural embedding $\GCo_{1ve}  \hookrightarrow \GCo$ 
is also a quasi-isomorphism of cochain complexes. This statement plays an important role in the proof of Willwacher's theorem 
\cite[Theorem 1.1]{Thomas} which links the zeroth cohomology of $\GC$ to
the Grothendieck-Teichmueller Lie algebra $\grt_1$ \cite[Section 4.2]{AT}, \cite[Section 6]{Drinfeld}. 
For this reason, we decided to write a careful proof of this statement. 

The easier parts of the proof of Theorem \ref{thm:dfGCo} are presented in Appendix \ref{app:polygons-paths}.
Finally, Appendix \ref{app:cohomol-T-V2} is devoted to auxiliary statements which are used in Section 
\ref{sec:dfGCo-conn-ge3} and in Appendix \ref{app:polygons-paths}.

\subsection{Notational conventions}
In this paper, the ground field $\bbK$ has characteristic zero and 
$\grVect_{\bbK}$ denotes the category of $\bbZ$-graded 
$\bbK$-vector spaces. 
The notation $\bs$ (resp. $\bsi$) is reserved for the operator 
which shifts the degree up by $1$ (resp. down by $1$), i.e.
$(\bs\, V)^{\bul} : = V^{\bul -1}$ and $(\bsi V)^{\bul} : = V^{\bul +1}\,.$
For a set $X$, we denote by $\span_{\bbK}(X)$ the $\bbK$-vector 
space of finite $\bbK$-linear combinations of elements in $X$. 

The notation $\bbS_n$ is reserved for the symmetric group on $n$ 
letters and  $\Sh_{p_1, \dots, p_k}$ denotes 
the subset of $(p_1, \dots, p_k)$-shuffles 
in $\bbS_{p_1 + \dots + p_k}$, i.e.  $\Sh_{p_1, \dots, p_k}$ consists of 
elements $\si \in \bbS_{p_1 + \dots + p_k}$, such that 
$\si(1) <  \si(2) < \dots < \si(p_1),$ $\si(p_1+1) <  \si(p_1 +2) < \dots < \si(p_1+p_2),$ $\dots$,
$\si(n-p_k+1) <  \si(n-p_k+2) < \dots < \si(n)$.
For a group $G$ acting on a vector space $V$, we denote by $V^G$ (resp. $V_G$)
the space of invariants (resp. the space of coinvariants).  

For a graded vector space (or a cochain complex) $V$, the notation $S(V)$
(resp. $\und{S}(V)$, resp. $\wh{S}(V)$) is reserved for the space of the symmetric 
algebra (resp. the truncated symmetric algebra, resp. the completed and truncated 
symmetric algebra). Namely, 
$$
S(V) : = \bbK \oplus V \oplus S^2(V) \oplus S^3(V) \oplus \dots, 
\qquad 
\und{S}(V) : = V \oplus S^2(V) \oplus S^3(V) \oplus \dots, 
$$
and 
$$
\wh{S}(V) : = \prod_{n \ge 1} S^n(V)\,,  \qquad
S^n(V) : = (V^{\otimes\, n})_{\bbS_n} \,.
$$

$\coCom$ denotes the cooperad (in the category of $\bbK$-vector spaces) which governs 
cocommutative (and coassociative) coalgebras without counit. In other words, 
$\coCom(n)$ is the trivial representation $\bbK$ of $\bbS_n$ for every $n \ge 1$ and $\coCom(0) = \bfzero$. 
The notation $\La$ is reserved for the underlying collection of the endomorphism operad $\End_{\bs \bbK}$. 
It is known that $\La$ is naturally a cooperad (in the category $\grVect_{\bbK}$). So, for a cooperad 
$\cC$  (in the category $\grVect_{\bbK}$) and an integer $m$, we denote by  $\La^m \cC$ the cooperad 
$\La^{\otimes m} \cC$. For example, $\La^2 \coCom(n)$ is the trivial representation of $\bbS_n$ placed 
in degree $2-2n$, if $n \ge 1$ and $\La^2 \coCom(0) = \bfzero$.

By a graph $\G$ we typically mean an undirected graph with a finite set 
of vertices $V(\G)$ and a finite set of edges $E(\G)$. Multiple edges as well as 
loops (i.e. cycles of length $1$) are allowed. A graph $\G$ is directed, if each 
edge of $\G$ is equipped with a direction. A {\it forest} is a finite (undirected) graph $F$ 
without cycles. A {\it tree} is a connected forest. 

For a set $X$ with $r < \infty$ elements we tacitly identify total orders 
on $X$ with labelings of elements of $X$ by numbers $\{1,2, \dots, r\}$.   

\vspace{0.3cm}

\noindent
\textbf{Acknowledgements:} We would like to thank Thomas Willwacher 
for several suggestions concerning various drafts of this paper.   
While this paper was in preparation, C.L.R. visited Philadelphia several 
times and he is thankful to Temple University for hospitality.
V.A.D. acknowledges NSF grants DMS-1161867 and DMS-1501001 
for a partial support. C.L.R. acknowledges the Simons Foundation for 
the partial support through an AMS-Simons Travel Grant.

\section{The graph complexes $\dfGC^{\oplus}$, $\dfGC$ and their modifications}
\label{sec:dfGC}

For a pair of integers $(n, r)$ with $n \ge 1$ and $r \ge 0$, the set $\dgra^r_{n}$ 
consists of directed graphs $\G$ with the set of vertices 
$V(\G) = \{\gray{1,2, \dots, n}\} $ and the set of edges $E(\G) = \{\und{1}, \und{2}, \dots, \und{r}\}$.
An example of an element $\G$ in $\dgra^5_4$ is shown in 
figure\footnote{In figures, the labels for vertices are shown in gray color and the 
labels for edges are underlined numbers.} \ref{fig:exam}. 
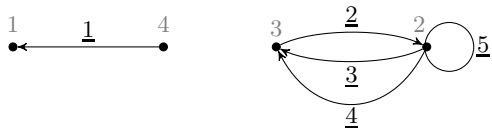
\begin{figure}[htp] 
\centering 
\begin{tikzpicture}[scale=0.5, >=stealth']
\tikzstyle{w}=[circle, draw, minimum size=3, inner sep=1]
\tikzstyle{b}=[circle, draw, fill, minimum size=3, inner sep=1]
\node [b] (b1) at (-5,0) {};
\draw (-5,0.6) node[anchor=center, color=gray] {{\small $1$}};
\node [b] (b3) at (2,0) {};
\draw (2, 0.5) node[anchor=center, color=gray] {{\small $3$}};
\node [b] (b2) at (6,0) {};
\draw (5.8, 0.6) node[anchor=center, color=gray] {{\small $2$}};
\node [b] (b4) at (-1,0) {};
\draw (-1, 0.6) node[anchor=center, color=gray] {{\small $4$}};
\draw [->] (b4) edge (b1);
\draw (-3, 0.4) node[anchor=center] {{\small $\und{1}$}};
\draw [->] (b3) ..controls (3,0.5) and (5,0.5) .. (b2);
\draw (4,0.8) node[anchor=center] {{\small $\und{2}$}}; 
\draw [<-] (b3) ..controls (3,-0.5) and (5,-0.5) .. (b2);
\draw (4,-0.8) node[anchor=center] {{\small $\und{3}$}}; 
\draw [<-] (b3) ..controls (3,-2) and (5,-2) .. (b2);
\draw (4,-1.9) node[anchor=center] {{\small $\und{4}$}}; 
\draw (6.6,0) circle (0.65);
\draw (7.5,0) node[anchor=center] {{\small $\und{5}$}}; 
\end{tikzpicture}
\caption{An example of a directed graph} \label{fig:exam}
\end{figure}

The set $\dgra_{n}^r$ is equipped with the obvious action of the group $\bbS_n \times \bbS_r$. 
Using this action we define the graded vector space 
\begin{equation}
\label{dfGC-oplus}
\dfGC^{\oplus} ~ : =~ \bigoplus_{n \ge 1,~r \ge 0} ~  \big( \bs^{2 n - 2 -r }  
\span_{\bbK}(\dgra_{n}^r) \otimes  \sgn_{r} \big)_{\bbS_n \times \bbS_r } \,,
\end{equation}
where $\sgn_r$ denotes the sign representation of $\bbS_r$.  

Since we take coinvariants with respect to the action of $\bbS_n$ on $\span_{\bbK}(\dgra_{n}^r)$, 
the labels on vertices do not play any essential role. On the other hand, the presence of the 
sign representation of $\bbS_r$ tells us that the change in the labels on edges results in 
a sign factor.  For example, in $\dfGC^{\oplus}$, we have
\begin{equation}
\label{pentagon}
\begin{tikzpicture}[scale=0.5, >=stealth']
\tikzstyle{w}=[circle, draw, minimum size=3, inner sep=1]
\tikzstyle{b}=[circle, draw, fill, minimum size=3, inner sep=1]
\node [b] (b1) at (2,0) {};
\draw (2.5,0) node[anchor=center, color =gray] {{\small $1$}};
\node [b] (b2) at (0.62, 1.90) {};
\draw (0.77, 2.38) node[anchor=center, color=gray] {{\small $2$}};
\node [b] (b3) at (-1.62, 1.18) {};
\draw (-2.02, 1.47) node[anchor=center, color=gray] {{\small $3$}};
\node [b] (b4) at (-1.62, -1.18) {};
\draw (-2.02, -1.47) node[anchor=center, color=gray] {{\small $4$}};
\node [b] (b5) at (0.62, -1.90) {};
\draw (0.77, -2.38) node[anchor=center, color=gray] {{\small $5$}};
\draw (1.62, 1.18) node[anchor=center] {{\small $\und{1}$}};
\draw (-0.62, 1.90) node[anchor=center] {{\small $\und{2}$}};
\draw (-2, 0) node[anchor=center] {{\small $\und{3}$}};
\draw (-0.62, -1.90) node[anchor=center] {{\small $\und{4}$}};
\draw (1.62, -1.18) node[anchor=center] {{\small $\und{5}$}};
\draw [->] (b1) edge (b2) (b2) edge (b3) 
(b3) edge (b4) (b4) edge (b5) (b5) edge (b1);
\draw (4,0) node[anchor=center] {{$ = $}};
\begin{scope}[shift={(8,0)}]
\node [b] (b1) at (2,0) {};
\draw (2.5,0) node[anchor=center, color =gray] {{\small $2$}};
\node [b] (b2) at (0.62, 1.90) {};
\draw (0.77, 2.38) node[anchor=center, color=gray] {{\small $1$}};
\node [b] (b3) at (-1.62, 1.18) {};
\draw (-2.02, 1.47) node[anchor=center, color=gray] {{\small $3$}};
\node [b] (b4) at (-1.62, -1.18) {};
\draw (-2.02, -1.47) node[anchor=center, color=gray] {{\small $4$}};
\node [b] (b5) at (0.62, -1.90) {};
\draw (0.77, -2.38) node[anchor=center, color=gray] {{\small $5$}};
\draw (1.62, 1.18) node[anchor=center] {{\small $\und{1}$}};
\draw (-0.62, 1.90) node[anchor=center] {{\small $\und{2}$}};
\draw (-2, 0) node[anchor=center] {{\small $\und{3}$}};
\draw (-0.62, -1.90) node[anchor=center] {{\small $\und{4}$}};
\draw (1.62, -1.18) node[anchor=center] {{\small $\und{5}$}};
\draw [->] (b1) edge (b2) (b2) edge (b3) 
(b3) edge (b4) (b4) edge (b5) (b5) edge (b1);
\end{scope}
\draw (12.5,0) node[anchor=center] {{$ =  ~ -$}};
\begin{scope}[shift={(16.5,0)}]
\node [b] (b1) at (2,0) {};
\draw (2.5,0) node[anchor=center, color =gray] {{\small $1$}};
\node [b] (b2) at (0.62, 1.90) {};
\draw (0.77, 2.38) node[anchor=center, color=gray] {{\small $2$}};
\node [b] (b3) at (-1.62, 1.18) {};
\draw (-2.02, 1.47) node[anchor=center, color=gray] {{\small $3$}};
\node [b] (b4) at (-1.62, -1.18) {};
\draw (-2.02, -1.47) node[anchor=center, color=gray] {{\small $4$}};
\node [b] (b5) at (0.62, -1.90) {};
\draw (0.77, -2.38) node[anchor=center, color=gray] {{\small $5$}};
\draw (1.62, 1.18) node[anchor=center] {{\small $\und{2}$}};
\draw (-0.62, 1.90) node[anchor=center] {{\small $\und{1}$}};
\draw (-2, 0) node[anchor=center] {{\small $\und{3}$}};
\draw (-0.62, -1.90) node[anchor=center] {{\small $\und{4}$}};
\draw (1.62, -1.18) node[anchor=center] {{\small $\und{5}$}};
\draw [->] (b1) edge (b2) (b2) edge (b3) 
(b3) edge (b4) (b4) edge (b5) (b5) edge (b1);
\end{scope}
\draw (20.5,0) node[anchor=center] {{$ = $}};
\begin{scope}[shift={(24.5,0)}]
\node [b] (b1) at (2,0) {};
\draw (2.5,0) node[anchor=center, color =gray] {{\small $1$}};
\node [b] (b2) at (0.62, 1.90) {};
\draw (0.77, 2.38) node[anchor=center, color=gray] {{\small $2$}};
\node [b] (b3) at (-1.62, 1.18) {};
\draw (-2.02, 1.47) node[anchor=center, color=gray] {{\small $3$}};
\node [b] (b4) at (-1.62, -1.18) {};
\draw (-2.02, -1.47) node[anchor=center, color=gray] {{\small $4$}};
\node [b] (b5) at (0.62, -1.90) {};
\draw (0.77, -2.38) node[anchor=center, color=gray] {{\small $5$}};
\draw (1.62, 1.18) node[anchor=center] {{\small $\und{2}$}};
\draw (-0.62, 1.90) node[anchor=center] {{\small $\und{3}$}};
\draw (-2, 0) node[anchor=center] {{\small $\und{4}$}};
\draw (-0.62, -1.90) node[anchor=center] {{\small $\und{5}$}};
\draw (1.62, -1.18) node[anchor=center] {{\small $\und{1}$}};
\draw [->] (b1) edge (b2) (b2) edge (b3) 
(b3) edge (b4) (b4) edge (b5) (b5) edge (b1);
\end{scope}
\end{tikzpicture}
\end{equation}

It is clear that a graph $\G \in \dgra_n^r$ gives us the zero vector in $\dfGCo$ if and only if 
it has an automorphism which induces an odd permutation in $\bbS_r$. For example,
\begin{equation}
\label{square}
\begin{tikzpicture}[scale=0.5, >=stealth']
\tikzstyle{w}=[circle, draw, minimum size=3, inner sep=1]
\tikzstyle{b}=[circle, draw, fill, minimum size=3, inner sep=1]
\draw (-2.5,1) node[anchor=center] {{$ \G ~ =  $}};
\node [b] (b1) at (2,0) {};
\draw (2.3,-0.3) node[anchor=center, color =gray] {{\small $1$}};
\node [b] (b2) at (2, 2) {};
\draw (2.3, 2.3) node[anchor=center, color=gray] {{\small $2$}};
\node [b] (b3) at (0, 2) {};
\draw (-0.3, 2.3) node[anchor=center, color=gray] {{\small $3$}};
\node [b] (b4) at (0, 0) {};
\draw (-0.3, -0.3) node[anchor=center, color=gray] {{\small $4$}};
\draw (2.4, 1) node[anchor=center] {{\small $\und{1}$}};
\draw (1, 2.4) node[anchor=center] {{\small $\und{2}$}};
\draw (-0.4,  1) node[anchor=center] {{\small $\und{3}$}};
\draw (1, -0.4) node[anchor=center] {{\small $\und{4}$}};
\draw [->] (b1) edge (b2) (b2) edge (b3) 
(b3) edge (b4) (b4) edge (b1);
\draw (4.5,1) node[anchor=center] {{$ = ~ - $}};
\begin{scope}[shift={(7,0)}]
\node [b] (b1) at (2,0) {};
\draw (2.3,-0.3) node[anchor=center, color =gray] {{\small $1$}};
\node [b] (b2) at (2, 2) {};
\draw (2.3, 2.3) node[anchor=center, color=gray] {{\small $2$}};
\node [b] (b3) at (0, 2) {};
\draw (-0.3, 2.3) node[anchor=center, color=gray] {{\small $3$}};
\node [b] (b4) at (0, 0) {};
\draw (-0.3, -0.3) node[anchor=center, color=gray] {{\small $4$}};
\draw (2.4, 1) node[anchor=center] {{\small $\und{2}$}};
\draw (1, 2.4) node[anchor=center] {{\small $\und{3}$}};
\draw (-0.4,  1) node[anchor=center] {{\small $\und{4}$}};
\draw (1, -0.4) node[anchor=center] {{\small $\und{1}$}};
\draw [->] (b1) edge (b2) (b2) edge (b3) (b3) edge (b4) (b4) edge (b1);
\end{scope}
\draw (11.5,1) node[anchor=center] {{$ = ~ - $}};
\begin{scope}[shift={(14,0)}]
\node [b] (b1) at (2,0) {};
\draw (2.3,-0.3) node[anchor=center, color =gray] {{\small $2$}};
\node [b] (b2) at (2, 2) {};
\draw (2.3, 2.3) node[anchor=center, color=gray] {{\small $3$}};
\node [b] (b3) at (0, 2) {};
\draw (-0.3, 2.3) node[anchor=center, color=gray] {{\small $4$}};
\node [b] (b4) at (0, 0) {};
\draw (-0.3, -0.3) node[anchor=center, color=gray] {{\small $1$}};
\draw (2.4, 1) node[anchor=center] {{\small $\und{2}$}};
\draw (1, 2.4) node[anchor=center] {{\small $\und{3}$}};
\draw (-0.4,  1) node[anchor=center] {{\small $\und{4}$}};
\draw (1, -0.4) node[anchor=center] {{\small $\und{1}$}};
\draw [->] (b1) edge (b2) (b2) edge (b3) (b3) edge (b4) (b4) edge (b1);
\end{scope}
\draw (19,1) node[anchor=center] {{$  = ~ - ~ \G.  $}};
\end{tikzpicture}
\end{equation} 
So $\G = 0$ in $\dfGC^{\oplus}$\,.

\begin{defi}
\label{dfn:even-odd}
We will say that a graph $\G \in \dgra_n^r$ is \emph{odd} if it has 
an automorphism which induces an odd permutation in $\bbS_r$. 
Otherwise, we say that $\G$ is \emph{even}. 
\end{defi}
For example, the pentagon shown in \eqref{pentagon} is even and 
the square shown in \eqref{square} is odd. 

It is clear that $\dfGCo$ is the span of isomorphism classes of 
all even graphs\footnote{Note that an even graph $\G$ without labels on edges determines 
a vector in $\dfGCo$ only up to a sign factor. To specify the sign factor, one has to
fix a total order on $E(\G)$ up to any even permutation in $\bbS_{E(\G)}$.}
with the rule that an even graph $\G \in \dgra_n^r$ carries the degree $2n - 2 - r$. 

It is easy to see that every graph which has multiple edges with the same 
direction is odd. Thus, the graph shown in figure \ref{fig:exam} is odd. Hence it 
corresponds to the zero vector in $\dfGCo$. 
On the other hand, there are even graphs with double edges
of opposite directions. An example of such a graph is shown in figure 
\ref{fig:funny-even}.
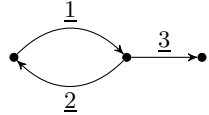
\begin{figure}[htp] 
\centering 
\begin{tikzpicture}[scale=0.5, >=stealth']
\tikzstyle{w}=[circle, draw, minimum size=3, inner sep=1]
\tikzstyle{b}=[circle, draw, fill, minimum size=3, inner sep=1]
\node [b] (b1) at (0,0) {};
\node [b] (b2) at (3,0) {};
\node [b] (b3) at (5,0) {};
\draw [->] (b1) ..controls (1,1) and (2,1) .. (b2);
\draw (1.5, 1.2) node[anchor=center] {{\small $\und{1}$}};
\draw [->] (b2) ..controls (2,-1) and (1,-1) .. (b1);
\draw (1.5, -1.2) node[anchor=center] {{\small $\und{2}$}};
\draw [->] (b2) edge (b3);
\draw (4, 0.4) node[anchor=center] {{\small $\und{3}$}};
\end{tikzpicture}
\caption{An even graph with double edges of opposite directions.
Labels for vertices are not shown} \label{fig:funny-even}
\end{figure}
%
%

\begin{defi}
\label{dfn:concordant}
Let  $\G, \ti{\G}$ be two even elements of $ \dgra_n^r$ for which there exists an isomorphism 
of directed graphs $\vf : \G \to \ti{\G}$. We say that $\G$ and $\ti{\G}$ are \emph{concordant} (resp. \emph{opposite}) 
if $\ti{\G} = \G$ (resp. $\ti{\G} = -\G$) in $\dfGCo$.
\end{defi}
\begin{remark}
\label{rem:sign-si-vf}
Since edges of $\G$ and $\ti{\G}$ are labeled, any isomorphism (of directed graphs) $\vf: \G \to \ti{\G}$ 
gives us a bijection $\{\und{1}, \dots, \und{r}\} \to \{\und{1}, \dots, \und{r}\}$
and hence an element $\si_{\vf} \in \bbS_r$. 
Moreover, since $\G$ (and hence $\ti{\G}$) is even, 
the sign of the permutation $\si_{\vf}$ does not depend on the choice of the isomorphism $\vf$. 
So  $\G$ and $\ti{\G}$ are concordant (resp. opposite) if and only if the permutation 
$\si_{\vf}$ is even (resp. odd).  
\end{remark}

For $\G \in \dgra_n^r$,  $\ti{\G} \in \dgra_m^q$, and $1 \le i \le n$, we denote by
$\G \circ_i \ti{\G}$ the vector in $\span_{\bbK}(\dgra_{m+n-1}^{r+q})$ which is obtained by following these steps: 

\begin{itemize}

\item we shift the labels $i+1, i+2, \dots, n$ on vertices of $\G$ up by $m-1$;

\item we shift all the labels on vertices of $\ti{\G}$ up by $i-1$;

\item we shift all the labels on edges of $\ti{\G}$ up by $r$;

\item finally, we replace vertex $i$ of $\G$ by the graph $\ti{\G}$ (with the new labels)
and sum over all possible ways of attaching the resulting free edges to vertices of
$\ti{\G}$ (with the new labels). 

\end{itemize}

Although $\G \circ_i \ti{\G}$ is defined as a vector in $\span_{\bbK}(\dgra_{m+n-1}^{r+q})$, 
we will often use the same notation $\G \circ_i \ti{\G}$ for its image in the space $\dfGCo$
of coinvariants. For example, a computation of $\G \circ_2 \ti{\G} \in \span_{\bbK}(\dgra_{4}^{4}) $ for two concrete graphs
$\G$ and $\ti{\G}$ is shown in figure \ref{fig:ins} and the image of $\G \circ_2 \ti{\G}$ in $\dfGCo$ is shown in 
figure \ref{fig:ins-dfGC}. This image has only three terms because the third graph in the second line of 
figure \ref{fig:ins} is odd.
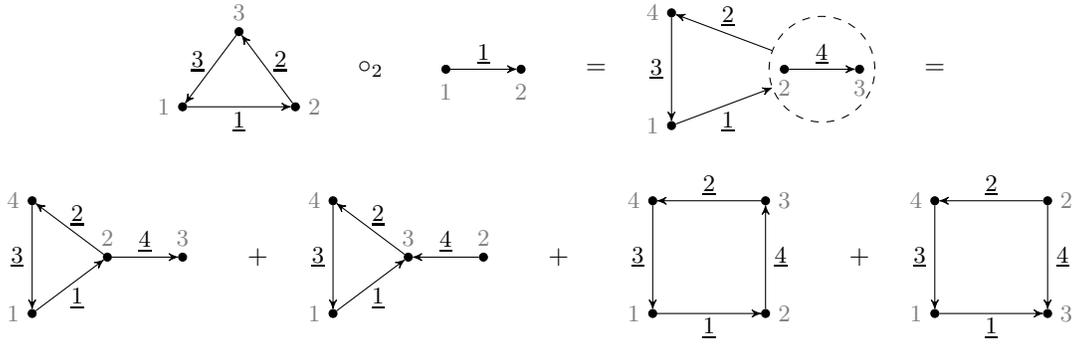
\begin{figure}[htp] 
\centering 
\begin{tikzpicture}[scale=0.5, >=stealth']
\tikzstyle{w}=[circle, draw, minimum size=3, inner sep=1]
\tikzstyle{b}=[circle, draw, fill, minimum size=3, inner sep=1]
\tikzstyle{circ}=[circle, draw, dashed, minimum size=40, inner sep=1]
\node [b] (b1) at (0,0) {};
\draw (-0.5, 0) node[anchor=center, color=gray] {{\small $1$}};
\node [b] (b2) at (3,0) {};
\draw (3.5, 0) node[anchor=center, color=gray] {{\small $2$}};
\node [b] (b3) at (1.5,2) {};
\draw (1.5, 2.5) node[anchor=center, color=gray] {{\small $3$}};
\draw [->] (b1) edge (b2) (b2) edge (b3) (b3) edge (b1); 
\draw (1.5, -0.4) node[anchor=center] {{\small $\und{1}$}};
\draw (2.6, 1.2) node[anchor=center] {{\small $\und{2}$}};
\draw (0.4, 1.2) node[anchor=center] {{\small $\und{3}$}};
\draw (5,1) node[anchor=center] {{$\circ_2$}};
\node [b] (bb1) at (7,1) {};
\draw (7, 0.4) node[anchor=center, color=gray] {{\small $1$}};
\node [b] (bb2) at (9,1) {};
\draw (9, 0.4) node[anchor=center, color=gray] {{\small $2$}};
\draw [->] (bb1) edge (bb2);
\draw (8, 1.4) node[anchor=center] {{\small $\und{1}$}};
\draw (11,1) node[anchor=center] {{$ = $}};
\begin{scope}[shift={(13,0)}]
\node [b] (b1) at (0,-0.5) {};
\draw (-0.5, -0.5) node[anchor=center, color=gray] {{\small $1$}};
\node [b] (b4) at (0,2.5) {};
\draw (-0.5, 2.5) node[anchor=center, color=gray] {{\small $4$}};
\node [circ] (cir) at (4,1) {~};
\draw [->] (b4) edge (b1) (cir) edge (b4) (b1) edge (cir);
\draw (1.5, -0.4) node[anchor=center] {{\small $\und{1}$}};
\draw (1.5, 2.4) node[anchor=center] {{\small $\und{2}$}};
\draw (-0.4, 1) node[anchor=center] {{\small $\und{3}$}};
%
\node [b] (b2) at (3,1) {};
\draw (3, 0.5) node[anchor=center, color=gray] {{\small $2$}};
\node [b] (b3) at (5,1) {};
\draw (5, 0.5) node[anchor=center, color=gray] {{\small $3$}};
\draw [->] (b2) edge (b3);
\draw (4, 1.4) node[anchor=center] {{\small $\und{4}$}};
\end{scope}
\draw (20,1) node[anchor=center] {{$ = $}};
\begin{scope}[shift={(-4,-5.5)}]
\node [b] (b1) at (0,0) {};
\draw (-0.5, 0) node[anchor=center, color=gray] {{\small $1$}};
\node [b] (b4) at (0,3) {};
\draw (-0.5, 3) node[anchor=center, color=gray] {{\small $4$}};
\node [b] (b2) at (2,1.5) {};
\draw (2, 2) node[anchor=center, color=gray] {{\small $2$}};
\node [b] (b3) at (4,1.5) {};
\draw (4, 2) node[anchor=center, color=gray] {{\small $3$}};
\draw [->] (b4) edge (b1) (b1) edge (b2) (b2) edge (b4) edge (b3);
\draw (1.2, 0.4) node[anchor=center] {{\small $\und{1}$}};
\draw (1.2, 2.6) node[anchor=center] {{\small $\und{2}$}};
\draw (-0.4, 1.5) node[anchor=center] {{\small $\und{3}$}};
\draw (3, 1.9) node[anchor=center] {{\small $\und{4}$}};
\draw (6,1.5) node[anchor=center] {{$ + $}};
\begin{scope}[shift={(8,0)}]
\node [b] (b1) at (0,0) {};
\draw (-0.5, 0) node[anchor=center, color=gray] {{\small $1$}};
\node [b] (b4) at (0,3) {};
\draw (-0.5, 3) node[anchor=center, color=gray] {{\small $4$}};
\node [b] (b3) at (2,1.5) {};
\draw (2, 2) node[anchor=center, color=gray] {{\small $3$}};
\node [b] (b2) at (4,1.5) {};
\draw (4, 2) node[anchor=center, color=gray] {{\small $2$}};
\draw [->] (b4) edge (b1) (b1) edge (b3) (b3) edge (b4) (b2) edge (b3);
\draw (1.2, 0.4) node[anchor=center] {{\small $\und{1}$}};
\draw (1.2, 2.6) node[anchor=center] {{\small $\und{2}$}};
\draw (-0.4, 1.5) node[anchor=center] {{\small $\und{3}$}};
\draw (3, 1.9) node[anchor=center] {{\small $\und{4}$}};
\draw (6,1.5) node[anchor=center] {{$ + $}};
\end{scope}
\begin{scope}[shift={(16.5,0)}]
\node [b] (b4) at (0,3) {};
\draw (-0.5, 3) node[anchor=center, color=gray] {{\small $4$}};
\node [b] (b1) at (0,0) {};
\draw (-0.5, 0) node[anchor=center, color=gray] {{\small $1$}};
\node [b] (b2) at (3,0) {};
\draw (3.5, 0) node[anchor=center, color=gray] {{\small $2$}};
\node [b] (b3) at (3,3) {};
\draw (3.5, 3) node[anchor=center, color=gray] {{\small $3$}};
\draw [->] (b4) edge (b1) (b1) edge (b2) (b2) edge (b3) (b3) edge (b4);
\draw (1.5, -0.4) node[anchor=center] {{\small $\und{1}$}};
\draw (1.5, 3.4) node[anchor=center] {{\small $\und{2}$}};
\draw (-0.4, 1.5) node[anchor=center] {{\small $\und{3}$}};
\draw (3.4, 1.5) node[anchor=center] {{\small $\und{4}$}};
\draw (5.5,1.5) node[anchor=center] {{$ + $}};
\end{scope}
\begin{scope}[shift={(24,0)}]
\node [b] (b4) at (0,3) {};
\draw (-0.5, 3) node[anchor=center, color=gray] {{\small $4$}};
\node [b] (b1) at (0,0) {};
\draw (-0.5, 0) node[anchor=center, color=gray] {{\small $1$}};
\node [b] (b3) at (3,0) {};
\draw (3.5, 0) node[anchor=center, color=gray] {{\small $3$}};
\node [b] (b2) at (3,3) {};
\draw (3.5, 3) node[anchor=center, color=gray] {{\small $2$}};
\draw [->] (b4) edge (b1) (b1) edge (b3) (b2) edge (b3) (b2) edge (b4);
\draw (1.5, -0.4) node[anchor=center] {{\small $\und{1}$}};
\draw (1.5, 3.4) node[anchor=center] {{\small $\und{2}$}};
\draw (-0.4, 1.5) node[anchor=center] {{\small $\und{3}$}};
\draw (3.4, 1.5) node[anchor=center] {{\small $\und{4}$}};
\end{scope}
\end{scope}
\end{tikzpicture}
\caption{An example of computing $\circ_2$} \label{fig:ins}
\end{figure}
%
%
\begin{figure}[htp] 
\centering 
\begin{tikzpicture}[scale=0.5, >=stealth']
\tikzstyle{w}=[circle, draw, minimum size=3, inner sep=1]
\tikzstyle{b}=[circle, draw, fill, minimum size=3, inner sep=1]
\tikzstyle{circ}=[circle, draw, dashed, minimum size=40, inner sep=1]
\node [b] (b1) at (0,0) {};
\draw (-0.5, 0) node[anchor=center, color=gray] {{\small $1$}};
\node [b] (b4) at (0,3) {};
\draw (-0.5, 3) node[anchor=center, color=gray] {{\small $4$}};
\node [b] (b2) at (2,1.5) {};
\draw (2, 2) node[anchor=center, color=gray] {{\small $2$}};
\node [b] (b3) at (4,1.5) {};
\draw (4, 2) node[anchor=center, color=gray] {{\small $3$}};
\draw [->] (b4) edge (b1) (b1) edge (b2) (b2) edge (b4) edge (b3);
\draw (1.2, 0.4) node[anchor=center] {{\small $\und{1}$}};
\draw (1.2, 2.6) node[anchor=center] {{\small $\und{2}$}};
\draw (-0.4, 1.5) node[anchor=center] {{\small $\und{3}$}};
\draw (3, 1.9) node[anchor=center] {{\small $\und{4}$}};
\draw (6,1.5) node[anchor=center] {{$ + $}};
\begin{scope}[shift={(8,0)}]
\node [b] (b1) at (0,0) {};
\draw (-0.5, 0) node[anchor=center, color=gray] {{\small $1$}};
\node [b] (b4) at (0,3) {};
\draw (-0.5, 3) node[anchor=center, color=gray] {{\small $4$}};
\node [b] (b3) at (2,1.5) {};
\draw (2, 2) node[anchor=center, color=gray] {{\small $3$}};
\node [b] (b2) at (4,1.5) {};
\draw (4, 2) node[anchor=center, color=gray] {{\small $2$}};
\draw [->] (b4) edge (b1) (b1) edge (b3) (b3) edge (b4) (b2) edge (b3);
\draw (1.2, 0.4) node[anchor=center] {{\small $\und{1}$}};
\draw (1.2, 2.6) node[anchor=center] {{\small $\und{2}$}};
\draw (-0.4, 1.5) node[anchor=center] {{\small $\und{3}$}};
\draw (3, 1.9) node[anchor=center] {{\small $\und{4}$}};
\draw (6,1.5) node[anchor=center] {{$ + $}};
\end{scope}
\begin{scope}[shift={(16.5,0)}]
\node [b] (b4) at (0,3) {};
\draw (-0.5, 3) node[anchor=center, color=gray] {{\small $4$}};
\node [b] (b1) at (0,0) {};
\draw (-0.5, 0) node[anchor=center, color=gray] {{\small $1$}};
\node [b] (b3) at (3,0) {};
\draw (3.5, 0) node[anchor=center, color=gray] {{\small $3$}};
\node [b] (b2) at (3,3) {};
\draw (3.5, 3) node[anchor=center, color=gray] {{\small $2$}};
\draw [->] (b4) edge (b1) (b1) edge (b3) (b2) edge (b3) (b2) edge (b4);
\draw (1.5, -0.4) node[anchor=center] {{\small $\und{1}$}};
\draw (1.5, 3.4) node[anchor=center] {{\small $\und{2}$}};
\draw (-0.4, 1.5) node[anchor=center] {{\small $\und{3}$}};
\draw (3.4, 1.5) node[anchor=center] {{\small $\und{4}$}};
\end{scope}
\end{tikzpicture}
\caption{The image of $\G \circ_2 \ti{\G}$ in $\dfGCo$} \label{fig:ins-dfGC}
\end{figure}

It is not hard to see that the map 
\begin{equation}
\label{bullet}
\G \bullet \ti{\G} : = \sum_{i=1}^n \G \circ_i \ti{\G} ~:~
\span_{\bbK}(\dgra_n^r) \otimes \span_{\bbK}(\dgra_m^q) ~\to~ \span_{\bbK}(\dgra_{m+n -1}^{r+q})
\end{equation}
descends to coinvariants and we get a degree $0$ bilinear operation on $\dfGCo$: 
$$
\bullet ~:~ \dfGCo \otimes  \dfGCo ~ \to  ~  \dfGCo\,. 
$$

We recall that  
\begin{prop}
\label{prop:dfGCo-Lie}
The bracket 
\begin{equation}
\label{dfGCo-Lie}
[ \, \G, \ti{\G} \,]  = \G \bullet \ti{\G}  - (-1)^{|\G| |\ti{\G}|} \ti{\G} \bullet \G
\end{equation}
equips $\dfGCo$ with the structure of a graded Lie algebra. 
\end{prop}
\begin{proof}
This statement follows directly from \cite[Proposition C.2]{DeligneTw}. 
More details about the relationship between $\dfGCo$ and a certain 
convolution Lie algebra \cite[Section 4]{notes}, \cite{MV-nado} are given 
in Section \ref{sec:dfGC-Conv} below.
\end{proof}

It is clear that 
\begin{equation}
\label{G-ed}
\begin{tikzpicture}[scale=0.5, >=stealth']
\tikzstyle{w}=[circle, draw, minimum size=3, inner sep=1]
\tikzstyle{b}=[circle, draw, fill, minimum size=3, inner sep=1]
\draw (-2,0) node[anchor=center] {{$\G_{\ed} ~ : =$}};
\node [b] (b1) at (0,0) {};
\draw (0,0.5) node[anchor=center, color=gray] {{\small $1$}};
\node [b] (b2) at (2,0) {};
\draw (2, 0.5) node[anchor=center, color=gray] {{\small $2$}};
\draw [->] (b1) edge (b2);
\end{tikzpicture}
\end{equation}
is a non-zero vector in $\dfGCo$ of degree $1$. 
Moreover, a direct computation shows that $\G_{\ed}$
satisfies the MC equation
\begin{equation}
\label{G-ed-MC}
[\G_{\ed}, \G_{\ed}] = 0. 
\end{equation}
So we use $\G_{\ed}$ to define the following differential on $\dfGCo$: 
\begin{equation}
\label{diff-dfGCo}
\pa =  [\G_{\ed},  ~ ].
\end{equation}

To define the (full) directed graph complex $\dfGC$, we denote by
$\dfGC(n)$ the subspace of $\dfGCo$ which is 
spanned by isomorphism classes of even graphs with exactly $n$ vertices. 
It is obvious that $\pa \big( \dfGC(n) \big) \subset \dfGC(n+1)$ and 
it allows us to give the following definition:
%
%
\begin{defi}
\label{dfn:dfGC}
The \emph{full directed graph complex} $\dfGC$ is the following completion of $\dfGCo$
\begin{equation}
\label{dfGC}
\dfGC : = \prod_{n \ge 1}  \dfGC(n). 
\end{equation}
\end{defi}

\begin{example}
\label{exam:G-bul-G-lp}
Let $\G_{\bul}$ be the graph in $\dgra^0_1$ which consists of a
single vertex and  $\G_{\lp}$ be the graph in $\dgra^1_1$ which consists of 
a single loop. For these graphs we have 
$$
\begin{tikzpicture}[scale=0.5, >=stealth']
\tikzstyle{w}=[circle, draw, minimum size=3, inner sep=1]
\tikzstyle{b}=[circle, draw, fill, minimum size=3, inner sep=1]
\draw (-7,0) node[anchor=center] {{$\pa \G_{\bul} ~=~ \G_{\ed}\,,$}};
\draw (-0.5,0) node[anchor=center] {{$ \pa \G_{\lp} ~ = $}};
\node [b] (b1) at (2,0) {};
\draw (1.5,0) node[anchor=center, color =gray] {{\small $1$}};
\node [b] (b2) at (4, 0) {};
\draw (4.5, 0) node[anchor=center, color=gray] {{\small $2$}};
\draw [<-] (b1) ..controls (2.5,1) and (3.5,1) .. (b2);
\draw [<-] (b2) ..controls (3.5, -1) and (2.5, -1) .. (b1);
\draw (3,1.2) node[anchor=center] {{\small $\und{1}$}};
\draw (3,-1.2) node[anchor=center] {{\small $\und{2}$}};
\draw (6,0) node[anchor=center] {{$ = ~ 0.$}};
\end{tikzpicture}
$$
Thus $\G_{\lp}$ represents a degree $-1$ (non-trivial) cocycle in $\dfGC$\,. 

Another example of the computation of the differential in $\dfGC$ is given in figure \ref{fig:G123}. 
In this example, only the uni-bivalent graphs coming from the insertion of $\G_{\ed}$ 
into vertex $2$ survive. All the other graphs cancel each other.
\begin{figure}[htp]
\centering 
\begin{tikzpicture}[scale=0.5, >=stealth']
\tikzstyle{w}=[circle, draw, minimum size=4, inner sep=1]
\tikzstyle{b}=[circle, draw, fill, minimum size=4, inner sep=1]
\draw (-1,0) node[anchor=center] {{$\pa$}};
\node [b] (b1) at (0,0) {};
\draw (0,-0.6) node[anchor=center, color =gray] {{\small $1$}};
\node [b] (b2) at (2,0) {};
\draw (2,-0.6) node[anchor=center, color =gray] {{\small $2$}};
\node [b] (b3) at (4,0) {};
\draw (4,-0.6) node[anchor=center, color =gray] {{\small $3$}};
\draw [->] (b1) edge (b2);
\draw (1,0.4) node[anchor=center] {{\small $\und{1}$}};
\draw [->] (b2) edge (b3);
\draw (3,0.4) node[anchor=center] {{\small $\und{2}$}};
\draw (5.3,0) node[anchor=center] {{$= ~ $}};
\begin{scope}[shift={(6.5,0)}]
\node [b] (b1) at (0,0) {};
\draw (0,-0.6) node[anchor=center, color =gray] {{\small $1$}};
\node [b] (b2) at (2,0) {};
\draw (2,-0.6) node[anchor=center, color =gray] {{\small $2$}};
\node [b] (b3) at (4,0) {};
\draw (4,-0.6) node[anchor=center, color =gray] {{\small $3$}};
\node [b] (b4) at (6,0) {};
\draw (6,-0.6) node[anchor=center, color =gray] {{\small $4$}};
\draw [->] (b1) edge (b2);
\draw (1,0.4) node[anchor=center] {{\small $\und{2}$}};
\draw [->] (b2) edge (b3);
\draw (3,0.4) node[anchor=center] {{\small $\und{3}$}};
\draw [->] (b3) edge (b4);
\draw (5,0.4) node[anchor=center] {{\small $\und{1}$}};
\end{scope}
\begin{scope}[shift={(15,0)}]
\draw (-1,0) node[anchor=center] {{$+ ~ $}};
\node [b] (b1) at (0,0) {};
\draw (0,-0.6) node[anchor=center, color =gray] {{\small $1$}};
\node [b] (b2) at (2,0) {};
\draw (2,-0.6) node[anchor=center, color =gray] {{\small $2$}};
\node [b] (b3) at (4,0) {};
\draw (4,-0.6) node[anchor=center, color =gray] {{\small $3$}};
\node [b] (b4) at (6,0) {};
\draw (6,-0.6) node[anchor=center, color =gray] {{\small $4$}};
\draw [->] (b1) edge (b2);
\draw (1,0.4) node[anchor=center] {{\small $\und{2}$}};
\draw [->] (b2) edge (b3);
\draw (3,0.4) node[anchor=center] {{\small $\und{3}$}};
\draw [<-] (b3) edge (b4);
\draw (5,0.4) node[anchor=center] {{\small $\und{1}$}};
\end{scope}
\begin{scope}[shift={(-3,-4)}]
\draw (-1.3,0) node[anchor=center] {{$+ ~ $}};
\node [b] (b1) at (0,0) {};
\draw (0,-0.6) node[anchor=center, color =gray] {{\small $1$}};
\node [b] (b2) at (2,0) {};
\draw (2,-0.6) node[anchor=center, color =gray] {{\small $2$}};
\node [b] (b3) at (4,0) {};
\draw (4,-0.6) node[anchor=center, color =gray] {{\small $3$}};
\node [b] (b4) at (6,0) {};
\draw (6,-0.6) node[anchor=center, color =gray] {{\small $4$}};
\draw [->] (b1) edge (b2);
\draw (1,0.4) node[anchor=center] {{\small $\und{1}$}};
\draw [->] (b2) edge (b3);
\draw (3,0.4) node[anchor=center] {{\small $\und{2}$}};
\draw [->] (b3) edge (b4);
\draw (5,0.4) node[anchor=center] {{\small $\und{3}$}};
\end{scope}
\begin{scope}[shift={(6,-4)}]
\draw (-1.3,0) node[anchor=center] {{$+ ~$}};
\node [b] (b1) at (0,0) {};
\draw (0,-0.6) node[anchor=center, color =gray] {{\small $1$}};
\node [b] (b2) at (2,0) {};
\draw (2,-0.6) node[anchor=center, color =gray] {{\small $2$}};
\node [b] (b3) at (4,0) {};
\draw (4,-0.6) node[anchor=center, color =gray] {{\small $3$}};
\node [b] (b4) at (6,0) {};
\draw (6,-0.6) node[anchor=center, color =gray] {{\small $4$}};
\draw [<-] (b1) edge (b2);
\draw (1,0.4) node[anchor=center] {{\small $\und{1}$}};
\draw [->] (b2) edge (b3);
\draw (3,0.4) node[anchor=center] {{\small $\und{2}$}};
\draw [->] (b3) edge (b4);
\draw (5,0.4) node[anchor=center] {{\small $\und{3}$}};
\end{scope}
\begin{scope}[shift={(15,-4)}]
\draw (-1.3,0) node[anchor=center] {{$+ ~ $}};
\node [b] (b1) at (0,0) {};
\draw (0,-0.6) node[anchor=center, color =gray] {{\small $1$}};
\node [b] (b2) at (2,0) {};
\draw (2,-0.6) node[anchor=center, color =gray] {{\small $2$}};
\node [b] (b3) at (4,0) {};
\draw (4,-0.6) node[anchor=center, color =gray] {{\small $3$}};
\node [b] (b4) at (2,2) {};
\draw (2,2.5) node[anchor=center, color =gray] {{\small $4$}};
\draw [->] (b1) edge (b2);
\draw (1,-0.5) node[anchor=center] {{\small $\und{2}$}};
\draw [->] (b2) edge (b3);
\draw (3,-0.5) node[anchor=center] {{\small $\und{3}$}};
\draw [->] (b2) edge (b4);
\draw (1.7,1) node[anchor=center] {{\small $\und{1}$}};
\end{scope}
\begin{scope}[shift={(22,-4)}]
\draw (-1.3,0) node[anchor=center] {{$+ ~ $}};
\node [b] (b1) at (0,0) {};
\draw (0,-0.6) node[anchor=center, color =gray] {{\small $1$}};
\node [b] (b2) at (2,0) {};
\draw (2,-0.6) node[anchor=center, color =gray] {{\small $2$}};
\node [b] (b3) at (4,0) {};
\draw (4,-0.6) node[anchor=center, color =gray] {{\small $3$}};
\node [b] (b4) at (2,2) {};
\draw (2,2.5) node[anchor=center, color =gray] {{\small $4$}};
\draw [->] (b1) edge (b2);
\draw (1,-0.5) node[anchor=center] {{\small $\und{2}$}};
\draw [->] (b2) edge (b3);
\draw (3,-0.5) node[anchor=center] {{\small $\und{3}$}};
\draw [<-] (b2) edge (b4);
\draw (1.7,1) node[anchor=center] {{\small $\und{1}$}};
\end{scope}
\begin{scope}[shift={(-2,-8)}]
\draw (-1.2,0) node[anchor=center] {{$- ~ $}};
\node [b] (b1) at (0,0) {};
\draw (0,-0.6) node[anchor=center, color =gray] {{\small $1$}};
\node [b] (b2) at (2,0) {};
\draw (2,-0.6) node[anchor=center, color =gray] {{\small $2$}};
\node [b] (b3) at (4,0) {};
\draw (4,-0.6) node[anchor=center, color =gray] {{\small $3$}};
\node [b] (b4) at (6,0) {};
\draw (6,-0.6) node[anchor=center, color =gray] {{\small $4$}};
\draw [->] (b1) edge (b2);
\draw (1,0.4) node[anchor=center] {{\small $\und{1}$}};
\draw [->] (b2) edge (b3);
\draw (3,0.4) node[anchor=center] {{\small $\und{2}$}};
\draw [->] (b3) edge (b4);
\draw (5,0.4) node[anchor=center] {{\small $\und{3}$}};
\end{scope}
\begin{scope}[shift={(7,-8)}]
\draw (-1.2,0) node[anchor=center] {{$- ~$}};
\node [b] (b1) at (0,0) {};
\draw (0,-0.6) node[anchor=center, color =gray] {{\small $1$}};
\node [b] (b2) at (2,0) {};
\draw (2,-0.6) node[anchor=center, color =gray] {{\small $2$}};
\node [b] (b3) at (4,0) {};
\draw (4,-0.6) node[anchor=center, color =gray] {{\small $3$}};
\node [b] (b4) at (6,0) {};
\draw (6,-0.6) node[anchor=center, color =gray] {{\small $4$}};
\draw [->] (b1) edge (b2);
\draw (1,0.4) node[anchor=center] {{\small $\und{1}$}};
\draw [->] (b2) edge (b3);
\draw (3,0.4) node[anchor=center] {{\small $\und{2}$}};
\draw [<-] (b3) edge (b4);
\draw (5,0.4) node[anchor=center] {{\small $\und{3}$}};
\end{scope}
\begin{scope}[shift={(16,-8)}]
\draw (-1.2,0) node[anchor=center] {{$- ~ $}};
\node [b] (b1) at (0,0) {};
\draw (0,-0.6) node[anchor=center, color =gray] {{\small $1$}};
\node [b] (b2) at (2,0) {};
\draw (2,-0.6) node[anchor=center, color =gray] {{\small $2$}};
\node [b] (b4) at (4,0) {};
\draw (4,-0.6) node[anchor=center, color =gray] {{\small $4$}};
\node [b] (b3) at (2,2) {};
\draw (2,2.6) node[anchor=center, color =gray] {{\small $3$}};
\draw [->] (b1) edge (b2);
\draw (1, -0.5) node[anchor=center] {{\small $\und{1}$}};
\draw [->] (b2) edge (b4);
\draw (3, - 0.5) node[anchor=center] {{\small $\und{2}$}};
\draw [->] (b2) edge (b3);
\draw (1.7,1) node[anchor=center] {{\small $\und{3}$}};
\end{scope}
\begin{scope}[shift={(23,-8)}]
\draw (-1.2,0) node[anchor=center] {{$- ~$}};
\node [b] (b1) at (0,0) {};
\draw (0,-0.6) node[anchor=center, color =gray] {{\small $1$}};
\node [b] (b2) at (2,0) {};
\draw (2,-0.6) node[anchor=center, color =gray] {{\small $2$}};
\node [b] (b4) at (4,0) {};
\draw (4,-0.6) node[anchor=center, color =gray] {{\small $4$}};
\node [b] (b3) at (2,2) {};
\draw (2,2.6) node[anchor=center, color =gray] {{\small $3$}};
\draw [->] (b1) edge (b2);
\draw (1, -0.5) node[anchor=center] {{\small $\und{1}$}};
\draw [->] (b2) edge (b4);
\draw (3, - 0.5) node[anchor=center] {{\small $\und{2}$}};
\draw [<-] (b2) edge (b3);
\draw (1.7,1) node[anchor=center] {{\small $\und{3}$}};
\end{scope}
\begin{scope}[shift={(-4,-12)}]
\draw (-1.2,0) node[anchor=center] {{$- ~ $}};
\node [b] (b1) at (0,0) {};
\draw (0,-0.6) node[anchor=center, color =gray] {{\small $1$}};
\node [b] (b2) at (2,0) {};
\draw (2,-0.6) node[anchor=center, color =gray] {{\small $2$}};
\node [b] (b3) at (4,0) {};
\draw (4,-0.6) node[anchor=center, color =gray] {{\small $3$}};
\node [b] (b4) at (6,0) {};
\draw (6,-0.6) node[anchor=center, color =gray] {{\small $4$}};
\draw [->] (b1) edge (b2);
\draw (1,0.4) node[anchor=center] {{\small $\und{3}$}};
\draw [->] (b2) edge (b3);
\draw (3,0.4) node[anchor=center] {{\small $\und{1}$}};
\draw [->] (b3) edge (b4);
\draw (5,0.4) node[anchor=center] {{\small $\und{2}$}};
\end{scope}
\begin{scope}[shift={(5,-12)}]
\draw (-1.2,0) node[anchor=center] {{$- ~ $}};
\node [b] (b1) at (0,0) {};
\draw (0,-0.6) node[anchor=center, color =gray] {{\small $1$}};
\node [b] (b2) at (2,0) {};
\draw (2,-0.6) node[anchor=center, color =gray] {{\small $2$}};
\node [b] (b3) at (4,0) {};
\draw (4,-0.6) node[anchor=center, color =gray] {{\small $3$}};
\node [b] (b4) at (6,0) {};
\draw (6,-0.6) node[anchor=center, color =gray] {{\small $4$}};
\draw [<-] (b1) edge (b2);
\draw (1,0.4) node[anchor=center] {{\small $\und{3}$}};
\draw [->] (b2) edge (b3);
\draw (3,0.4) node[anchor=center] {{\small $\und{1}$}};
\draw [->] (b3) edge (b4);
\draw (5,0.4) node[anchor=center] {{\small $\und{2}$}};
\end{scope}
\begin{scope}[shift={(14,-12)}]
\draw (-1.2,0) node[anchor=center] {{$- ~ $}};
\node [b] (b1) at (0,0) {};
\draw (0,-0.6) node[anchor=center, color =gray] {{\small $1$}};
\node [b] (b2) at (2,0) {};
\draw (2,-0.6) node[anchor=center, color =gray] {{\small $2$}};
\node [b] (b3) at (4,0) {};
\draw (4,-0.6) node[anchor=center, color =gray] {{\small $3$}};
\node [b] (b4) at (6,0) {};
\draw (6,-0.6) node[anchor=center, color =gray] {{\small $4$}};
\draw [->] (b1) edge (b2);
\draw (1,0.4) node[anchor=center] {{\small $\und{1}$}};
\draw [->] (b2) edge (b3);
\draw (3,0.4) node[anchor=center] {{\small $\und{3}$}};
\draw [->] (b3) edge (b4);
\draw (5,0.4) node[anchor=center] {{\small $\und{2}$}};
\end{scope}
\begin{scope}[shift={(23,-12)}]
\draw (-1.2,0) node[anchor=center] {{$- ~ $}};
\node [b] (b1) at (0,0) {};
\draw (0,-0.6) node[anchor=center, color =gray] {{\small $1$}};
\node [b] (b2) at (2,0) {};
\draw (2,-0.6) node[anchor=center, color =gray] {{\small $2$}};
\node [b] (b3) at (4,0) {};
\draw (4,-0.6) node[anchor=center, color =gray] {{\small $3$}};
\node [b] (b4) at (6,0) {};
\draw (6,-0.6) node[anchor=center, color =gray] {{\small $4$}};
\draw [->] (b1) edge (b2);
\draw (1,0.4) node[anchor=center] {{\small $\und{1}$}};
\draw [<-] (b2) edge (b3);
\draw (3,0.4) node[anchor=center] {{\small $\und{3}$}};
\draw [->] (b3) edge (b4);
\draw (5,0.4) node[anchor=center] {{\small $\und{2}$}};
\end{scope}
\begin{scope}[shift={(4,-16)}]
\draw (-1.5,0) node[anchor=center] {{$= ~$}};
\node [b] (b1) at (0,0) {};
\draw (0,-0.6) node[anchor=center, color =gray] {{\small $1$}};
\node [b] (b2) at (2,0) {};
\draw (2,-0.6) node[anchor=center, color =gray] {{\small $2$}};
\node [b] (b3) at (4,0) {};
\draw (4,-0.6) node[anchor=center, color =gray] {{\small $3$}};
\node [b] (b4) at (6,0) {};
\draw (6,-0.6) node[anchor=center, color =gray] {{\small $4$}};
\draw [->] (b1) edge (b2);
\draw (1,0.4) node[anchor=center] {{\small $\und{1}$}};
\draw [->] (b2) edge (b3);
\draw (3,0.4) node[anchor=center] {{\small $\und{2}$}};
\draw [->] (b3) edge (b4);
\draw (5,0.4) node[anchor=center] {{\small $\und{3}$}};
\end{scope}
\begin{scope}[shift={(13, -16)}]
\draw (-1.2,0) node[anchor=center] {{$+ ~$}};
\node [b] (b1) at (0,0) {};
\draw (0,-0.6) node[anchor=center, color =gray] {{\small $1$}};
\node [b] (b2) at (2,0) {};
\draw (2,-0.6) node[anchor=center, color =gray] {{\small $2$}};
\node [b] (b3) at (4,0) {};
\draw (4,-0.6) node[anchor=center, color =gray] {{\small $3$}};
\node [b] (b4) at (6,0) {};
\draw (6,-0.6) node[anchor=center, color =gray] {{\small $4$}};
\draw [->] (b1) edge (b2);
\draw (1,0.4) node[anchor=center] {{\small $\und{1}$}};
\draw [<-] (b2) edge (b3);
\draw (3,0.4) node[anchor=center] {{\small $\und{2}$}};
\draw [->] (b3) edge (b4);
\draw (5,0.4) node[anchor=center] {{\small $\und{3}$}};
\end{scope}
\end{tikzpicture}
\caption{An example of computing $\pa(\G)$} \label{fig:G123}
\end{figure} 
\end{example}

Let us observe that 
\begin{prop}
\label{prop:dfGC} 
If $\G$ is a connected (even) graph in $\dgra^r_n$ with $r \ge 1$, then   
\begin{equation}
\label{pa-dfGC-simpler}
\pa \G =  - (-1)^{|\G|}  \sum_{i=1}^n \ga_i \,,
\end{equation}
where $\ga_i$ is obtained from the linear combination $\G \circ_i \G_{\ed}$
by discarding all graphs in which either vertex $i$ or vertex 
$i+1$ has valency $1$.
\end{prop}
\begin{proof} 
Since each vertex of $\G$ is adjacent to at least one edge, in $\dfGCo$, we have 
$$
\G_{\ed} \circ_1 \G + \G_{\ed} \circ_2 \G   =  (-1)^{|\G|} \sum_{i=1}^n \ga'_i 
$$
where $\ga'_i$ is obtained from the linear combination 
$\G \circ_i \G_{\ed}$ by keeping only the graphs in which either
vertex $i$ or vertex $i+1$ has valency $1$.

Thus \eqref{pa-dfGC-simpler} follows. 
\end{proof}

%
%
\subsection{$\fGC$, $\GC$, and other variants of $\dfGC$. Examples of cocycles}
\label{sec:fGC-GC}

For a pair of integers $n \ge 1, ~ r \ge 0$, we introduce the auxiliary set 
$\gra_n^r$.  An element $\G$ of $\gra_n^r$ is an undirected graph with the set of vertices 
$V(\G) = \{\gray{1,2, \dots, n}\} $ and the set of edges $E(\G) = \{\und{1}, \und{2}, \dots, \und{r}\}$.

The set $\gra_{n}^r$ is equipped with the obvious action of the group $\bbS_n \times \bbS_r$.
So, by analogy with \eqref{dfGC-oplus}, we set 
\begin{equation}
\label{fGC-oplus}
\fGCo ~ : =~ \bigoplus_{n \ge 1,~r \ge 0} ~  \big( \bs^{2 n - 2 -r }  
\span_{\bbK}(\gra_{n}^r) \otimes  \sgn_{r} \big)_{\bbS_n \times \bbS_r }\,.
\end{equation}

Just as for $\dfGCo$, a graph $\G \in \gra_n^r$ 
gives us the zero vector in $\fGCo$ if and only if 
$\G$ has an automorphism which induces an odd permutation in $\bbS_r$.
So, by analogy with directed graphs, we say that a graph $\G \in \gra_n^r$ is {\it odd} 
if it has  an automorphism which induces an odd permutation in $\bbS_r$. 
Otherwise, we say that $\G$ is {\it even}. 
For example, the tetrahedron shown in figure \ref{fig:tetra} is even 
and the triangle shown in figure \ref{fig:triangle} is odd. 

It is easy to see that $\fGCo$ is the span of isomorphism classes of 
all even (undirected) graphs with the rule that an even graph $\G \in \gra_n^r$ 
carries the degree $2n - 2 - r$. 
\begin{figure}[htp]
\begin{minipage}[t]{0.45\linewidth}
\centering 
\begin{tikzpicture}[scale=0.5, >=stealth']
\tikzstyle{w}=[circle, draw, minimum size=4, inner sep=1]
\tikzstyle{b}=[circle, draw, fill, minimum size=4, inner sep=1]
\node [b] (b1) at (0,0) {};
\draw (0,-0.6) node[anchor=center, color=gray] {{\small $1$}};
\node [b] (b2) at (0,2) {};
\draw (0,2.6) node[anchor=center, color=gray] {{\small $2$}};
\node [b] (b3) at (2,2) {};
\draw (2,2.6) node[anchor=center, color=gray] {{\small $3$}};
\node [b] (b4) at (2,0) {};
\draw (2,-0.6) node[anchor=center, color=gray] {{\small $4$}};
\draw  (b1) edge (b2) edge (b3) edge (b4) (b2) edge 
(b3) edge (b4) (b3) edge (b4);
\end{tikzpicture}
~\\[0.3cm]
\caption{This graph is even} \label{fig:tetra}
\end{minipage}
\begin{minipage}[t]{0.45\linewidth}
\centering 
\begin{tikzpicture}[scale=0.5, >=stealth']
\tikzstyle{w}=[circle, draw, minimum size=4, inner sep=1]
\tikzstyle{b}=[circle, draw, fill, minimum size=4, inner sep=1]
\node [b] (b1) at (0,0) {};
\draw (0,-0.6) node[anchor=center,  color=gray] {{\small $1$}};
\node [b] (b2) at (2,0) {};
\draw (2,-0.6) node[anchor=center,  color=gray] {{\small $2$}};
\node [b] (b3) at (0,2) {};
\draw (0,2.5) node[anchor=center,  color=gray] {{\small $3$}};
\draw  (b1) edge (b2) edge (b3) (b2) edge (b3);
\end{tikzpicture}
~\\[0.3cm]
\caption{This graph is odd} \label{fig:triangle}
\end{minipage}
\end{figure} 

We should mention that in figures \ref{fig:tetra} and \ref{fig:triangle} as well 
as in some further figures, we often omit labels for edges. The reader should keep in 
mind that an even graph without labels on edges define 
a vector in $\fGCo$ only up to a sign factor.   

Using the analogous map $\circ_i  :  \span_{\bbK}(\gra_{n}^r) \otimes \span_{\bbK}(\gra_{m}^q)  \to  \fGCo$,
we define the degree $0$ binary operation
$\bullet ~:~ \fGCo \otimes  \fGCo ~ \to  ~  \fGCo$
by setting
\begin{equation}
\label{bullet-fGCo}
\G \bullet \ti{\G} : = \sum_{i=1}^n \G \circ_i \ti{\G}\,.
\end{equation}
Moreover, we claim that the same formula \eqref{dfGCo-Lie} 
defines a Lie bracket on the graded vector space $\fGCo$. 

A direct computation shows that the (non-zero) degree $1$ 
vector\footnote{It is convenient to have the factor $1/2$ in the definition of $\G^{un}_{\ed}$.} 
$$
\begin{tikzpicture}[scale=0.5, >=stealth']
\tikzstyle{w}=[circle, draw, minimum size=3, inner sep=1]
\tikzstyle{b}=[circle, draw, fill, minimum size=3, inner sep=1]
\draw (-2.6,0.1) node[anchor=center] {{$\displaystyle \G^{un}_{\ed} ~ : = ~\frac{1}{2}$}};
\node [b] (b1) at (0,0) {};
\draw (0,0.5) node[anchor=center, color=gray] {{\small $1$}};
\node [b] (b2) at (2,0) {};
\draw (2, 0.5) node[anchor=center, color=gray] {{\small $2$}};
\draw (b1) edge (b2);
\end{tikzpicture}
$$
satisfies the MC equation 
$$
[\G^{un}_{\ed} \, , \, \G^{un}_{\ed} ] = 0. 
$$
So we define the differential on $\fGCo$ by the formula: 
\begin{equation}
\label{diff-fGCo}
\pa =  [\G^{un}_{\ed},  ~ ].
\end{equation}

Just as for $\dfGCo$, we denote by $\fGC(n)$ the subspace of $\fGCo$ 
which is spanned by isomorphism classes of even graphs with exactly $n$ vertices. 
We also observe that $\pa \big( \fGC(n) \big) \subset \fGC(n+1)$.
So we define the {\it full graph complex} $\fGC$ as the following 
completion of $\fGCo$
\begin{equation}
\label{fGC}
\fGC ~ : = ~ \prod_{n \ge 1}  \fGC(n). 
\end{equation}

\subsubsection{$\GC$ and Kontsevich's graph complex $\GC_{1ve}$}
\label{sec:GC-1ve}

Recall that a vertex $v$ of a graph $\G$ is called a 
{\it cut vertex} if $\G$ becomes disconnected  upon 
deleting $v$. A graph $\G$ without cut vertices is called 
{\it $1$-vertex irreducible}. For example, the tetrahedron 
shown in figure \ref{fig:tetra} is $1$-vertex irreducible while the graph 
shown in figure \ref{fig:kissing-tetra} is not. 
%
%
\begin{figure}[htp]
\centering 
\begin{tikzpicture}[scale=0.5, >=stealth']
\tikzstyle{w}=[circle, draw, minimum size=4, inner sep=1]
\tikzstyle{b}=[circle, draw, fill, minimum size=4, inner sep=1]
\node [b] (b1) at (0,0) {};
\draw (0,0.6) node[anchor=center, color=gray] {{\small $1$}};
\node [b] (b2) at (1,1) {};
\draw (1,1.6) node[anchor=center, color=gray] {{\small $2$}};
\node [b] (b3) at (1,-1) {};
\draw (1,-1.6) node[anchor=center, color=gray] {{\small $3$}};
\node [b] (b4) at (2,0) {};
\draw (2,0.6) node[anchor=center, color=gray] {{\small $4$}};
\node [b] (b5) at (3,1) {};
\draw (3,1.6) node[anchor=center, color=gray] {{\small $5$}};
\node [b] (b6) at (3,-1) {};
\draw (3,-1.6) node[anchor=center, color=gray] {{\small $6$}};
\node [b] (b7) at (4,0) {};
\draw (4.4,0) node[anchor=center, color=gray] {{\small $7$}};
\draw  (b1) edge (b2) edge (b3) edge (b4) (b2) edge (b3) edge (b4) (b3) edge (b4);
\draw (b4) edge (b5) edge (b6) edge (b7) (b5) edge (b6) edge (b7) (b6) edge (b7);
\end{tikzpicture}
~\\[0.3cm]
\caption{Vertex $4$ is a cut vertex of this graph} \label{fig:kissing-tetra}
\end{figure}
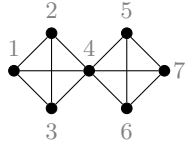

Let us denote by $\GC(n)$ the subspace of $\fGC(n)$ which is 
spanned by isomorphism classes of even graphs $\G$ satisfying these 
technical conditions: 
\begin{itemize}

\item $\G$ is connected;

\item each vertex of $\G$ has valency $\ge 3$. 

\end{itemize}

We denote by $\GC_{1ve}(n)$ the subspace of $\GC(n)$ 
which is spanned by isomorphism classes of even $1$-vertex irreducible 
graphs $\G$. Finally, we set 
\begin{equation}
\label{GC}
\GCo : = \bigoplus_{n \ge 1} \GC(n), \qquad 
\GC : = \prod_{n \ge 1} \GC(n),
\end{equation}
\begin{equation}
\label{GC-1ve}
\GC_{1ve}^{\oplus} : = \bigoplus_{n \ge 1} \GC_{1ve}(n), \qquad 
\GC_{1ve} : = \prod_{n \ge 1} \GC_{1ve}(n).
\end{equation} 

We claim that 
\begin{prop}
\label{prop:GC-GC-1ve}
\begin{itemize}

\item If $\G$ is a connected (even) graph in $\gra^r_n$ with $r \ge 1$, then   
\begin{equation}
\label{pa-fGC-simpler}
\pa \G =  - (-1)^{|\G|}  \sum_{i=1}^n \ga_i \,,
\end{equation}
where $\ga_i$ is obtained from the linear combination $\G \circ_i \G^{un}_{\ed}$
by discarding all graphs in which either vertex $i$ or vertex 
$i+1$ has valency $1$.

\item If $\G$ is a connected (even) graph with all its vertices having 
valencies $\ge 3$, then 
\begin{equation}
\label{pa-GC}
\pa \G  =  - (-1)^{|\G|} \sum_{i=1}^n \,  \ga^{\ge 3}_i\,,
\end{equation}
where $\ga^{\ge 3}_i$ is obtained from the linear combination  $\G \circ_i \G_{\ed}$
by discarding all graphs which have a univalent or a bivalent vertex. 

\item The subspaces $\GCo$ and $\GCo_{1ve}$ 
(resp. $\GC$ and $\GC_{1ve}$) are dg Lie subalgebras of $\fGCo$
(resp. $\fGC$). In particular,  $\GCo_{1ve}$ (resp. $\GC_{1ve}$) is 
a subcomplex of $\fGCo$ (resp. $\fGC$). 

\end{itemize}
\end{prop}
\begin{proof}
The proof of \eqref{pa-fGC-simpler} is almost identical to the proof of \eqref{pa-dfGC-simpler}. So we skip it. 

It is easy to see that the subspaces 
$$
\GCo \subset \fGCo, \qquad \GC \subset \fGC
$$
are closed with respect to the binary operation \eqref{bullet-fGCo}. 
Hence $\GCo$ (resp. $\GC$) are Lie subalgebras of $\fGCo$ (resp. $\fGC$). 

Let $\G$ be a connected (even) graph in $\gra^r_n$ whose all vertices have 
valencies $\ge 3$ (hence $r > 1$). It is clear that every graph in $\pa \G$ is connected. 

Let $1 \le q \le r$ and the edge labeled by $\und{q}$ in $\G$ connects two distinct vertices 
$i$ and $j$ with $i < j$. 
We denote by $\G^+_q$ and $\G^-_q$  the elements of $\gra^{r+1}_{n+1}$ which are 
obtained from $\G$ via replacing the edge $\und{q}$ by two edges which connect the additional 
vertex $n+1$ to $i$ and $j$ respectively. For $\G^+_q$, the edge connecting $n+1$ to $j$ is labeled 
by $\und{q}$ and the edge connecting $n+1$ to $i$ is labeled by $\und{r+1}$. 
For $\G^-_q$, the edge connecting $n+1$ to $j$ is labeled 
by $\und{r+1}$ and the edge connecting $n+1$ to $i$ is labeled by $\und{q}$. 
(See figure \ref{fig:G-q-pm} for the illustration.)
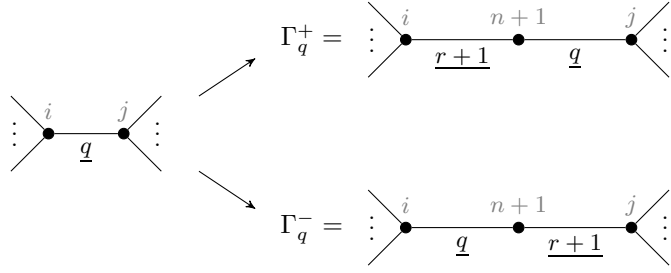
\begin{figure}[htp]
\centering 
\begin{tikzpicture}[scale=0.5, >=stealth']
\tikzstyle{w}=[circle, draw, minimum size=4, inner sep=1]
\tikzstyle{b}=[circle, draw, fill, minimum size=4, inner sep=1]
\node [b] (bi) at (0,0) {};
\draw (0,0.6) node[anchor=center,  color=gray] {{\small $i$}};
\draw (1,-0.5) node[anchor=center] {{\small $\und{q}$}};
\node [b] (bj) at (2,0) {};
\draw (2,0.6) node[anchor=center,  color=gray] {{\small $j$}};
\draw (bi) edge (bj);
\draw (bi) edge (-1,1) edge (-1,-1);
\draw (-0.9,0.2) node[anchor=center] {{\small $\vdots$}};
\draw (bj) edge (3,1) edge (3,-1);
\draw (2.9,0.2) node[anchor=center] {{\small $\vdots$}};
\draw [->] (4, 1) -- (5.5,2);
\draw [->] (4, -1) -- (5.5, -2);
\begin{scope}[shift={(9.5,2.5)}]
\draw (-2.5,0) node[anchor=center] {{$\G^+_{q} = $}};
\node [b] (bi) at (0,0) {};
\draw (0,0.6) node[anchor=center,  color=gray] {{\small $i$}};
\draw (1.5, -0.5) node[anchor=center] {{\small $\und{r+1}$}};
\node [b] (bn1) at (3,0) {};
\draw (3,0.6) node[anchor=center,  color=gray] {{\small $n+1$}};
\draw (4.5, -0.5) node[anchor=center] {{\small $\und{q}$}};
\node [b] (bj) at (6,0) {};
\draw (6,0.6) node[anchor=center,  color=gray] {{\small $j$}};
\draw (bn1) edge (bi) edge (bj);
\draw (bi) edge (-1,1) edge (-1,-1);
\draw (-0.9,0.2) node[anchor=center] {{\small $\vdots$}};
\draw (bj) edge (7,1) edge (7,-1);
\draw (6.9,0.2) node[anchor=center] {{\small $\vdots$}};
\end{scope}
\begin{scope}[shift={(9.5,-2.5)}]
\draw (-2.5,0) node[anchor=center] {{$\G^-_{q} = $}};
\node [b] (bi) at (0,0) {};
\draw (0,0.6) node[anchor=center,  color=gray] {{\small $i$}};
\draw (1.5, -0.5) node[anchor=center] {{\small $\und{q}$}};
\node [b] (bn1) at (3,0) {};
\draw (3,0.6) node[anchor=center,  color=gray] {{\small $n+1$}};
\draw (4.5, -0.5) node[anchor=center] {{\small $\und{r+1}$}};
\node [b] (bj) at (6,0) {};
\draw (6,0.6) node[anchor=center,  color=gray] {{\small $j$}};
\draw (bn1) edge (bi) edge (bj);
\draw (bi) edge (-1,1) edge (-1,-1);
\draw (-0.9,0.2) node[anchor=center] {{\small $\vdots$}};
\draw (bj) edge (7,1) edge (7,-1);
\draw (6.9,0.2) node[anchor=center] {{\small $\vdots$}};
\end{scope}
\end{tikzpicture}
\caption{Producing $\G^+_q$ and $\G^-_q$ from $\G$} \label{fig:G-q-pm}
\end{figure}

According to \eqref{pa-fGC-simpler}, 
$$
\pa(\G) ~ + ~(-1)^{|\G|} \sum_{i=1}^n \,  \ga^{\ge 3}_i ~ = ~
-(-1)^{|\G|} \sum_{q} (\G^+_q + \G^-_q), 
$$
where the summation goes over all edges with distinct end 
points\footnote{The unwanted terms in $\ga_i$ should be analyzed separately in the case 
when $\G$ has a loop based at $i$. This analysis is straightforward and we leave it to the reader.}.

On the other hand, $\G^-_q = - \G^+_q$ in the space of coinvariants. 
Thus \eqref{pa-GC} follows. In particular, the subspaces $\GCo$ and $\GC$ are closed 
with respect to the differential $\pa$. 

Let $\G$ and $\ti{\G}$ be (connected) $1$-vertex irreducible graphs. 
It is clear that graphs with a cut vertex in  
$$
\G \circ_i \ti{\G}
$$
are obtained only when we connect all edges, which were adjacent to vertex $i$, 
to the same vertex of $\ti{\G}$.  It is not hard to see that all such graphs cancel 
each other in the sum
$$
\G \bullet \ti{\G} - (-1)^{|\G| |\ti{\G}|} \ti{\G} \bullet \G. 
$$
Thus,  $\GCo_{1ve}$ (resp.  $\GC_{1ve}$) is closed with respect to the bracket 
in $\fGCo$ (resp. $\fGC$). 

Since $\G^{un}_{\ed}$ is $1$-vertex irreducible, $\GCo_{1ve}$ and $\GC_{1ve}$ 
are also closed with respect to the differential.   
\end{proof}

Equation \eqref{pa-GC} implies that 
\begin{cor}
\label{cor:3-valent}
Any linear combination of connected trivalent graphs is a cocycle in $\GC$. \qed
\end{cor}

\begin{remark}
\label{rem:GC-Kontsevich}
It is the graph complex  $(\GC_{1ve}, \pa)$ which was introduced in 
\cite[Section 5.2]{K-conj}. In this paper, we refer to $(\GC_{1ve}, \pa)$ 
as \emph{Kontsevich's graph complex}.
\end{remark}

\subsubsection{$\fGC$ as a subcomplex of $\dfGC$} 

Let $\G$ be an element in $\dgra^r_n$\,.  
We denote by $\rho_j(\G)$ the graph which is obtained from 
$\G$ by changing the direction of the edge with label $\und{j}$.

It is convenient to draw the linear combination
$\G + \rho_j(\G)$ as a graph which is obtained from 
$\G$ by forgetting the direction of the edge with label $\und{j}$. 
For example, 
\begin{equation}
\label{G-un-ed}
\begin{tikzpicture}[scale=0.5, >=stealth']
\tikzstyle{w}=[circle, draw, minimum size=4, inner sep=1]
\tikzstyle{b}=[circle, draw, fill, minimum size=4, inner sep=1]
\node [b] (b1) at (0,0) {};
\draw (0,0.6) node[anchor=center,  color=gray] {{\small $1$}};
\node [b] (b2) at (2,0) {};
\draw (2,0.6) node[anchor=center,  color=gray] {{\small $2$}};
\draw (b1) edge (b2);
\draw (3.5,0.1) node[anchor=center] {{$: =$}};
\begin{scope}[shift={(5,0)}]
\node [b] (b1) at (0,0) {};
\draw (0,0.6) node[anchor=center,  color=gray] {{\small $1$}};
\node [b] (b2) at (2,0) {};
\draw (2,0.6) node[anchor=center,  color=gray] {{\small $2$}};
\draw [->] (b1) -- (b2);
\draw (3.5,0.1) node[anchor=center] {{$+$}};
\end{scope}
\begin{scope}[shift={(9.7,0)}]
\node [b] (b1) at (0,0) {};
\draw (0,0.6) node[anchor=center,  color=gray] {{\small $1$}};
\node [b] (b2) at (2,0) {};
\draw (2,0.6) node[anchor=center,  color=gray] {{\small $2$}};
\draw [->] (b2) -- (b1);
\end{scope}
\end{tikzpicture}
\end{equation}

Similarly, if the graph $\G'$ is obtained from 
$\G$ by forgetting the directions on the edges with 
labels $\und{j_1}, \dots, \und{j_p} \in \{\und{1}, \dots, \und{r}\}$, then 
$\G'$ denotes the sum 
$$
\G' ~~ = ~~ \sum_{(k_1, \dots, k_p)\in  \{0,1\}^p} ~
(\rho_{j_1})^{k_1} (\rho_{j_2})^{k_2} \dots (\rho_{j_p})^{k_p}\, (\G).
$$
For example,
\begin{equation}
\label{tri}
\begin{tikzpicture}[scale=0.5, >=stealth']
\tikzstyle{w}=[circle, draw, minimum size=4, inner sep=1]
\tikzstyle{b}=[circle, draw, fill, minimum size=4, inner sep=1]
\node [b] (b2) at (0,0) {};
\node [b] (b3) at (1,1) {};
\node [b] (b1) at (2,0) {};
\draw (0.2,0.8) node[anchor=center, color=gray] {{\small $\und{1}$}};
\draw (1.8,0.8) node[anchor=center, color=gray] {{\small $\und{2}$}};
\draw (b2) edge (b3);
\draw (b3) edge (b1);
\end{tikzpicture}
\quad
=
\quad
\begin{tikzpicture}[scale=0.5, >=stealth']
\tikzstyle{w}=[circle, draw, minimum size=4, inner sep=1]
\tikzstyle{b}=[circle, draw, fill, minimum size=4, inner sep=1]
\node [b] (b2) at (0,0) {};
\node [b] (b3) at (1,1) {};
\node [b] (b1) at (2,0) {};
\draw (0.2,0.8) node[anchor=center, color=gray] {{\small $\und{1}$}};
\draw (1.8,0.8) node[anchor=center, color=gray] {{\small $\und{2}$}};
\draw [->](b2) edge (b3);
\draw (b3) edge (b1);
\end{tikzpicture}
\quad
+
\quad
\begin{tikzpicture}[scale=0.5, >=stealth']
\tikzstyle{w}=[circle, draw, minimum size=4, inner sep=1]
\tikzstyle{b}=[circle, draw, fill, minimum size=4, inner sep=1]
\node [b] (b2) at (0,0) {};
\node [b] (b3) at (1,1) {};
\node [b] (b1) at (2,0) {};
\draw (0.2,0.8) node[anchor=center, color=gray] {{\small $\und{1}$}};
\draw (1.8,0.8) node[anchor=center, color=gray] {{\small $\und{2}$}};
\draw [<-] (b2) edge (b3);
\draw (b3) edge (b1);
\end{tikzpicture}
\end{equation} 
$$
= \quad 
\begin{tikzpicture}[scale=0.5, >=stealth']
\tikzstyle{w}=[circle, draw, minimum size=4, inner sep=1]
\tikzstyle{b}=[circle, draw, fill, minimum size=4, inner sep=1]
\node [b] (b2) at (0,0) {};
\node [b] (b3) at (1,1) {};
\node [b] (b1) at (2,0) {};
\draw (0.2,0.8) node[anchor=center, color=gray] {{\small $\und{1}$}};
\draw (1.8,0.8) node[anchor=center, color=gray] {{\small $\und{2}$}};
\draw [->] (b2) edge (b3);
\draw [->](b3) edge (b1);
\end{tikzpicture}
\quad
+
\quad
\begin{tikzpicture}[scale=0.5, >=stealth']
\tikzstyle{w}=[circle, draw, minimum size=4, inner sep=1]
\tikzstyle{b}=[circle, draw, fill, minimum size=4, inner sep=1]
\node [b] (b2) at (0,0) {};
\node [b] (b3) at (1,1) {};
\node [b] (b1) at (2,0) {};
\draw (0.2,0.8) node[anchor=center, color=gray] {{\small $\und{1}$}};
\draw (1.8,0.8) node[anchor=center, color=gray] {{\small $\und{2}$}};
\draw [->] (b2) edge (b3);
\draw [<-](b3) edge (b1);
\end{tikzpicture}
\quad
+
\quad
\begin{tikzpicture}[scale=0.5, >=stealth']
\tikzstyle{w}=[circle, draw, minimum size=4, inner sep=1]
\tikzstyle{b}=[circle, draw, fill, minimum size=4, inner sep=1]
\node [b] (b2) at (0,0) {};
\node [b] (b3) at (1,1) {};
\node [b] (b1) at (2,0) {};
\draw (0.2,0.8) node[anchor=center, color=gray] {{\small $\und{1}$}};
\draw (1.8,0.8) node[anchor=center, color=gray] {{\small $\und{2}$}};
\draw [<-] (b2) edge (b3);
\draw [->](b3) edge (b1);
\end{tikzpicture}
\quad
+
\quad
\begin{tikzpicture}[scale=0.5, >=stealth']
\tikzstyle{w}=[circle, draw, minimum size=4, inner sep=1]
\tikzstyle{b}=[circle, draw, fill, minimum size=4, inner sep=1]
\node [b] (b2) at (0,0) {};
\node [b] (b3) at (1,1) {};
\node [b] (b1) at (2,0) {};
\draw (0.2,0.8) node[anchor=center, color=gray] {{\small $\und{1}$}};
\draw (1.8,0.8) node[anchor=center, color=gray] {{\small $\und{2}$}};
\draw [<-] (b2) edge (b3);
\draw [<-] (b3) edge (b1);
\end{tikzpicture}
\quad .
$$
This way, we may view undirected graphs as well as graphs with both directed and 
undirected edges as vectors in $\dfGCo$ (and $\dfGC$). 

This identification is compatible with the binary operations \eqref{bullet} and \eqref{bullet-fGCo}.
So we will view $\GC_{1ve} \subset \GC \subset \fGC$
(resp.  $\GCo_{1ve} \subset \GCo \subset \fGCo$) as Lie 
subalgebras of $\dfGC$ (resp. $\dfGCo$). 
In addition equation \eqref{G-un-ed} implies 
that\footnote{From now on, we will drop the superscript $un$ in $\G^{un}_{\ed}$.} 
$\G^{un}_{\ed} = \G_{\ed}$.
So $\GC_{1ve} \subset \GC \subset \fGC$
(resp.  $\GCo_{1ve} \subset \GCo \subset \fGCo$) are, in fact, 
{\it dg} Lie subalgebras of $\dfGC$ (resp. $\dfGCo$).

\subsection{$\dfGC$ as the convolution Lie algebra}
\label{sec:dfGC-Conv}

According to\footnote{In \cite{stable}, the author considers directed graphs \und{without loops}. 
It is easy to see that, after adding loops, $\dGra$ is still an operad.} 
\cite[Section 3]{stable}, the graded vector spaces $(n \ge 1)$
\begin{equation}
\label{dGra-n}
\dGra(n) : =  \bigoplus_{r \ge 0} ~  \big( \bs^{ -r }  
\span_{\bbK}(\dgra_{n}^r) \otimes  \sgn_{r} \big)_{\bbS_r } 
\end{equation}
assemble to form an operad $\dGra$ in the category $\grVect_{\bbK}$.

We denote by  $\Conv(\La^2\coCom, \dGra)$ the convolution Lie algebra  \cite[Section 4]{notes}
corresponding to the cooperad $\La^2\coCom$  Moreover, we denote by  $\Conv^{\oplus}(\La^2\coCom, \dGra)$
the following subspace of  $\Conv(\La^2\coCom, \dGra)$:
\begin{equation}
\label{Conv-oplus}
\Conv^{\oplus} (\La^2\coCom, \dGra) : = \bigoplus_{n=1}^{\infty} \Hom_{\bbS_n} \big(\La^2\coCom(n), \dGra(n) \big). 
\end{equation}

It is clear that
\begin{equation}
\label{Conv-dGra-oplus}
\Conv^{\oplus} (\La^2\coCom, \dGra) = \bigoplus_{n=1}^{\infty} \bs^{2n-2} \big( \dGra(n) \big)^{\bbS_n}
\end{equation}
and
\begin{equation}
\label{Conv-dGra}
\Conv(\La^2\coCom, \dGra) = \prod_{n=1}^{\infty} \bs^{2n-2} \big( \dGra(n) \big)^{\bbS_n}\,.
\end{equation}

Since the space of invariants  $\big( \dGra(n) \big)^{\bbS_n}$ can be identified 
with the quotient space of coinvariants $\big( \dGra(n) \big)_{\bbS_n}$ via the 
isomorphism
$$
\Av (v) : = \sum_{\si \in \bbS_n} \si(v) ~:~ \big( \dGra(n) \big)_{\bbS_n}  ~\to~ \big( \dGra(n) \big)^{\bbS_n}\,,
$$
we have the obvious isomorphisms of graded vector spaces 
\begin{equation}
\label{Av-gives}
\dfGC~ \stackrel{\cong}{\longrightarrow}~ \Conv(\La^2\coCom, \dGra), \qquad 
\dfGCo ~ \stackrel{\cong}{\longrightarrow}~ \Conv^{\oplus}(\La^2\coCom, \dGra).
\end{equation}

Due to \cite[Proposition C.2]{DeligneTw}, the isomorphisms in \eqref{Av-gives}
send the bracket \eqref{dfGCo-Lie} to the Lie bracket on  $\Conv(\La^2\coCom, \dGra)$
and  $\Conv^{\oplus}(\La^2\coCom, \dGra)$, respectively. Thus equation  \eqref{dfGCo-Lie}
indeed defines a Lie bracket and $\dfGC$ (resp. $\dfGCo$) can be identified with 
$\Conv(\La^2\coCom, \dGra)$ (resp.   $\Conv^{\oplus}(\La^2\coCom, \dGra)$).

\section{The main theorem and the outline of the proof}
\label{sec:thm-proof}

Let us denote by $\dfGC_{\conn}(n)$ the subspace of $\dfGC(n)$ spanned by
(even) connected graphs with exactly $n$ vertices. It is clear that
for any pair of connected graphs $\G$ and $\ti{\G}$,
every term in the linear combination 
$[\G, \ti{\G}]$ is a connected graph. So setting
\begin{equation}
\label{dfGC-conn}
\dfGC_{\conn} : = \prod_{n \ge 1}  \dfGC_{\conn}(n)
\qquad \textrm{and} \qquad
\dfGCo_{\conn} : = \bigoplus_{n \ge 1}  \dfGC_{\conn}(n)
\end{equation}
we get the dg Lie subalgebra $\dfGC_{\conn}$ of $\dfGC$ and 
the dg Lie subalgebra $\dfGCo_{\conn}$ of $\dfGCo$, respectively.

Since every disconnected graph is a union of (finitely many) connected graphs, it is clear that 
\begin{equation}
\label{dfGC-as-S-hat}
\dfGC \cong \bs^{-2}\,  \wh{S} \big( \,\bs^2  \dfGC_{\conn} \big)
\qquad
\textrm{and}
\qquad
\dfGCo \cong  \bs^{-2} \und{S}\big( \bs^{2}\, \dfGCo_{\conn}\big).
\end{equation}

Let us denote by $\GC^{\dia}_{1ve}$ the cochain complex 
\begin{equation}
\label{GC-1ve-dia}
\GC^{\dia}_{1ve} : = \GC_{1ve}  ~ \oplus ~ \bigoplus_{m \ge 0} \bbK v_{4m- 1} \,,
\end{equation}
where $v_{4 m- 1}$ is a vector of degree $4m-1$ and $\bbK v_{4m- 1}$ is considered as 
the cochain complex with the zero differential. 

Next we upgrade the embedding $\GC_{1ve} \hookrightarrow \dfGC$
to the map of cochain complexes
$$
\Psi : \GC^{\dia}_{1ve}  \to \dfGC_{\conn}
$$ 
by setting
$$
\begin{tikzpicture}[scale=0.5, >=stealth']
\tikzstyle{w}=[circle, draw, minimum size=4, inner sep=1]
\tikzstyle{b}=[circle, draw, fill, minimum size=4, inner sep=1]
\draw (-3.4,0.2) node[anchor=center] {{$\Psi(v_{-1}) : = \G_{\lp}  = $}};
\node [b] (b1) at (0,0) {};
\draw (b1) ..controls (1,1) and (-1,1) .. (b1);
\draw (9,0.2) node[anchor=center] {{and ~~ $\Psi(v_{4m-1}) : = \G^{\dia}_{4m + 1}$~~ for~~ $m \ge 1$,}};
\end{tikzpicture}
$$
where $\G^{\dia}_{4m + 1}$ is the graph shown in figure \ref{fig:4m1}.
%
%
\begin{figure}[htp]
\centering 
\begin{tikzpicture}[scale=0.5, >=stealth']
\tikzstyle{w}=[circle, draw, minimum size=4, inner sep=1]
\tikzstyle{b}=[circle, draw, fill, minimum size=4, inner sep=1]
\node [b] (b1) at (4,0) {};
\draw (4.4,0) node[anchor=center, color=gray] {{\small $5$}};
\draw (3.40,-0.84) node[anchor=center] {{\small $\und{5}$}};
\node [b] (b2) at (3.54,-1.86) {};
\draw (4,-1.86) node[anchor=center, color=gray] {{\small $6$}};
\draw (2.62,-2.32) node[anchor=center] {{\small $\und{6}$}};
\node [b] (b3) at (2.27,-3.29) {};
\draw (2.6,-3.7) node[anchor=center, color=gray] {{\small $7$}};
\draw (1.24,-3.27) node[anchor=center] {{\small $\und{7}$}};
\node [b] (b4) at (0.48,-3.97) {};
\draw (0.48,-4.5) node[anchor=center, color=gray] {{\small $8$}};
\draw (-0.42, -3.47) node[anchor=center] {{\small $\und{8}$}};
\node [b] (b5) at (-1.42,-3.74) {};
\draw (-1.5, -4.3) node[anchor=center, color=gray] {{\small $9$}};
\draw (-1.99, -2.88) node[anchor=center] {{\small $\und{9}$}};
\node [b] (b6) at (-2.99,-2.65) {};
\draw (-3.3,-3.2) node[anchor=center, color=gray] {{\small $10$}};
\draw (-3.01,-1.63) node[anchor=center] {{\small $\und{10}$}};
\node [b] (b7) at (-3.88,-0.96) {};
\draw (-4.5,-0.96) node[anchor=center, color=gray] {{\small $11$}};
\node [b] (b9) at (-2.99,2.65) {};
\draw (-4,3.2) node[anchor=center, color=gray] {{\small $4m+1$}};
\node [b] (b10) at (-1.42,3.74) {};
\draw (-1.42,4.3) node[anchor=center, color=gray] {{\small $1$}};
\draw (-0.42, 3.47) node[anchor=center] {{\small $\und{1}$}};
\node [b] (b11) at (0.48,3.97) {};
\draw (0.48,4.6) node[anchor=center, color=gray] {{\small $2$}};
\draw (1.24, 3.27) node[anchor=center] {{\small $\und{2}$}};
\node [b] (b12) at (2.27,3.29) {};
\draw (2.27,3.9) node[anchor=center, color=gray] {{\small $3$}};
\draw (2.62,2.32) node[anchor=center] {{\small $\und{3}$}};
\node [b] (b13) at (3.54,1.86) {};
\draw (4,1.86) node[anchor=center, color=gray] {{\small $4$}};
\draw (3.40,0.84) node[anchor=center] {{\small $\und{4}$}};
\draw (-4,0.2) node[anchor=center] {{\large $\vdots$}};
\draw  (b1) edge (b2);
\draw  (b2) edge (b3);
\draw  (b3) edge (b4);
\draw  (b4) edge (b5);
\draw  (b5) edge (b6);
\draw  (b6) edge (b7);
\draw  (b9) edge (b10);
\draw (b10) edge (b11);
\draw (b11) edge (b12);
\draw (b12) edge (b13);
\draw (b13) edge (b1);
\end{tikzpicture}
\caption{The graph $\G^{\dia}_{4m + 1}$} \label{fig:4m1}
\end{figure}

Due to the first isomorphism in 
\eqref{dfGC-as-S-hat}, the map $\Psi$ upgrades further to the map of cochain complexes
\begin{equation}
\label{Psi}
\Psi :  \bs^{-2} \wh{S} \big( \,\bs^2 \GC^{\dia}_{1ve}  \, \big) ~~\to~~ \dfGC. 
\end{equation}
Moreover, the restriction of $\Psi$ to 
\begin{equation}
\label{und-S-stuff}
\bs^{-2} \und{S} \big( \,\bs^2 \GCo_{1ve}  ~ \oplus ~ \bigoplus_{m \ge 0} \bs^{2}\, \bbK v_{4m- 1}   \, \big)
\end{equation}
gives us the map of cochain complexes
\begin{equation}
\label{Psi-oplus}
\Psi^{\oplus} \,:\, 
\bs^{-2} \und{S} \big( \,\bs^2 \GCo_{1ve}  ~ \oplus ~ \bigoplus_{m \ge 0} \bs^{2}\, \bbK v_{4m- 1}   \, \big) 
~\to~ \dfGC^{\oplus}\,.
\end{equation}

In this paper, we give a careful proof of the following statements about
the full directed graph complex $\dfGC$ and its subcomplex $\dfGCo$: 
%
%
\begin{thm}
\label{thm:main}
The map \eqref{Psi} is a quasi-isomorphism of cochain complexes. 
\end{thm}
\begin{thm}
\label{thm:dfGCo}
The map \eqref{Psi-oplus} is a quasi-isomorphism of cochain complexes. 
\end{thm}

\subsection{Theorem \ref{thm:dfGCo} implies Theorem \ref{thm:main}}
\label{sec:from-oplus}

In this section, we tacitly identify $v_{-1}$ with the (isomorphism class of the) graph $\G_{\lp}$ and $v_{4m-1}$ (for $m \ge 1$)
with the (isomorphism class of the) graph $\G^{\dia}_{4m+1}$ shown in figure \ref{fig:4m1}. We also denote by $\GC^{\dia}_{1ve}(n)$ 
the following subspace of $\fGC(n)$: 
\begin{equation}
\label{GC-1ve-dia-n}
\GC^{\dia}_{1ve}(n) : = \fGC(n) \cap \GC^{\dia}_{1ve}\,.
\end{equation}
In other words, 
$$
\GC^{\dia}_{1ve}(n)  : =
\begin{cases}
\GC_{1ve}(n) \oplus \bbK [\G^{\dia}_{4m+1}]  \qquad \textrm{if} ~~ n = 4m+1 \textrm{ for some } m \ge 1 \,, \\[0.18cm]
\, \bbK [\G_{\lp}] \quad \qquad \textrm{if} ~~ n = 1\,, \\[0.18cm]
\GC_{1ve}(n) \qquad \textrm{if} ~~ n \neq 1 \mod 4\,,
\end{cases} 
$$
where  $[\G^{\dia}_{4m+1}]$ and $[\G_{\lp}]$ denotes the isomorphism class of $\G^{\dia}_{4m+1}$ and $\G_{\lp}$, 
respectively.

\subsubsection{The map \eqref{Psi} induces a surjective map on the level of cohomology.}
Every vector $\ga \in \dfGC$ of a fixed degree $d$ can be written as an infinite sum 
$$
\ga = \sum_{n \ge 1} \ga_n\,,
$$
where $\ga_n  \in \dfGC(n)^d$. In other words, $\ga_n$ is a \emph{necessarily finite} linear combination 
of graphs with $n$ vertices and $2n-2 -d$ edges. Let us assume that $\ga$ is a cocycle in $\dfGC$. 

Since $\pa \big( \dfGC(n) \big)^d  \subset \dfGC(n+1)^{d+1}$ for every $n \ge 1$, the condition
$\pa \ga = 0$ is equivalent to 
$$
\pa (\ga_n) = 0, \qquad \forall ~ n \ge 1. 
$$

Since each $\ga_n$ belongs to $\dfGCo$, Theorem \ref{thm:dfGCo} implies that 
there exists a degree $d$ cocycles $\ka_n$ in \eqref{und-S-stuff} and degree $d-1$
vectors $\te_{n-1} \in  \dfGC(n-1)$ such that 
$$
\ga_n = \begin{cases}
\Psi(\ka_n) + \pa (\te_{n-1})  \qquad {\rm if} ~~ n \ge 2 \,, \\
\Psi(\ka_1) \qquad {\rm if} ~~ n =1\,.
\end{cases}
$$

For every $n\ge 1$, $\ka_n$ is a finite linear combination of monomials of the form
\begin{equation}
\label{ka-n-terms}
\bs^{-2}\, ( \bs^2 w_1 \, \bs^2 w_2 \,\dots\, \bs^2 w_q  ), \qquad w_j \in   \GC^{\dia}_{1ve}(n_j),  
\end{equation}
where $n_1 + n_2 + \dots + n_q = n$.
 
Therefore, since  
\begin{equation}
\label{wh-S}
\bs^{-2} \wh{S} \big( \,\bs^2 \GC^{\dia}_{1ve}  \, \big) ~ = ~
\prod_{q \ge 1} ~
\prod_{n_1, \dots, n_q \ge 1} ~
\bs^{-2}\, \big( \bs^2 \GC^{\dia}_{1ve} (n_1) 
\otimes \dots \otimes \bs^2 \GC^{\dia}_{1ve} (n_q) \big)_{\bbS_q}
\end{equation}
the vector  $\displaystyle \ka : =  \sum_{n=1}^{\infty} \ka_n$
belongs to \eqref{wh-S} (i.e. the source of \eqref{Psi}). 
The vector $\ka$ is a cocycle in \eqref{wh-S} 
and we have $\ga = \Psi(\ka)  +  \pa \te$,
where $\displaystyle \te : =  \sum_{n = 1}^{\infty} \te_{n}$.

\subsubsection{The map \eqref{Psi} induces an injective map on the level of cohomology.}

Since every vector $\ka$ in \eqref{wh-S} can written as the sum 
$
\displaystyle \ka  : = \sum_{n=1}^{\infty} \ka_n\,,
$
where $\ka_n$ is a finite linear combination of monomials of the form \eqref{ka-n-terms}, we have 
$$
\Psi(\ka_n) \in \dfGC(n) \qquad \forall ~~n \ge 1.
$$

Let us now assume that $\ka$ is a degree $d$ cocycle in \eqref{wh-S} such that 
\begin{equation}
\label{Psi-kappa-exact}
\Psi (\ka) = \pa \te, 
\end{equation}
where  $\displaystyle \te : =  \sum_{n = 1}^{\infty} \te_{n}$ and $\te_n \in \dfGC(n)$. 

Since $\pa \big(  \dfGC(n) \big) \subset  \dfGC(n+1) $ for every $n \ge 1$, equation \eqref{Psi-kappa-exact} 
is equivalent to 
\begin{equation}
\label{Psi-ka-n}
\Psi (\ka_{n}) = \pa \te_{n-1} \qquad \forall ~~ n \ge 2
\end{equation}
and
$$
\Psi(\ka_1) = 0.  
$$
In particular $\ka_1 = 0$.

Since $\ka_{n+1}$ belongs to \eqref{und-S-stuff} and $\te_n \in \dfGCo$, equation \eqref{Psi-ka-n} 
implies that $\ka_{n}$ is exact in \eqref{und-S-stuff} for every $n \ge 2$. The exactness of $\ka$ in 
\eqref{wh-S} follows.  

\subsection{The cohomology of $\dfGCo$ and $\dfGC$ and the loopless version $\dfGC^{\nl}$ of $\dfGC$}
\label{sec:loopless}
Recall that, in characteristic zero, the functor $H^{\bul}$ commutes with 
taking coinvariants (and invariants) with respect to an action of a finite group. 
Therefore, combining Theorem \ref{thm:dfGCo} 
(resp. Theorem \ref{thm:main}) with the K\"unneth theorem, we get the 
following corollaries: 
\begin{cor}
\label{cor:H-dfGCo}
For the graph complex $\dfGCo$, we have 
$$
H^{\bul} (\dfGCo) \cong 
\bs^{-2} \und{S} \big( \,\bs^2 H^{\bul}(\GCo_{1ve})  ~ \oplus ~ \bigoplus_{m \ge 0} \bs^{4m+1}\, \bbK \, \big). 
$$
\end{cor} \qed
\begin{cor}
\label{cor:H-dfGC}
For the full directed graph complex $\dfGC$, we have 
$$
H^{\bul} (\dfGC) \cong 
\bs^{-2} \wh{S} \big( \,\bs^2 H^{\bul}(\GC_{1ve})  ~ \oplus ~ \bigoplus_{m \ge 0} \bs^{4m+1}\, \bbK \, \big). 
$$
\end{cor} \qed

~\\

It is clear that, if $\G \in \dgra^r_n$ does not have loops, then the linear 
combination $\pa \G$ cannot involve graphs with loops.  
So we denote by $\dfGC^{\nl}$  the ``loopless versions'' of $\dfGC$, i.e. 
\begin{equation}
\label{dfGC-nl}
\dfGC^{\nl} : = \prod_{n \ge 1}  \dfGC^{\nl}(n), 
\end{equation} 
where $\dfGC^{\nl}(n)$ is the subspace of $\dfGC(n)$ which 
is spanned by graphs (with $n$ vertices) without loops. 

Let us also denote by $\dfGC^{\lp}$ the following graded vector space
\begin{equation}
\label{dfGC-lp}
\dfGC^{\lp} : = \prod_{n \ge 1}  \dfGC^{\lp}(n), 
\end{equation} 
where $\dfGC^{\lp}(n)$ is the subspace of $\dfGC(n)$ which is spanned 
by graphs with $n$ vertices and with at least one loop. 

%
%
\begin{figure}[htp]
\centering 
\begin{tikzpicture}[scale=0.5, >=stealth']
\tikzstyle{w}=[circle, draw, minimum size=4, inner sep=1]
\tikzstyle{b}=[circle, draw, fill, minimum size=4, inner sep=1]
\tikzstyle{circ}=[circle, draw, dashed, minimum size=26, inner sep=1]
\node [b] (b) at (0,0) {};
\draw (b) ..controls (2,2) and (-2,2) .. (b);
\draw (0,2) node[anchor=center] {{\small $\und{1}$}};
\draw  [<-] (b) -- (-1,-1);
\draw  [<-] (b) -- (-2,-0.5);
\draw  [->] (b) -- (1,-1);
\draw (0,-1) node[anchor=center] {{\small $\dots$}};
\draw (2.5,0.2) node[anchor=center] {{$\stackrel{\pa}{\longrightarrow}$}};
\begin{scope}[shift={(6.5,0)}]
\draw (-2,0) node[anchor=center] {{$-$}};

\node [circ] (cir) at (0,0) {~};
\node [b] (b1) at (-0.5,0) {};
\node [b] (b2) at (0.5,0) {};
\draw (b1) -- (b2);
\draw (0,-0.4) node[anchor=center] {{\small $\und{1}$}};
\draw (cir) ..controls (2,2) and (-2,2) .. (cir);
\draw (0,2) node[anchor=center] {{\small $\und{2}$}};

\draw  [<-] (cir) -- (-1.5,-1.5);
\draw  [<-] (cir) -- (-2.5,-1);
\draw  [->] (cir) -- (1.5,-1.5);
\draw (0,-1.4) node[anchor=center] {{\small $\dots$}};
\end{scope}
\begin{scope}[shift={(14.5,0)}]
\draw (-4.5,0) node[anchor=center] {{$\longrightarrow~~\displaystyle -~ \sum ~\Big($}};
\node [b] (b1) at (-1,0) {};
\draw (-1.4,0.2) node[anchor=center, color=gray] {{\small $1$}};
\node [b] (b2) at (1,0) {};
\draw (1.4,0.2) node[anchor=center, color=gray] {{\small $2$}};
\draw [->] (b1) ..controls (-0.5,0.7) and (0.5,0.7) .. (b2);
\draw (0,1) node[anchor=center] {{\small $\und{2}$}};
\draw [->] (b2) ..controls (0.5,-0.7) and (-0.5,-0.7) .. (b1);
\draw (0,-1) node[anchor=center] {{\small $\und{1}$}};
\draw  [<-] (b1) -- (-2,-1);
\draw  [->] (b1) -- (-0.8,-1.2);
\draw (-1.4,-1) node[anchor=center] {{\small $\dots$}};
\draw  [->] (b2) -- (2,-1);
\draw  [<-] (b2) -- (0.8,-1.2);
\draw (1.4,-1) node[anchor=center] {{\small $\dots$}};
\draw (3,0) node[anchor=center] {{$+$}};
\end{scope}
\begin{scope}[shift={(20.2,0)}]
\node [b] (b1) at (-1,0) {};
\draw (-1.4,0.2) node[anchor=center, color=gray] {{\small $1$}};
\node [b] (b2) at (1,0) {};
\draw (1.4,0.2) node[anchor=center, color=gray] {{\small $2$}};
\draw [<-] (b1) ..controls (-0.5,0.7) and (0.5,0.7) .. (b2);
\draw (0,1) node[anchor=center] {{\small $\und{2}$}};
\draw [<-] (b2) ..controls (0.5,-0.7) and (-0.5,-0.7) .. (b1);
\draw (0,-1) node[anchor=center] {{\small $\und{1}$}};
\draw  [<-] (b1) -- (-2,-1);
\draw  [->] (b1) -- (-0.8,-1.2);
\draw (-1.4,-1) node[anchor=center] {{\small $\dots$}};
\draw  [->] (b2) -- (2,-1);
\draw  [<-] (b2) -- (0.8,-1.2);
\draw (1.4,-1) node[anchor=center] {{\small $\dots$}};
\draw (3,0) node[anchor=center] {{$\Big)$}};
\end{scope}
\end{tikzpicture}
\caption{Collecting terms in $\pa(\G)$ without loops} \label{fig:loop-nl}
\end{figure}
Let $\G$ be an element in $\dgra_n^r$ with exactly one loop based 
at vertex $1$ which has valency $\ge 3$. 
Figure \ref{fig:loop-nl} shows that terms without loops in 
the linear combination $\pa(\G)$ form the zero vector in $\dfGC$.

Combining this observation with the fact that $\G_{\lp}$ is a cocycle
in $\dfGC$, we conclude that the cochain complex $(\dfGC, \pa)$ splits into 
the direct sum of its subcomplexes: 
\begin{equation}
\label{dfGC-dir-sum}
\dfGC = \dfGC^{\nl} \oplus \dfGC^{\lp}\,.
\end{equation}

Let us also observe that the cochain complex \eqref{wh-S} splits into the direct sum of subcomplexes:
\begin{equation}
\label{wh-S-sum}
\bs^{-2} \wh{S} \big( \,\bs^2 \GC_{1ve}  ~ \oplus ~ \bigoplus_{m \ge 0} \bs^2 \bbK v_{4m-1}   \, \big)
=  \bs^{-2} \wh{S} \big( \bs^2 \GC^{\hs}_{1ve} \big) ~\oplus~ \Big( \bbK v_{-1} ~\oplus~
 \wh{S} \big( \bs^2 \GC^{\hs}_{1ve} \big) \otimes \bbK v_{-1}\Big),
\end{equation}
where 
$$
 \GC^{\hs}_{1ve} : =  \GC_{1ve}   \oplus \bbK v_{3} \oplus \bbK v_{7} \oplus \bbK v_{11} \oplus \dots 
$$

It is clear\footnote{Note that every $1$-vertex irreducible graph $\G$ with 
all vertices having valencies $\ge 3$ cannot have a loop.} 
that the map $\Psi$ is compatible with the splittings 
\eqref{dfGC-dir-sum} and \eqref{wh-S-sum}. 
 
Hence we proved the following statement:
\begin{prop}
\label{prop:dfGC-nl-dfGC}
The restriction of the map $\Psi$ to the subspace 
\begin{equation}
\label{wh-S-nl}
\bs^{-2} \wh{S} \big( \,\bs^2 \GC_{1ve}  \oplus \bs^2 \bbK v_{3} \oplus \bs^2 \bbK v_{7} \oplus 
\bs^2 \bbK v_{11} \oplus \dots \big) 
\end{equation}
gives us a quasi-isomorphism  
\begin{equation}
\label{Psi-nl}
\Psi^{\nl} ~:~ \bs^{-2} \wh{S} \big( \,\bs^2 \GC_{1ve}  ~ \oplus ~ \bigoplus_{m \ge 1} \bs^2 \bbK v_{4m-1}   \, \big)
 ~\to~ \dfGC^{\nl}\,.   
\end{equation}
In particular,
\begin{equation}
\label{H-dfGC-nl}
H^{\bul}(\dfGC^{\nl}) ~\cong~ 
\bs^{-2} \wh{S} \big( \,\bs^2 H^{\bul}(\GC_{1ve})  ~ \oplus ~ \bigoplus_{m \ge 1} \bs^{4m+1}\, \bbK \, \big). 
\end{equation}
\end{prop}
\qed

Combining Proposition \ref{prop:dfGC-nl-dfGC} with \cite[Theorem 1.1]{Thomas}, we deduce that
\begin{cor}  
\label{cor:dfGC-nl-grt1}
For the loopless version $\dfGC^{\nl}$ of the full directed graph 
complex $\dfGC$, we have 
$$
H^{0}(\dfGC^{\nl}) \cong \grt_1, \qquad \textrm{and} \qquad  H^{\le -1}(\dfGC^{\nl}) = \bfzero.
$$
$\qed$
\end{cor}  

\subsection{The version $\dfGC_d$ for an arbitrary even dimension $d$}
\label{sec:arb-dimen}

The cochain complex $\dfGC$ (resp.  $\dfGCo$) is a member of the family of 
graph complexes $\{ \dfGC_d \}_{d \in 2\bbZ}$ (resp.  $\{ \dfGCo_d \}_{d \in 2\bbZ}$)
indexed by an \emph{even} integer $d$. 

As the graded vector space, 
\begin{equation}
\label{dfGC-d-oplus}
\dfGC^{\oplus}_d ~ : =~ \bigoplus_{n \ge 1,~r \ge 0} ~  \big( \bs^{d n - d  + r (1-d) }  
\span_{\bbK}(\dgra_{n}^r) \otimes  \sgn_{r} \big)_{\bbS_n \times \bbS_r } \,.
\end{equation}

The Lie bracket on $\dfGC^{\oplus}_d$ is defined by the same formula \eqref{dfGCo-Lie}.
It is easy to see that the vector $\G_{\ed}$ (see \eqref{G-ed}) has degree $1$ in 
\eqref{dfGC-d-oplus} and satisfies the MC equation \eqref{G-ed-MC}. So we use the 
same formula (see \eqref{diff-dfGCo}) for the differential $\pa$ on \eqref{dfGC-d-oplus}. 

Just as for $\dfGCo$, we denote by $\dfGC_d(n)$ the subspace of $\dfGC^{\oplus}_d$ 
spanned by isomorphism classes of even graphs with exactly $n$ vertices and observe that 
$\pa \big( \dfGC_d(n) \big) \subset  \dfGC_d(n+1).$
So we define $\dfGC_d$ as the following completion of  $\dfGCo_d$:
\begin{equation}
\label{dfGC-d}
\dfGC_d ~ : =~ \prod_{n \ge 1}  \dfGC_d(n).  
\end{equation}
It is clear that  $\dfGC = \dfGC_2$ and  $\dfGCo = \dfGCo_2$. 

Similarly, by using undirected graphs, we define the subcomplexes 
$\fGC_d$, $\GC_d$ and $\GC_{1ve, d}$ of $\dfGC_d$ and the their 
``uncompleted'' version $\fGCo_d$, $\GCo_d$ and $\GCo_{1ve, d}$, 
respectively. 

To describe a link between $\dfGC_d$ (resp. $\dfGCo_d$) and $\dfGC$ (resp. $\dfGCo$), 
we remark that, for every even graph $\G$ of Euler characteristic $\chi$, $\pa(\G)$ is 
the linear combination of (even) graphs of the same Euler characteristic. Thus $\dfGCo_d$ splits into 
the direct sum of cochain complexes
\begin{equation}
\label{dfGCo-Euler}
\dfGCo_d ~ = ~ \bigoplus_{\chi \in \bbZ} \dfGC_{d, \chi} 
\end{equation}
and $\dfGC_d$ is isomorphic to the direct product
\begin{equation}
\label{dfGC-Euler}
\dfGC_d ~ = ~ \prod_{\chi \in \bbZ} \dfGC_{d, \chi}\,, 
\end{equation}
where $\dfGC_{d, \chi}$ is the subcomplex of $\dfGCo$ spanned by isomorphism classes 
of even graphs whose Euler characteristic is $\chi$. 
   
Note that  the degree of an even graph $\G$ can be expressed in terms of the
number of vertices $n$ and its Euler characteristic $\chi$ by the formula
\begin{equation}
\label{deg-n-chi}
\deg(\G) = n + (d-1)\chi - d. 
\end{equation}
Thus, there are finitely many (isomorphisms classes of) graphs of fixed Euler characteristic and 
fixed degree. 

Due to equation \eqref{deg-n-chi}, we have the obvious isomorphism of 
cochain complexes (for every even $d$):
\begin{equation}
\label{isom-shift}
 \dfGC_{d, \chi} \cong \bs^{(d-2) (\chi-1)} \, \dfGC_{2, \chi}\,.
\end{equation}

Using this isomorphism, we easily obtain the following generalizations of statements 
of Theorems \ref{thm:main} and  \ref{thm:dfGCo}:
\begin{thm}  
\label{thm:arb-d}
Let $d$ be any even integer and $v_{4m+1-d}$ be a symbol of degree $4m+1-d$. 
The natural embeddings
\begin{equation}
\label{arb-d-oplus}
\bs^{-d} \und{S} \big( \,\bs^d \GCo_{1ve, d}  ~ \oplus ~ \bigoplus_{m \ge 0} \bs^{d}\, \bbK v_{4m+1-d}   \, \big) 
~\hookrightarrow~ \dfGC^{\oplus}_d\,,
\end{equation}
\begin{equation}
\label{arb-d}
\bs^{-d} \wh{S} \big( \,\bs^d \GC_{1ve, d}  ~ \oplus ~ \bigoplus_{m \ge 0} \bs^{d}\, \bbK v_{4m+1-d}   \, \big) 
~\hookrightarrow~ \dfGC_d
\end{equation}
are quasi-isomorphisms of cochain complexes. 
\end{thm}  
$\qed$

\begin{remark}  
\label{rem:arb-d}
One can define families of graph complexes $\{\fGC_d\}_{d \in \bbZ}$,  $\{\GC_d\}_{d \in \bbZ}$, 
and $ \{\GC_{1ve, d}\}_{d \in \bbZ}$ indexed by an arbitrary (not necessarily even) integer $d$ and 
we refer the reader 
to \cite{A-Turchin}, \cite{hairy-stuff}, \cite{slides}, \cite{KWZ}, \cite{KWZ11}, \cite{Thomas-slides}, 
\cite{Thomas}, \cite{WZ} for more details about these families of 
graph complexes and their generalizations. 
For odd $d$, the directions on edges play a special role. So Theorem \ref{thm:arb-d}
does not have a generalization to the case when $d$ odd. 
\end{remark}

\subsection{The proof of Theorem \ref{thm:dfGCo}}
\label{sec:outline}
In this section, we deduce Theorem \ref{thm:dfGCo} from several 
auxiliary statements. The remainder of this paper is devoted to 
the proofs of these auxiliary statements. 

Recall that $\dfGC_{\conn}(n)$ is the subspace of $\dfGC(n)$ spanned by
(even) \emph{connected} graphs. Moreover, $\dfGCo_{\conn}$ is the dg Lie subalgebra of 
$\dfGCo$ introduced in \eqref{dfGC-conn}.

Let us split the cochain complex $(\dfGCo_{\conn}, \pa)$ into the following direct sum of subcomplexes: 
\begin{equation}
\label{dfGCo-conn-oplus}
\dfGCo_{\conn} ~ = ~ \dfGCo_{\conn, \ge 3} ~\oplus~ \dfGCo_{\conn, \dia}  ~\oplus~   
 \dfGCo_{\conn, -}\,, 
\end{equation}
where
\begin{itemize}
\item  $\dfGCo_{\conn, \ge 3}$ is the subcomplex spanned by (even) connected graphs with at least 
one vertex having valency $\ge 3$, 

\item $\dfGCo_{\conn, \dia}$ is the subcomplex spanned by (even) connected graphs with 
all vertices having valency $2$ (i.e. $\G_{\lp}$ and various polygons), and

\item $\dfGCo_{\conn, -}$ is the subcomplex spanned by $\G_{\bul} \in \dgra^0_1$ and
uni-bivalent (even) connected graphs, i.e. a path graphs (an example of a path graph is shown in figure \ref{fig:path}). 

\end{itemize}
\begin{figure}[htp]
\centering 
\begin{tikzpicture}[scale=0.5, >=stealth']
\tikzstyle{w}=[circle, draw, minimum size=4, inner sep=1]
\tikzstyle{b}=[circle, draw, fill, minimum size=4, inner sep=1]
\tikzstyle{circ}=[circle, draw, dashed, minimum size=26, inner sep=1]
\node [b] (b1) at (0,0) {};
\node [b] (b2) at (2,0) {};
\node [b] (b3) at (4,0) {};
\node [b] (b4) at (6,0) {};
\draw  [<-] (b1) -- (b2);
\draw [->] (b2) -- (b3);
\draw [->] (b3) -- (b4);
\end{tikzpicture}
\caption{An example of an even uni-bivalent graph} \label{fig:path}
\end{figure}
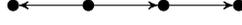

It is clear that the restriction of $\Psi^{\oplus}$ (see \eqref{Psi-oplus}) to 
$\displaystyle \bigoplus_{m \ge 0} \bbK v_{4m-1}$
gives us a chain map 
\begin{equation}
\label{Psi-polygons}
\Psi_{\dia} ~ : ~  \bigoplus_{m \ge 0} \bbK v_{4m-1} ~\to~ \dfGCo_{\conn, \dia}\,.
\end{equation}

Thus Theorem \ref{thm:dfGCo} follows from the following propositions:
\begin{prop}
\label{prop:GCo-dfGCo3}
The natural embeddings 
\begin{equation}
\label{GCo-to-dfGCo3}
\GCo  \hookrightarrow \dfGCo_{\conn, \ge 3}
\end{equation}
and 
\begin{equation}
\label{GCo1ve-to-GCo}
\GCo_{1ve}  \hookrightarrow \GCo
\end{equation}
are quasi-isomorphisms of cochain complexes. 
\end{prop}
\begin{prop}
\label{prop:dfGCo-diamond}
The chain map \eqref{Psi-polygons} is a quasi-isomorphism of cochain complexes. 
\end{prop}
\begin{prop}
\label{prop:dfGCo-path}
The subcomplex $\dfGCo_{\conn, -}$ is acyclic. 
\end{prop}
The first part of Proposition \ref{prop:GCo-dfGCo3} is proved in Section \ref{sec:dfGCo-conn-ge3} and 
the second part is proved in Section \ref{sec:GC1ve-to-GC}.  Propositions \ref{prop:dfGCo-diamond} and 
\ref{prop:dfGCo-path} are proved in Appendix \ref{app:polygons-paths}.

\begin{remark}
\label{rem:Thomas-App}
We should remark that the proof of the statement which is very similar to 
the first part of Proposition \ref{prop:GCo-dfGCo3} is sketched in Appendix K of \cite{Thomas}. 
The sketch of the proof given in {\it loc. cit.} is admittedly very brief. 
So we decided to give a careful proof of this statement. 
\end{remark}

\section{Analyzing the subcomplex $\dfGCo_{\conn, \ge 3}$ }
\label{sec:dfGCo-conn-ge3}

In this section, we prove that the embedding \eqref{GCo-to-dfGCo3} is
a quasi-isomorphism of cochain complexes. 

Let $\G$ be an even connected graph with at least one vertex of valency $\ge 3$. 
Let us denote by $\nu_2(\G)$ the number of bivalent 
vertices of $\G$.

It is clear that the linear combination  
$$
\pa \G
$$
may involve only graphs $\G'$ with $\nu_2(\G') = \nu_2(\G)$
or  $\nu_2(\G') = \nu_2(\G) + 1$\,.

Thus we may introduce on the complex $\dfGCo_{\conn, \ge 3}$ an 
ascending filtration 
\begin{equation}
\label{cF-dfGC-ge-3}
\dots \subset \cF^{m-1} \dfGC^{\oplus}_{\conn, \ge 3}
\subset  \cF^m \dfGC^{\oplus}_{\conn, \ge 3} \subset  \cF^{m+1} \dfGC^{\oplus}_{\conn, \ge 3}
\subset \dots
\end{equation}
where $ \cF^m \dfGC^{\oplus}_{\conn, \ge 3}$ consists of vectors  
$\ga \in \dfGC^{\oplus}_{\conn, \ge 3}$
which only involve graphs $\G$ satisfying the inequality $\nu_2(\G)  -  |\ga|  \le  m. $

It is clear that 
$$
 \cF^{m} \dfGCo_{\conn, \ge 3}
$$
does not have non-zero vectors in degree $< -m$\,.
Therefore, the filtration \eqref{cF-dfGC-ge-3} is locally bounded 
from the left. Furthermore, 
$$
\dfGCo_{\conn, \ge 3} = \bigcup_{m} \cF^{m} \dfGCo_{\conn, \ge 3} \,.
$$
In other words,  the filtration \eqref{cF-dfGC-ge-3} is cocomplete.
 
It is also clear that the differential $\pa^{\Gr}$ on the associated graded 
complex 
\begin{equation}
\label{Gr-sfgraphs}
\Gr( \dfGC^{\oplus}_{\conn, \ge 3}) =
\bigoplus_m \cF^m \dfGC^{\oplus}_{\conn, \ge 3}  ~ \Big/~ \cF^{m-1} \dfGC^{\oplus}_{\conn, \ge 3}\,. 
\end{equation} 
is obtained from $\pa$ by keeping only the terms 
which raise the number of the bivalent vertices. 

Thus, since $\GCo$ is a subcomplex of $\dfGC^{\oplus}_{\conn, \ge 3}$, 
we conclude that
$$
(\GC^{\oplus})^k \subset  \cF^{-k} (\dfGC^{\oplus}_{\conn, \ge 3})^k~ \cap ~ \ker \pa^{\Gr}\,, 
$$ 
where $(\GC^{\oplus})^k$ (resp. $ \cF^{-k} (\dfGC^{\oplus}_{\conn, \ge 3})^k$) denotes the 
subspace of degree $k$ vectors in $\GC^{\oplus}$  (resp. in $ \cF^{-k} \dfGC^{\oplus}_{\conn, \ge 3}$)\,.

We need the following technical lemma which is proved in Section \ref{sec:lem:Gr} below.
\begin{lem}
\label{lem:Gr}
For the filtration  \eqref{cF-dfGC-ge-3}  on $\dfGCo_{\conn, \ge 3}$ we have 
\begin{equation}
\label{Gr-sfgraphs-k-m}
 H^{k}\Big( \cF^m \dfGC^{\oplus}_{\conn, \ge 3}  \big/  \cF^{m-1} \dfGC^{\oplus}_{\conn, \ge 3} \Big) = 0 
\end{equation}
for all $m > - k$\,. Moreover, 
\begin{equation}
\label{sfgraphs-OK}
(\GC^{\oplus})^k =  \cF^{-k} (\dfGC^{\oplus}_{\conn, \ge 3})^k~ \cap ~ \ker \pa^{\Gr}\,.
\end{equation} 
\end{lem}

It is easy to see that the restriction of  
\eqref{cF-dfGC-ge-3} to the subcomplex 
$\GC^{\oplus}$ gives us the ``silly'' filtration: 
\begin{equation}
\label{silly-filtr}
\cF^{m} (\GC^{\oplus})^k  = 
\begin{cases}
(\GC^{\oplus})^k \qquad {\rm if} ~~ m \ge -k\,,  \\
\bfzero  \qquad {\rm otherwise}
\end{cases} 
\end{equation}
and the associated graded complex $\Gr(\GC^{\oplus})$ for this 
filtration has the zero differential.

Since 
$$
 \cF^{m} (\dfGC^{\oplus}_{\conn, \ge 3})^k = \bfzero \qquad \forall ~~ m < -k\,,
$$
we have 
$$
\cF^{-k} (\dfGC^{\oplus}_{\conn, \ge 3})^k~ \cap ~ \ker \pa^{\Gr} = 
  H^k \big( \, \cF^{-k} \dfGC^{\oplus}_{\conn, \ge 3} \big/  \cF^{-k-1} \dfGC^{\oplus}_{\conn, \ge 3}  \, \big)\,.
$$

Thus, Lemma \ref{lem:Gr} implies that, the embedding  \eqref{GCo-to-dfGCo3} 
induces a quasi-isomorphism of cochain complexes 
$\displaystyle \Gr (\GCo)  \stackrel{\sim}{\longrightarrow}  \Gr(\dfGCo_{\conn, \ge 3}).$
On the other hand, both filtrations \eqref{cF-dfGC-ge-3} and 
\eqref{silly-filtr} are locally bounded from 
the left and cocomplete. 
Therefore the embedding \eqref{GCo-to-dfGCo3}
satisfies all the conditions of Lemma A.3 from \cite[Appendix A]{notes}
and hence it is a quasi-isomorphism.

\subsection{Proof of Lemma \ref{lem:Gr}} \label{sec:lem:Gr}

\subsubsection{Frames}

To every positive integer $n$, we assign an auxiliary groupoid $\Frame_{n}$\,. 
An object\footnote{We call objects of the groupoid $\Frame_{n}$ {\it frames}.} 
of this groupoid is a \und{connected} directed graph $\Gim \in \dgra_n^r$ for some 
$r \ge 2$ satisfying the following properties: 
\begin{itemize}

\item $\Gim$ does not have bivalent vertices and at least one vertex of $\Gim$
has valency $\ge 3$; for example the graphs shown in figure \ref{fig:not-frames} are not frames
while the graphs shown in figures \ref{fig:kissing-loops} and \ref{fig:poked-loop} are frames;

\item each edge adjacent to a univalent vertex (if any) of $\Gim$ originates at this univalent vertex; 
 
\item the edges of $\Gim$ are labeled in 
such a way that edges incident to univalent vertices 
(if any) precede all the remaining edges; 

\item finally, loops of $\Gim$ (if any) go after all the remaining edges.

\end{itemize}
\begin{figure}[htp] 
\begin{minipage}[t]{0.35\linewidth}
\centering 
\begin{tikzpicture}[scale=1, >=stealth']
\tikzstyle{w}=[circle, draw, minimum size=3, inner sep=1]
\tikzstyle{b}=[circle, draw, fill, minimum size=3, inner sep=1]
\node [b] (b1) at (0,0) {};
\node [b] (b2) at (2,0) {};
\draw [->] (b1) edge (b2);
\draw (1,1) circle (0.5);
\node [b] (bb) at (1,0.5) {};
\end{tikzpicture}
\caption{These graphs \und{are not} frames} \label{fig:not-frames}
\end{minipage} 
~
\begin{minipage}[t]{0.3\linewidth}
\centering 
\begin{tikzpicture}[scale=0.5, >=stealth']
\tikzstyle{w}=[circle, draw, minimum size=3, inner sep=1]
\tikzstyle{b}=[circle, draw, fill, minimum size=3, inner sep=1]
\node [b] (b1) at (0,0) {};
\draw (-1,0) circle (1);
\draw (1,0) circle (1);
\end{tikzpicture}
\caption{This graph \und{is} a frame} \label{fig:kissing-loops}
\end{minipage}
~
\begin{minipage}[t]{0.3\linewidth}
\centering 
\begin{tikzpicture}[scale=0.5, >=stealth']
\tikzstyle{w}=[circle, draw, minimum size=3, inner sep=1]
\tikzstyle{b}=[circle, draw, fill, minimum size=3, inner sep=1]
\node [b] (b1) at (1,0) {};
\node [b] (b2) at (2.5,0) {};
\draw [->] (b2) edge (b1);
\draw (0,0) circle (1);
\end{tikzpicture}
\caption{This graph \und{is} a frame} \label{fig:poked-loop}
\end{minipage}
\end{figure}

A morphism from a frame $\Gim$ to a frame $\Gim'$ is 
an isomorphism of the underlying graphs which respects  
neither the total order on the set of edges, nor the directions of
edges. For example, there are exactly two isomorphism classes of 
frames with $2$ edges and the corresponding representatives 
are shown in figures \ref{fig:kissing-loops} and \ref{fig:poked-loop}.

\begin{example}
\label{exam:frame}
Note that frames {\it may} have multiple edges with the same direction. For example, 
the frame $\Gim$ shown in figure \ref{fig:frame-Gim} has a double edge.
It is clear that the frame $\Gim$ has only one non-trivial automorphism, 
i.e. the one which interchanges edges $\und{5}$ and $\und{6}$.

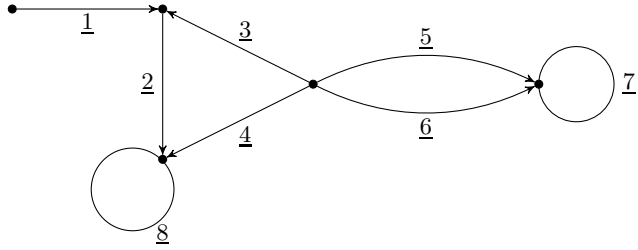
\begin{figure}[htp] 
\centering 
\begin{tikzpicture}[scale=1, >=stealth']
\tikzstyle{w}=[circle, draw, minimum size=3, inner sep=1]
\tikzstyle{b}=[circle, draw, fill, minimum size=3, inner sep=1]
\draw (0,1.5) node[anchor=center] {{\phantom{aaaaaaaaaaaa}}};
\node [b] (b1) at (0,0) {};
\node [b] (b2) at (3,0) {};
\node [b] (b3) at (-2,1) {};
\node [b] (b5) at (-4,1) {};
\node [b] (b4) at (-2,-1) {};
\draw [->] (b1) ..controls (1,0.5) and (2,0.5) .. (b2);
\draw (1.5,0.6) node[anchor=center] {{\small $\und{5}$}};
\draw [->] (b1) ..controls (1,-0.5) and (2,-0.5) .. (b2);
\draw (1.5,-0.6) node[anchor=center] {{\small $\und{6}$}};
\draw (3.5,0) circle (0.5);
\draw (4.2,0) node[anchor=center] {{\small $\und{7}$}};
\draw [->] (b1) edge (b3);
\draw (-0.9,0.7) node[anchor=center] {{\small $\und{3}$}};
\draw [->] (b1) edge (b4);
\draw (-0.9,-0.7) node[anchor=center] {{\small $\und{4}$}};
\draw [->] (b3) edge (b4);
\draw (-2.2,0) node[anchor=center] {{\small $\und{2}$}};
\draw [->] (b5) edge (b3);
\draw (-3,0.8) node[anchor=center] {{\small $\und{1}$}};
\draw (-2.4,-1.4) circle (0.55);
\draw (-2,-2) node[anchor=center] {{\small $\und{8}$}};
\end{tikzpicture}
\caption{An example of a frame $\Gim$. Labels for vertices are not shown} \label{fig:frame-Gim}
\end{figure}
\end{example}

The total number of edges $e$ of any frame $\Gim$ splits into the sum $e = e_{\bul}  + e_{-} + e_{\c}\,,$
where $e_{\bul}$ is the number of 
edges of $\Gim$ adjacent to univalent vertices (if any), $e_{\c}$ is the number of loops of $\Gim$ (if any) 
and $e_{-}$ is the number of the remaining edges. 

To describe the associated graded complex 
$$
\Gr(\dfGCo_{\conn, \ge 3}),
$$
we denote by $\cV_2$ the two dimensional vector space 
spanned by symbols $a$ and $b$ placed in degree $1$. 
We consider the truncated tensor algebra of $\cV_2$
\begin{equation} 
\label{red-tensor-alg}
\und{T}(\cV_2) : = \cV_2 \oplus (\cV_2)^{\otimes \, 2} \oplus  (\cV_2)^{\otimes \, 3} \oplus \dots,
\end{equation}
introduce on (the underlying vector space of) $\und{T}(\cV_2)$ the following two differentials
\begin{equation}
\mb_{1}(v_1 v_2 \dots v_n) ~: =~ \sum_{i =1}^{n} (-1)^{i+1} \, v_1 \dots v_{i} (a+b) v_{i+1} \dots v_n 
\end{equation}
\begin{equation}
\mb_{2} (v_1 v_2 \dots v_n) ~: =~ - (a+b) v_1  \dots v_n + 
\mb_{1}(v_1 v_2 \dots v_n)\,, \quad v_i  \in \cV_2
\end{equation}
and denote by\footnote{$U$ stands for ``{\bf u}nivalent'' and $R$ stands for ``{\bf r}emaining''.} $U$ and $R$ 
the corresponding cochain complexes:
\begin{equation}
\label{UR}
U =(\und{T}(\cV_2),\mb_{1}), \qquad R = (\und{T}(\cV_2),\mb_{2}).
\end{equation}
For our purposes, it is convenient to use the simplified notation $v_1 v_2 \dots v_n$ 
for the tensor monomial $v_1 \otimes v_2 \otimes \dots \otimes v_n$ in $(\cV_2)^{\otimes \, n}$\,.
Of course, in general,  $v_1 \dots v_i v_{i+1}\dots  v_n \neq  - v_1 \dots v_{i+1} v_{i}\dots  v_n $ 
in $(\cV_2)^{\otimes \, n}$.
 
We observe that the cochain complex $R$ carries the following action of the group $\bbS_2$
\begin{equation}
\label{S2-action}
\si (v_1 v_2 \dots v_n) : = (-1)^{\frac{n(n-1)}{2}}\, \overline{v_n} \, \overline{v_{n-1}} \, \dots \,  \overline{v_1}\,,  
\end{equation}
where $\si = (1,2) \in \bbS_2$, each $v_j$ is either $a$ or $b$, and 
\[
\overline{v_{j}} =
\begin{cases}
b & \text{if $v_{j} =a$},\\
a  & \text{if $v_{j} =b$}.
\end{cases}
\]

Given a frame $\Gim \in \Frame_{n}$ with $e = e_{\bul}  + e_{-} + e_{\c}$ edges and 
a tensor product of monomials 
\begin{equation}
\label{P-here}
P  =  p_{1} \tensor \cdots \tensor p_{e_{\bul}} \tensor p_{e_{\bul} +1} \tensor  \cdots  \tensor  p_{e} \in 
U^{\otimes\, e_{\bul}}  \otimes   R^{\otimes\, (e_{-} + e_{\c})} 
\end{equation}
we can form a graph $\G \in \dgra_m^r$, where
$$
m =  n +  \sum_{i=1}^{e} (\deg(p_i) - 1) \qquad \textrm{and} \qquad r = \deg(p_1) + \deg(p_2) + \dots + \deg(p_e).
$$

To better visualize this construction, it is convenient to think of the tensor 
product of monomials $P$ as the sequence of disjoint arrows: the symbol $a$
corresponds to the arrows pointing to the right and the symbol $b$ corresponds to 
the arrow pointing to the left. Such a sequence of arrows has exactly $2 e$ vertices. 
We label these vertices in the natural order from left to right. For example, the tensor 
product 
$$
b \otimes  a b \otimes  a \otimes b \otimes b a \otimes a \otimes aba \otimes b
$$
corresponds to the disjoint graph:
$$
\begin{tikzpicture}[scale=0.7, >=stealth']
\tikzstyle{w}=[circle, draw, minimum size=3, inner sep=1]
\tikzstyle{b}=[circle, draw, fill, minimum size=3, inner sep=1]
\node [b] (v1) at (0,0) {};
\draw (0,0.25) node[anchor=center] {{\tiny $1$}};
\node [b] (v2) at (1,0) {};
\draw (1,0.25) node[anchor=center] {{\tiny $2$}};
\draw [<-] (v1) edge (v2);
\node [b] (v3) at (1.5,0) {};
\draw (1.5,0.25) node[anchor=center] {{\tiny $3$}};
\node [b] (v4) at (2.5,0) {};
\draw (2.5,0.25) node[anchor=center] {{\tiny $4$}};
\draw [->] (v3) edge (v4);
\node [b] (v5) at (3,0) {};
\draw (3,0.25) node[anchor=center] {{\tiny $5$}};
\node [b] (v6) at (4,0) {};
\draw (4,0.25) node[anchor=center] {{\tiny $6$}};
\draw [<-] (v5) edge (v6);
\node [b] (v7) at (4.5,0) {};
\draw (4.5,0.25) node[anchor=center] {{\tiny $7$}};
\node [b] (v8) at (5.5,0) {};
\draw (5.5,0.25) node[anchor=center] {{\tiny $8$}};
\draw [->] (v7) edge (v8);
\node [b] (v9) at (6,0) {};
\draw (6,0.25) node[anchor=center] {{\tiny $9$}};
\node [b] (v10) at (7,0) {};
\draw (7,0.25) node[anchor=center] {{\tiny $10$}};
\draw [<-] (v9) edge (v10);
\node [b] (v11) at (7.5,0) {};
\draw (7.5,0.25) node[anchor=center] {{\tiny $11$}};
\node [b] (v12) at (8.5,0) {};
\draw (8.5,0.25) node[anchor=center] {{\tiny $12$}};
\draw [<-] (v11) edge (v12);
\node [b] (v13) at (9,0) {};
\draw (9,0.25) node[anchor=center] {{\tiny $13$}};
\node [b] (v14) at (10,0) {};
\draw (10,0.25) node[anchor=center] {{\tiny $14$}};
\draw [->] (v13) edge (v14);
\node [b] (v15) at (10.5,0) {};
\draw (10.5,0.25) node[anchor=center] {{\tiny $15$}};
\node [b] (v16) at (11.5,0) {};
\draw (11.5,0.25) node[anchor=center] {{\tiny $16$}};
\draw [->] (v15) edge (v16);
\node [b] (v17) at (12,0) {};
\draw (12,0.25) node[anchor=center] {{\tiny $17$}};
\node [b] (v18) at (13,0) {};
\draw (13,0.25) node[anchor=center] {{\tiny $18$}};
\draw [->] (v17) edge (v18);
\node [b] (v19) at (13.5,0) {};
\draw (13.5,0.25) node[anchor=center] {{\tiny $19$}};
\node [b] (v20) at (14.5,0) {};
\draw (14.5,0.25) node[anchor=center] {{\tiny $20$}};
\draw [<-] (v19) edge (v20);
\node [b] (v21) at (15,0) {};
\draw (15,0.25) node[anchor=center] {{\tiny $21$}};
\node [b] (v22) at (16,0) {};
\draw (16,0.25) node[anchor=center] {{\tiny $22$}};
\draw [->] (v21) edge (v22);
\node [b] (v23) at (16.5,0) {};
\draw (16.5,0.25) node[anchor=center] {{\tiny $23$}};
\node [b] (v24) at (17.5,0) {};
\draw (17.5,0.25) node[anchor=center] {{\tiny $24$}};
\draw [<-] (v23) edge (v24);
\end{tikzpicture}
$$

To form the graph $\G$ corresponding to the monomial \eqref{P-here} and the 
frame $\Gim$, we follow these steps: 

\begin{itemize}

\item For every $ 1 \le i \le e$, we introduce vertices $s_i$ and $t_i$ and 
convert the symbols of the  monomial $p_i$ into $k_i$ directed edges which 
connect $s_i$ with $t_i$ by these rule: if a symbol $w$ in $p_i$ stands to the left of the symbol $w'$
then the edge corresponding to $w$ is closer to $s_i$ than the edge corresponding to $w'$, 
if the $j$-th symbol in $p_i$
is $a$ then we direct corresponding edge toward $t_i$; otherwise, we direct 
the corresponding edge towards $s_i$. For example, if $p_i = aaba$ then we 
assign to $p_i$ this string of edges 
$$
\begin{tikzpicture}[scale=1, >=stealth']
\tikzstyle{w}=[circle, draw, minimum size=3, inner sep=1]
\tikzstyle{b}=[circle, draw, fill, minimum size=3, inner sep=1]
\node [b] (b1) at (0,0) {};
\draw (0,0.3) node[anchor=center] {{\small $s_i$}};
\node [b] (b2) at (1,0) {};
\node [b] (b3) at (2,0) {};
\node [b] (b4) at (3,0) {};
\node [b] (b5) at (4,0) {};
\draw (4,0.3) node[anchor=center] {{\small $t_i$}};
\draw [->] (b1) edge (b2) (b2) edge (b3) (b4) edge (b3) edge (b5);
\end{tikzpicture}
$$ 

\item If the $i$-th edge of $\Gim$ is not a loop, then we replace it
with the above string of directed edges corresponding to $p_i$ via identifying
$s_i$ (resp. $t_i$) with what was the source (resp. the target) of the $i$-th edge of $\Gim$. 

\item If the $i$-th edge of $\Gim$ is a loop based, say, at vertex $v$ 
then we remove it and identify both $s_i$ and $t_i$ with $v$. 

\item Since the resulting graph $\G$ is obtained by gluing edges corresponding  
to symbols in the tensor product $P$, we get the obvious map 
$$
\psi: \{1,2, \dots, 2 e \} \to V(\G),
$$
where $V(\G)$ is the set of vertices of $\G$. We order these vertices according 
to this rule: $v_1 < v_2$ $~\Leftrightarrow~$ $\min(\psi^{-1}(v_1)) < \min(\psi^{-1}(v_2))$.  

\item  Since the set of edges of $\G$ is in bijection with symbols in  
\eqref{P-here}, it is already equipped with a total order. 

\end{itemize}
In what follows, we will say that $\G$ is {\it the graph reconstructed from}
the tensor product of monomials in \eqref{P-here} using the frame $\Gim$. 

It is easy to see that the assignment  $F_{\Gim} (P) : = \G$
defines a degree zero map 
\begin{equation}
\label{F-Gim}
F_{\Gim} : \bs^{2n - 2 - 2 e} \,  U^{\otimes\, e_{\bul}}  
\otimes   R^{\otimes\, (e_{-} + e_{\c})} 
~\to~  \Gr (\dfGCo_{\conn, \ge 3})\,.
\end{equation}

%
%
\begin{example}
\label{exam:F-Gim}
Let $\Gim$ be the frame shown in figure \ref{fig:frame-Gim} and 
$P$ be the tensor product of monomials  
\begin{equation}
\label{P-exam}
P = b \otimes ( a b \otimes  a \otimes b \otimes b a \otimes a \otimes aba \otimes b) 
\in U \otimes R^{\otimes \, 7}\,. 
\end{equation}
Then $F_{\Gim} (P) = \G$, where $\G$ is the labeled graph 
shown in figure \ref{fig:G-ogo-go}.
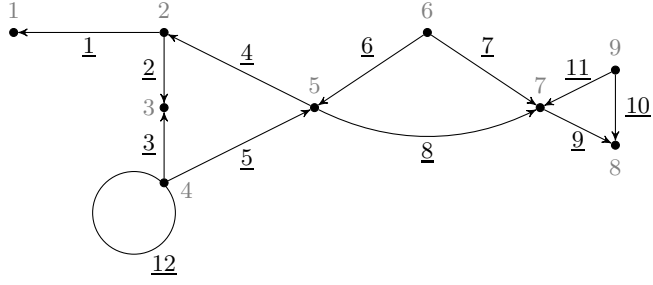
\begin{figure}[htp] 
\centering 
\begin{tikzpicture}[scale=1, >=stealth']
\tikzstyle{w}=[circle, draw, minimum size=3, inner sep=1]
\tikzstyle{b}=[circle, draw, fill, minimum size=3, inner sep=1]
\node [b] (b1) at (0,0) {};
\draw (0,0.3) node[anchor=center, color = gray] {{\small $5$}};
\node [b] (b7) at (1.5,1) {};
\draw (1.5,1.3) node[anchor=center, color = gray] {{\small $6$}};
\node [b] (b2) at (3,0) {};
\draw (3,0.3) node[anchor=center, color = gray] {{\small $7$}};
\node [b] (b3) at (-2,1) {};
\draw (-2,1.3) node[anchor=center, color = gray] {{\small $2$}};
\node [b] (b6) at (-2,0) {};
\draw (-2.2,0) node[anchor=center, color = gray] {{\small $3$}};
\node [b] (b5) at (-4,1) {};
\draw (-4,1.3) node[anchor=center, color = gray] {{\small $1$}};
\node [b] (b4) at (-2,-1) {};
\draw (-1.7,-1.1) node[anchor=center, color = gray] {{\small $4$}};
\node [b] (b8) at (4,-0.5) {};
\draw (4,-0.8) node[anchor=center, color = gray] {{\small $8$}};
\node [b] (b9) at (4,0.5) {};
\draw (4,0.8) node[anchor=center, color = gray] {{\small $9$}};
\draw [->] (b2) edge (b8); 
\draw [->] (b9) edge (b8); 
\draw [->] (b9) edge (b2); 
\draw (3.5,-0.45) node[anchor=center] {{\small $\und{9}$}};
\draw (3.5,0.5) node[anchor=center] {{\small $\und{11}$}};
\draw (4.3,0) node[anchor=center] {{\small $\und{10}$}};
\draw [<-] (b1) edge (b7); \draw [->] (b7) edge (b2); 
\draw (0.7,0.8) node[anchor=center] {{\small $\und{6}$}};
\draw (2.3,0.8) node[anchor=center] {{\small $\und{7}$}};
\draw [->] (b1) ..controls (1,-0.5) and (2,-0.5) .. (b2);
\draw (1.5,-0.6) node[anchor=center] {{\small $\und{8}$}};
\draw [->] (b1) edge (b3);
\draw (-0.9,0.7) node[anchor=center] {{\small $\und{4}$}};
\draw [->] (b4) edge (b1);
\draw (-0.9,-0.7) node[anchor=center] {{\small $\und{5}$}};
\draw [->] (b3) edge (b6);
\draw (-2.2, 0.5) node[anchor=center] {{\small $\und{2}$}};
\draw [<-] (b6) edge (b4);
\draw (-2.2, -0.5) node[anchor=center] {{\small $\und{3}$}};
\draw [<-] (b5) edge (b3);
\draw (-3,0.8) node[anchor=center] {{\small $\und{1}$}};
\draw (-2.4,-1.4) circle (0.55);
\draw (-2,-2.1) node[anchor=center] {{\small $\und{12}$}};
\end{tikzpicture}
\caption{The graph reconstructed from the tensor product of monomials \eqref{P-exam}} \label{fig:G-ogo-go}
\end{figure}
\end{example}

\begin{remark}
\label{rem:zero}
It may happen that the graph $\G$ reconstructed from a tensor product of monomials 
$P$ in $ U^{\otimes\, e_{\bul}}  \otimes   R^{\otimes\, (e_{-} + e_{\c}) } $   
is odd and hence $F_{\Gim}(P) = 0$\,. For example if $\Gim$ is the frame shown 
in figure  \ref{fig:frame-Gim} then the graph reconstructed from the monomial 
$$
P' = b \otimes ( a b \otimes  a \otimes b \otimes b \otimes b \otimes aba \otimes b )
\in U \otimes R^{\otimes \, 7} 
$$
has a pair of edges with the same direction connecting the same vertices.  
Hence $F_{\Gim}(P') = 0$\,.

Here is another example. If 
$$
Q = aba \otimes aba
$$
then the graph reconstructed from $Q$ using the frame 
$\Gim_{\lp \lp}$ in figure \ref{fig:kissing-loops} is shown in figure 
\ref{fig:kissing-triangles}.
\begin{figure}[htp] 
\centering 
\begin{tikzpicture}[scale=1, >=stealth']
\tikzstyle{w}=[circle, draw, minimum size=3, inner sep=1]
\tikzstyle{b}=[circle, draw, fill, minimum size=3, inner sep=1]
\node [b] (b1) at (0,0) {};
\node [b] (b2) at (-1,-0.5) {};
\node [b] (b3) at (-1,0.5) {};
\node [b] (b4) at (1,-0.5) {};
\node [b] (b5) at (1,0.5) {};
\draw [->] (b1) edge (b2) (b3) edge (b2) edge (b1)
(b1) edge (b4) (b5) edge (b4) edge (b1); 
\draw (-0.4,-0.4) node[anchor=center] {{\small $\und{1}$}};
\draw (-1.2,0) node[anchor=center] {{\small $\und{2}$}};
\draw (-0.4,0.45) node[anchor=center] {{\small $\und{3}$}};
\draw (0.4,-0.4) node[anchor=center] {{\small $\und{4}$}};
\draw (1.2,0) node[anchor=center] {{\small $\und{5}$}};
\draw (0.4,0.45) node[anchor=center] {{\small $\und{6}$}};
\end{tikzpicture}
\caption{The labeling of vertices is not specified} \label{fig:kissing-triangles}
\end{figure}
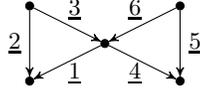
This graph has the obvious automorphism which ``switched the triangles'' and 
this automorphism gives us the odd permutation $(1,4)(2,5)(3,6) $ in $\bbS_6$. 
Hence $F_{\Gim_{\lp \lp}} (Q) = 0$\,. 
\end{remark}

Let us denote by $\pa^{\Gr}$ the differential on 
the associated graded complex $\Gr( \dfGC^{\oplus}_{\conn, \ge 3})$\,.
It is clear from the definition of the filtration \eqref{cF-dfGC-ge-3}
on $\dfGC^{\oplus}_{\conn, \ge 3}$ that  $\pa^{\Gr}$
is obtained from $\pa$ by keeping only the terms 
which raise the number of 
the bivalent vertices. Hence the image of the map 
$F_{\Gim}$ is closed with respect to the action the differential 
$\pa^{\Gr}$\,. 

Using equation \eqref{pa-dfGC-simpler} it is easy to show that
for every frame $\Gim$,  
\begin{equation}
\label{F-Gim-diff}
\pa^{\Gr} \circ F_{\Gim} = F_{\Gim} \circ \de \,,
\end{equation}
where $\de$ is the differential on the vector space
\begin{equation}
\label{source-FGim}
\bs^{2n-2 - 2e} \,  U^{\otimes\, e_{\bul}}  \otimes  R^{\otimes\, (e_{-} + e_{\c})}
\end{equation}
given by the formula $\delta = \mb_{1} \tensor 1 + 1 \tensor \mb_{2}\,.$

\subsubsection{The kernel of $F_{\Gim}$}
\label{sec:ker-F-Gim}

To describe the kernel of the map $F_{\Gim}$, we introduce the semi-direct product 
$\bbS_e \ltimes \big(\bbS_2\big)^{e}$
of the groups $\bbS_e$ and $\big(\bbS_2\big)^{e}$
with the multiplication rule: 
\begin{equation}
\label{group-law}
(\tau; \si_1, \dots, \si_{e}) \cdot (\la; \si'_1, \dots, \si'_{e}) =  
(\tau \la; \si_{\la(1)} \si'_1, \dots, \si_{\la(e)} \si'_e)\,,
\end{equation}
where, as above, $e$ is the number of edges of the frame $\Gim$.

The subgroup
\begin{equation}
\label{really-rearrange}
\cG_{\Gim} : = \Big( \bbS_{e_{\bul}} \times \bbS_{e_{-}} \times \bbS_{e_{\c}} \Big) 
 \ltimes \Big( \{\id\}^{e_{\bul}}  \times  \bbS_2^{\,(e_{-}  + e_{\c})}  \Big)
\end{equation}
of $\bbS_e \ltimes \big(\bbS_2\big)^{e}$ acts on the graded vector space \eqref{source-FGim} in the following way:
\begin{itemize}

\item{If $\si = (1,2) \in \bbS_2$ and $e_{\bul} < i  \le  e $  
then
$$
(1, \dots, 1,
\underbrace{\si}_{i\textrm{-th spot}},1, \dots 1) (p_{1}  \tensor \cdots  \tensor p_{e_{\bul}} \tensor p_{e_{\bul} +1}  \tensor \cdots  \tensor p_{e}  ) : = 
$$
$$
p_{1} \tensor  \cdots  \tensor p_{e_{\bul}} \tensor p_{e_{\bul} +1}  \tensor \cdots  \tensor \si(p_i)  \tensor \cdots  \tensor p_{e} ,
$$
where $\si(p_i)$ is defined in \eqref{S2-action}. }

\item{ For every $\tau \in  \bbS_{e_{\bul}} \times \bbS_{e_{-}} \times \bbS_{e_{\c}}$ we set
$$
\tau (p_{1} \tensor p_2 \tensor \cdots \tensor p_{e}) ~ : =~ 
(-1)^{\ve(\tau)} \, p_{\tau^{-1}(1)} \tensor p_{\tau^{-1}(2)} \tensor \cdots \tensor  p_{\tau^{-1} (e)}\,, 
$$
where the sign factor $(-1)^{\ve(\tau)}$ is determined by  the usual Koszul rule. 
}

\end{itemize}

We will now prove the following claim: 
\begin{claim} 
\label{cl:grpaction}
Let $\Gim$ be a frame with $e = e_{\bul} + e_{-} + e_{\c}$ edges and $\G, \G'$ be
the graphs reconstructed from the tensor products of monomials 
$$ 
X= p_{1} \tensor p_2 \tensor \cdots \tensor p_{e}\,,
\qquad
X' = p'_{1} \tensor p'_2 \tensor \cdots \tensor p'_{e}
$$
in \eqref{source-FGim}, respectively, using $\Gim$. 
Let $k$ be the number of edges of $\G$ and $\vf  \maps \G \to \G'$ be
an isomorphism of graphs which induces a permutation $\si_{\vf} \in \bbS_k$. 
If $\si_{\vf}$ is even, then there exists an element $g \in \cG_{\Gim}$ such that
$g (X) = X'.$  If $\si_{\vf}$ is odd, then there exists an
element $g \in \cG_{\Gim}$ such that $g (X) = - X'.$  
\end{claim}
\begin{proof}
Let $\ti{\vf}$ be the automorphism of $\Gim$ corresponding to the isomorphism $\vf : \G \to \G'$
and $\tau$ be the permutation in $\bbS_{e_{\bul}} \times \bbS_{e_{-}} \times \bbS_{e_{\c}}$ coming from $\ti{\vf}$.   

By construction, the edges of $\G$ (resp. $\G'$) are in bijection with 
symbols of the monomial in $X$ (resp. $X'$). Hence the isomorphism $\vf$ gives 
us a bijection from the set of symbols of $X$ to the set of symbols of $X'$. 
Furthermore, for every $1 \le i \le e$ the isomorphism $\vf$ gives us a bijection 
from the set of symbols of $p_i$ to the set of symbols of $p'_{\tau(i)}$.  
 
Thus we set 
$$
g = (\tau; \si_{1}, \dots, \si_{e}) \in \cG_{\Gim}\,,
$$ 
where the elements $\si_i \in \bbS_2$ are specified by considering 
these three cases:  
\begin{itemize}

\item[{\bf Case 1:}]  $1 \le i \le e_{\bul}$, i.e. the $i$-th edge of $\Gim$ originates at 
a univalent vertex. Since the isomorphism $\vf$ is compatible with the directions of edges (in $\G$ and $\G'$), 
in this case we have  $p'_{\tau(i)} = p_{i}$.

\item[{\bf Case 2:}]   $e_{\bul} + 1 \le i \le e_{\bul} + e_{-}$, i.e. the $i$-th edge of $\Gim$ 
connects two distinct vertices of valencies $\ge 3$. In this case we have two possibilities:
if the automorphism $\ti{\vf}$ of $\Gim$ respects the directions of edges $i$ and $\tau(i)$, 
then $p'_{\tau(i)} = p_{i}$; 
if the automorphism $\ti{\vf}$ of $\Gim$ does not respect the directions of edges $i$ and $\tau(i)$, 
then $p'_{\tau(i)}$ coincides with $\si(p_i)$ up to the appropriate sign factor, where $\si= (12) \in \bbS_2$
and  $\si(p_i)$ is defined in \eqref{S2-action}. If the first possibility realizes, we set $\si_i = \id \in \bbS_2$. 
Otherwise, we set $\si_i = (12) \in \bbS_2$. 
 
\item[{\bf Case 3:}]   $e_{\bul} + e_{-} +1 \le i \le e$, i.e. the $i$-th edge of $\Gim$ is a loop.  
Here we have four possibilities. First, if $p'_{\tau(i)}  = p_i = a $ or $p'_{\tau(i)}  = p_i = b $ then we set  
we set $\si_i = \id \in \bbS_2$. Second, if $p'_{\tau(i)}  =a, ~p_i = b $ or  $p'_{\tau(i)}  =b, ~p_i = a $, then we 
set $\si_i = (12) \in \bbS_2$. Third, if  $p_i$ has more than one symbol and $\vf$ respects the 
orders on the edges coming from symbols of $p_i$ and $p'_{\tau(i)}$, then $p'_{\tau(i)} = p_i$ and 
we set  $\si_i = \id \in \bbS_2$. Fourth, if $p_i$ has more than one symbol and $\vf$ does not respect the 
orders on the edges coming from symbols of $p_i$ and $p'_{\tau(i)}$, then $p'_{\tau(i)}$ coincides with $\si(p_i)$ up 
to the appropriate sign factor and we set  $\si_i = (12) \in \bbS_2$.

\end{itemize}

This analysis shows that, for the element $g \in \cG_{\Gim}$ constructed in this way, 
$g(X) = X'$ or $g (X) = -X'$ depending on whether the permutation $\si_{\vf} \in \bbS_k$
is even or odd, respectively. 
\end{proof}

Setting $X = X'$ in Claim \ref{cl:grpaction}, we get the following obvious corollary:   
\begin{cor} 
\label{cor:when-Gamma-odd}
Let $\Gim$ be a frame with $e = e_{\bul} + e_{-} + e_{\c}$ edges and $\G$ 
be the graph reconstructed from the tensor product of monomials 
\[ 
X= p_{1} \tensor p_{2} \tensor \cdots \tensor p_{e}
\]
in \eqref{source-FGim} using $\Gim$. If $\G$ has an automorphism 
which induces an odd permutation on the set of edges (i.e. $\G$ is odd) 
then there exists an element $g \in \cG_{\Gim}$ such that $g (X) = - X.$  
$~~~\qed$
\end{cor}

\begin{example}
\label{exam:P-Ppr-ker}
Let $\Gim$ be the frame shown in figure \ref{fig:frame-Gim} and 
$P, P'$ be the following tensor products of monomials in $U \otimes R^{\otimes \, 7}$
$$
P = b \otimes ( a b \otimes  a \otimes b \otimes b a \otimes a \otimes aba \otimes b )\,, 
\qquad 
P' = b \otimes ( a b \otimes  a \otimes b \otimes a \otimes b a \otimes bab \otimes a )\,, 
$$
The graph $\G$ (resp. $\G'$) reconstructed from $P$ (resp. $P'$) using the frame $\Gim$
is shown in figure \ref{fig:G-ogo-go} (resp. in figure \ref{fig:from-Ppr}).
\begin{figure}[htp] 
\centering 
\begin{tikzpicture}[scale=1, >=stealth']
\tikzstyle{w}=[circle, draw, minimum size=3, inner sep=1]
\tikzstyle{b}=[circle, draw, fill, minimum size=3, inner sep=1]
\node [b] (b1) at (0,0) {};
\draw (0,0.3) node[anchor=center, color = gray] {{\small $5$}};
\node [b] (b7) at (1.5,-1) {};
\draw (1.5, -1.3) node[anchor=center, color = gray] {{\small $7$}};
\node [b] (b2) at (3,0) {};
\draw (3,0.3) node[anchor=center, color = gray] {{\small $6$}};
\node [b] (b3) at (-2,1) {};
\draw (-2,1.3) node[anchor=center, color = gray] {{\small $2$}};
\node [b] (b6) at (-2,0) {};
\draw (-2.2,0) node[anchor=center, color = gray] {{\small $3$}};
\node [b] (b5) at (-4,1) {};
\draw (-4,1.3) node[anchor=center, color = gray] {{\small $1$}};
\node [b] (b4) at (-2,-1) {};
\draw (-1.7,-1.1) node[anchor=center, color = gray] {{\small $4$}};
\node [b] (b8) at (4,-0.5) {};
\draw (4,-0.8) node[anchor=center, color = gray] {{\small $9$}};
\node [b] (b9) at (4,0.5) {};
\draw (4,0.8) node[anchor=center, color = gray] {{\small $8$}};
\draw [->] (b2) edge (b8); 
\draw [->] (b9) edge (b8); 
\draw [->] (b9) edge (b2); 
\draw (3.5,-0.5) node[anchor=center] {{\small $\und{11}$}};
\draw (3.5,0.5) node[anchor=center] {{\small $\und{9}$}};
\draw (4.3,0) node[anchor=center] {{\small $\und{10}$}};
\draw [<-] (b1) edge (b7); \draw [->] (b7) edge (b2); 
\draw (0.6,-0.7) node[anchor=center] {{\small $\und{7}$}};
\draw (2.4,-0.7) node[anchor=center] {{\small $\und{8}$}};
\draw [->] (b1) ..controls (1,0.5) and (2,0.5) .. (b2);
\draw (1.5,0.56) node[anchor=center] {{\small $\und{6}$}};
\draw [->] (b1) edge (b3);
\draw (-0.9,0.7) node[anchor=center] {{\small $\und{4}$}};
\draw [->] (b4) edge (b1);
\draw (-0.9,-0.7) node[anchor=center] {{\small $\und{5}$}};
\draw [->] (b3) edge (b6);
\draw (-2.2, 0.5) node[anchor=center] {{\small $\und{2}$}};
\draw [<-] (b6) edge (b4);
\draw (-2.2, -0.5) node[anchor=center] {{\small $\und{3}$}};
\draw [<-] (b5) edge (b3);
\draw (-3,0.8) node[anchor=center] {{\small $\und{1}$}};
\draw (-2.4,-1.4) circle (0.55);
\draw (-2,-2.1) node[anchor=center] {{\small $\und{12}$}};
\end{tikzpicture}
\caption{The graph $\G'$ reconstructed from $P'$ using the frame $\Gim$} \label{fig:from-Ppr}
\end{figure}
We have the obvious isomorphism $\vf : \G \to \G'$. This isomorphism acts by identity
on the set of vertices $\{1,2,\dots, 9\}$ and the action on the ordinals of edges is 
defined by this permutation in $\bbS_{12}$
$$
\si_{\vf} = (6,7,8) (9,11)\,.
$$

The automorphism $\ti{\vf} \in \Aut(\Gim)$ corresponding to the isomorphism $\vf$
is the only non-trivial automorphism, i.e. the one which flips the double edges
$\und{5}$ and $\und{6}$ of $\Gim$. So the permutation $\tau \in \bbS_8$ is the transposition $(5,6)$. 

Finally, we see that $P' = - (\tau; 1,1,1,1,1,1, \si, \si)\, (P)\,, $
where $\tau = (5,6) \in \bbS_8$ and $\si = (1,2) \in \bbS_2$\,.
It agrees with the fact that the permutation $\si_{\vf} = (6,7,8)(9,11)$ is odd. 
Thus, $F_{\Gim}(P) = - F_{\Gim}(P')$\,.
\end{example}

\begin{example}
\label{exam:kissing-loops}
Let $\Gim_{\lp \lp}$ is the frame shown  in figure \ref{fig:kissing-loops}
and $Q = aba \otimes aba$. The graph $\G_{Q}$ reconstructed from $Q$ using $\Gim_{\lp \lp}$
is shown in figure \ref{fig:kissing-triangles}. As we saw above, $\G_Q$ has the obvious automorphism
$\vf$ which ``switched the triangles'' and induces the permutation $\si_{\vf} = (1,4)(2,5)(3,6) \in \bbS_6$. 
It is clear that $Q = - (\tau; 1,1)\, (Q),$ where $\tau = (1,2)$\,.
\end{example}

Let now use Claim \ref{cl:grpaction} and Corollary \ref{cor:when-Gamma-odd} to 
describe the kernel of the map $F_{\Gim}$ \eqref{F-Gim}: 

\begin{claim} 
\label{cl:ker-F-Gim}
Let $\Gim \in \Frame_{n}$ be a frame with $e$ edges 
$$
e = e_{\bul} + e_{-} + e_{\c}\,,
$$
where $e_{\bul}$ is the number of edges of $\Gim$ 
adjacent to univalent  vertices, $e_{\c}$ is the 
number of loops and $e_{-} = e - e_{\bul} - e_{\c}$\,.
Then the kernel of $F_{\Gim}$ is spanned 
by vectors of the form 
\begin{equation}
\label{ker-F-Gim}
X -  g (X)\,,
\end{equation}
where $X$ is a vector in \eqref{source-FGim} and $g \in \cG_{\Gim}$.
\end{claim}
\begin{proof} Let us show that, for every tensor product of monomials $X$ in \eqref{source-FGim}
and for every $g \in \cG_{\Gim}$, 
\begin{equation}
\label{req-inclusion}
X - g (X)  \in  \ker F_{\Gim}\,.
\end{equation}

If the graph $\G$ reconstructed from $X$ is odd then  $F_{\Gim}(X) = 0$
and, similarly, $F_{\Gim}(g(X)) = 0$. So let us assume that $\G$ is even.

If the element $g$ induces an even permutation on the ordinal of edges of $\G$ then 
\begin{itemize}

\item[{\it i)}] $g(X)$ is also a tensor product of monomials, 

\item[{\it ii)}] $g$ gives us an isomorphism from $\G$ to 
the graph $\G'$ reconstructed from $g(X)$, and

\item[{\it iii)}] this isomorphism induces an even permutation in $\bbS_k$, where $k$ is the number of edges of $\G$
(i.e. the graphs $\G$ and $\G'$ are\footnote{See Definition \ref{dfn:concordant} and Remark \ref{rem:sign-si-vf}.}
concordant).  

\end{itemize}

Thus $\G = \G'$ in $\dfGCo$ and, in this case, inclusion \eqref{req-inclusion} holds.

If the element $g$ induces an odd permutation on the ordinal of edges of $\G$ then 
\begin{itemize}

\item[{\it i)}] $-g(X)$ is a tensor product of monomials, 

\item[{\it ii)}] $g$ gives us an isomorphism from $\G$ to 
the graph $\G'$ reconstructed from $-g(X)$, and

\item[{\it iii)}] this isomorphism induces an odd permutation in $\bbS_k$, where $k$ is the number of edges of $\G$.  
(i.e. the graphs $\G$ and $\G'$ are opposite).

\end{itemize}

Thus $\G + \G' = 0$ in $\dfGCo$ and, in this case, inclusion \eqref{req-inclusion} also holds. 

Let us now consider a tensor product of monomials $Y$ in  \eqref{source-FGim} 
satisfying the property
\begin{equation}
\label{F-Gim-Y-0}
F_{\Gim}(Y) = 0\,.
\end{equation}
The latter means that the graph $\G$ reconstructed from $Y$ is odd, i.e.\ there 
exists an automorphism $\vf$ of $\G$ which induces an odd permutation on the set of edges of $\G$\,.
Corollary \ref{cor:when-Gamma-odd} implies that there exists $g \in \cG_{\Gim}$ such that $Y = - g (Y)$. 
Hence, $Y = \frac{1}{2} (Y - g(Y))$.
Thus every tensor product of monomials $Y$  in  \eqref{source-FGim}
satisfying equation \eqref{F-Gim-Y-0} belongs to the span
of vectors of the form \eqref{ker-F-Gim}. 

Let us now consider a linear combination 
\begin{equation}
\label{combin-Ys}
c_1 Y_1 + c_2 Y_2 + \dots + c_m Y_m\,, \qquad c_i \in \bbK 
\end{equation}
of tensor products of monomials $Y_1, \dots, Y_m$ in   \eqref{source-FGim} such that 
\begin{equation}
\label{F-Gim-Y-s}
\sum_i c_i F_{\Gim}(Y_i) = 0\,.
\end{equation}

Due to the above observation about monomials 
satisfying \eqref{F-Gim-Y-0} we may assume, without loss 
of generality, that 
$$
F_{\Gim}(Y_i)  \neq 0 \qquad \forall~~ 1 \le i \le m\,. 
$$ 
In other words, the graph $\G_i$ reconstructed from $Y_i$ is even 
for every $1 \le i \le m$. 

We may also assume, without loss of generality, that 
the graphs $\{\G_i \}_{1 \le i \le m}$ have the same number 
of vertices $n+n'$.

Thus, for every $1 \le i \le m$, the graph $\G_i \in \dgra_{n+n'}$
is even.

By definition of $\dfGCo$, the number $m$ is even and the set of graphs 
$\{\G_i \}_{1 \le i \le m}$ splits into pairs
$$
(\G_{i_{t}}, \G_{i'_{t}})\,, \qquad t \in  \{1, \dots, m/2 \}
$$
such that for every $t$ the graphs $\G_{i_{t}}$  and $ \G_{i'_{t}}$
are either opposite or concordant. For every 
pair $(\G_{i_{t}}, \G_{i'_{t}})$ of opposite graphs we have 
\begin{equation}
\label{for-opposite}
c_{i_t} = c_{i'_t}\,.  
\end{equation}
For every pair $(\G_{i_{t}}, \G_{i'_{t}})$ of concordant graphs we have 
\begin{equation}
\label{for-concordant}
c_{i_t} = -c_{i'_t}\,.  
\end{equation}

Let $e_t$ denote the number of edges of $\G_{i_t}$ (or  $\G_{i'_t}$) and 
let $\vf_t$ be the isomorphism from $\G_{i_t}$ to $\G_{i'_t}$ which 
induces an odd or even permutation in $\bbS_{e_t}$
depending on whether  $\G_{i_t}$ and $\G_{i'_t}$ are  opposite or 
concordant. Let $g_t \in \cG_{\Gim}$ be the element induced by the isomorphism $\vf_t$
as in Claim \ref{cl:grpaction}.
Equations \eqref{for-opposite} and \eqref{for-concordant} and Claim \ref{cl:grpaction} imply that
$$
\sum_{i=1}^m c_i Y_i = 
\sum_{t=1}^{m/2} c_{i_t} (Y_{i_t} - g_t (Y_{i_t}))\,.
$$

In other words, the linear combination \eqref{combin-Ys} belongs 
to the span of vectors of the form \eqref{ker-F-Gim} and the claim
follows. 
\end{proof}

\subsubsection{Description of the associated graded complex $\Gr( \dfGC^{\oplus}_{\conn, \ge 3} )$} 

Now we are ready to give a convenient description of $\Gr( \dfGC^{\oplus}_{\conn, \ge 3} )$\,. 
\begin{claim}
\label{cl:Gr-sfgraphs}
Let us choose a representative $\Gim_z$ for every isomorphism class 
$z \in \pi_0(\Frame_{n})$\,. Let 
$$
e^z = e^z_{\bul} + e^{z}_{-} + e^{z}_{\c}
$$
be the number of edges of $\Gim_z$, where $e^z_{\bul}$ is the number of edges 
adjacent to univalent  vertices, $e^z_{\c}$ is the 
number of loops and $e^z_{-}$ is the number of the remaining edges. 
Then the cochain complex  $\Gr( \dfGC^{\oplus}_{\conn, \ge 3} )$ splits into the 
direct sum 
$$
\Gr( \dfGC^{\oplus}_{\conn, \ge 3}) \cong 
$$
\begin{equation}
\label{Gr-sfgraphs-desc}
\bigoplus_{n \ge 1}~
\bigoplus_{z \in \pi_0(\Frame_{n})}~
\bs^{2 n - 2 -2 e^z} 
\bigl( U^{\otimes\, e^z_{\bul}} \otimes R^{\otimes \, (e^z_{-} + e^z_{\c})} \bigr)_{\cG_{\Gim_{z}}} \,,
\end{equation}
where $U$ and $R$ are the cochain complexes introduced in \eqref{UR}. 
\end{claim} 

\begin{proof}
Let us recall that the map $F_{\Gim_z}$ \eqref{F-Gim} is a morphism 
from the cochain complex 
$$
\bs^{2 n - 2 -2 e^z}  U^{\otimes\, e^z_{\bul}} \otimes R^{\otimes \, (e^z_{-} + e^z_{\c})} 
$$
to $\Gr( \dfGC^{\oplus}_{\conn, \ge 3})$\,. 

Thus, Claim \ref{cl:ker-F-Gim} implies that $F_{\Gim_z}$ induces 
an isomorphism from the cochain complex of 
coinvariants
$$
\bs^{2 n - 2 -2 e^z} 
\bigl( U^{\otimes\, e^z_{\bul}} \otimes R^{\otimes \, (e^z_{-} + e^z_{\c})} \bigr)_{\cG_{\Gim_{z}}}
$$
to the subcomplex  
$$
\Im(F_{\Gim_z}) \subset  \Gr( \dfGC^{\oplus}_{\conn, \ge 3})\,.
$$

On the other hand, the cochain complex $\Gr( \dfGC^{\oplus}_{\conn, \ge 3})$ 
is obviously the direct sum  
\begin{equation}
\label{sum-over-Gim}
\Gr( \dfGC^{\oplus}_{\conn, \ge 3}) =
\bigoplus_{n \ge 1}~
\bigoplus_{z \in \pi_0(\Frame_{n})}~ \Im(F_{\Gim_z})\,.
\end{equation}

Thus, the desired statement follows. 
\end{proof}

\subsubsection{The end of the proof of Lemma \ref{lem:Gr}}
We now have all we need to prove Lemma \ref{lem:Gr}. 
Let $\Gim_z$ be a representative of an isomorphism class $z \in \pi_0(\Frame_{n})$
and let 
$$
e^z = e^z_{\bul} + e^{z}_{-} + e^{z}_{\c}
$$
be the number of edges of $\Gim_z$, where $e^z_{\bul}$ is the number of edges 
adjacent to univalent vertices, $e^z_{\c}$ is the 
number of loops and $e^z_{-}$ is the number of the remaining edges.

Let us consider the cochain complex 
\begin{equation} 
\label{cohom_cmplx}
\bs^{2 n - 2 -2 e^z}  U^{\otimes\, e^z_{\bul}} \otimes R^{\otimes \, (e^z_{-} + e^z_{\c})}\,. 
\end{equation}

Due to Claim \ref{cl:U-and-R} from Appendix \ref{app:cohomol-T-V2}, the cochain complex 
$U$ is acyclic, 
$$
H^k (R) = \begin{cases}
 \bbK \qquad  {\rm if} ~~ k=1\,, \\
 \bfzero \qquad \textrm{otherwise}\,,
\end{cases}
$$ 
and the $H^1(R)$ is spanned by the class represented by the cocycle $(a+b)$. 

Thus K\"{u}nneth's theorem implies that 
$$
H^k \Big( \bs^{2 n - 2 -2 e^z}  U^{\otimes\, e^z_{\bul}} \otimes R^{\otimes \, (e^z_{-} + e^z_{\c})} \Big) =
\begin{cases}
\bbK  \qquad {\rm if} ~~  e^z_{\bul} =0~ \textrm{ and }~ k = 2n - 2 - e^z\,, \\
\bfzero  \qquad {\rm otherwise}\,.
\end{cases}
$$

Combining this observation with Claim \ref{cl:Gr-sfgraphs}, we conclude that 
\eqref{Gr-sfgraphs-k-m} and \eqref{sfgraphs-OK} indeed hold.  

Thus Lemma \ref{lem:Gr} is proved. \qed

\section{On the subcomplex of 1-vertex irreducible graphs $\GC_{1ve}$}
\label{sec:GC1ve-to-GC}

The goal of this section is to prove the second part of Proposition \ref{prop:GCo-dfGCo3}, 
i.e. that the chain map 
\begin{equation}
\label{from-GCo1ve}
\GCo_{1ve}  \hookrightarrow \GCo
\end{equation}
is a quasi-isomorphism of cochain complexes. 
 
The corresponding statement for the (commutative) graph complex 
introduced in \cite{K-noncom-symp} was proved in 
\cite{C-Vogtmann}. However, since the (commutative) graph complex 
from \cite{K-noncom-symp} is quite different 
from $\GCo$, we decided to write a detailed proof of the fact 
that \eqref{from-GCo1ve} is a quasi-isomorphism. 
Needless to say, this proof is similar to the one given 
in \cite{C-Vogtmann} for the (commutative) graph complex 
from \cite{K-noncom-symp}.

\subsection{The filtration by the number of separating edges}
\label{sec:dual-GC}
Let $\G$ be a connected graph whose all vertices have valencies $\ge 3$. 
Recall that $i \in V(\G)$ is a {\it cut vertex} if $\G$ becomes disconnected upon 
deleting $i$. Furthermore, we call $e \in E(\G)$ a {\it separating edge} if $\G$ becomes 
disconnected upon deleting $e$. For a graph $\G$ with at least one cut vertex, 
we denote by $\G^{\red}$ the graph 
which is obtained from $\G$ by contracting all separating edges. If $\G$ does not have 
separating edges, we simply set $\G^{\red} : =  \G$.

To prove that \eqref{from-GCo1ve} is a quasi-isomorphism, we denote by $ \QCo$
the quotient complex 
\begin{equation}
\label{QCo}
\QCo\, : = \,  \GCo \, \big/\, \GCo_{1ve}
\end{equation}
and observe that $\QCo$ is the span of isomorphism classes of even connected
graphs with at least one cut vertex and with all vertices having valency $\ge 3$.  

Next, we equip  the complex $\QCo$ with the following ascending filtration 
\begin{equation}
\label{eq:filt_sepedge}
\dots \subset \cF_{m-1} \QCo
\subset  \cF_m \QCo  \subset  \cF_{m+1} \QCo  \subset \dots
\end{equation}
where $\cF_m \QCo$ consists of vectors  $\ga \in \QCo$
which only involve graphs $\G$ satisfying the inequality
$$
\text{$\#$ of separating edges of $\G$}  -  |\ga|  \le m.
$$

To describe the associated graded cochain complex, we consider 
graphs $\beth$ which fulfill the condition:
\begin{cond} 
\label{cond:beth}
$\beth$ is a connected graph with all vertices having valency $\ge 3$. 
$\beth$ has at least 1 cut vertex but it does not have separating edges. 
\end{cond}
An example of a graph satisfying the above condition is shown in figure \ref{fig:beth}.
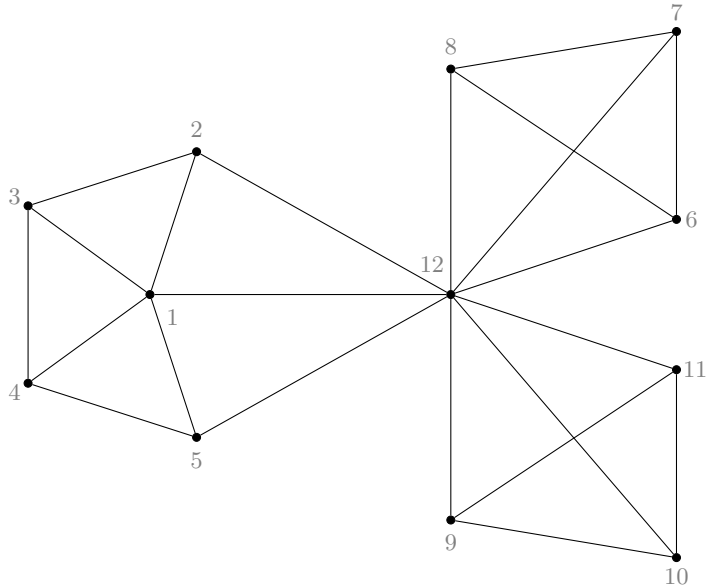
\begin{figure}[htp]
\centering 
\begin{tikzpicture}[scale=1, >=stealth']
\tikzstyle{w}=[circle, draw, minimum size=3, inner sep=1]
\tikzstyle{b}=[circle, draw, fill, minimum size=3, inner sep=1]
\node [b] (b0) at (0,0) {};
\draw (0.3,-0.3) node[anchor=center, color =gray] {{\small $1$}};
\node [b] (b1) at (4,0) {}; 
\draw (3.75,0.4) node[anchor=center, color =gray] {{\small $12$}};
\node [b] (b2) at (0.62, 1.90) {};
\draw (0.62, 2.2) node[anchor=center, color =gray] {{\small $2$}};
\node [b] (b3) at (-1.62, 1.18) {};
\draw (-1.8, 1.3) node[anchor=center, color =gray] {{\small $3$}};
\node [b] (b4) at (-1.62, -1.18) {};
\draw (-1.8, -1.3) node[anchor=center, color =gray] {{\small $4$}};
\node [b] (b5) at (0.62, -1.90) {};
\draw (0.62, -2.2) node[anchor=center, color =gray] {{\small $5$}};
\draw (b1) edge (b2) (b2) edge (b3) 
(b3) edge (b4) (b4) edge (b5) (b5) edge (b1);
\draw (b0) edge (b1)  edge (b2) edge (b3) edge (b4)  edge (b5);
\node [b] (t1) at (7,1) {};
\draw (7.2, 1) node[anchor=center, color =gray] {{\small $6$}};
\node [b] (t2) at (4,3) {};
\draw (4, 3.3) node[anchor=center, color =gray] {{\small $8$}};
\node [b] (t3) at (7,3.5) {};
\draw (7, 3.75) node[anchor=center, color =gray] {{\small $7$}};
\draw (b1) edge (t1) edge (t2) edge (t3) (t1) edge (t2) edge (t3)  (t2) edge (t3);
\node [b] (tt1) at (7,-1) {};
\draw (7.25, -1) node[anchor=center, color =gray] {{\small $11$}};
\node [b] (tt2) at (4,-3) {};
\draw (4, -3.3) node[anchor=center, color =gray] {{\small $9$}};
\node [b] (tt3) at (7,-3.5) {};
\draw (7, -3.75) node[anchor=center, color =gray] {{\small $10$}};
\draw (b1) edge (tt1) edge (tt2) edge (tt3) (tt1) edge (tt2) edge (tt3)  (tt2) edge (tt3);
\end{tikzpicture}
\caption{An example of a graph $\beth$ satisfying Condition \ref{cond:beth}. The edges 
are ordered lexicographically} \label{fig:beth}
\end{figure}

We denote by $\Home$ the set of isomorphism classes of such graphs $\beth$.
For every $z\in \Home$, we denote by $\beth_z$ once and for all chosen 
representative of $z$. Finally, we denote by 
\begin{equation}
\label{GCo-beth}
\QCo_{\beth}
\end{equation}
the subspace of $\QCo$ spanned by graphs $\G$ for which 
$\G^{\red}$ is isomorphic to $\beth$. 

Since the differential $\pa_{\Gr}$ of the associated graded complex can 
only raise the number of separating edges, the subspace 
$\QCo_{\beth}$ is closed with respect to $\pa_{\Gr}$.
Thus the associated graded complex  $\Gr_{\cF} \QCo$ of $\QCo$  
splits into the direct sum 
\begin{equation}
\label{Gr-QCo}
\Gr_{\cF} \QCo  \cong  \bigoplus_{z\in \Home} \QCo_{\beth_z}\,.
\end{equation}

We claim that 
\begin{prop}
\label{prop:QC-beth-acyclic}
For every $z \in \Home$ the cochain complex 
$$
(\QCo_{\beth_z}, \pa_{\Gr})
$$
is acyclic. 
\end{prop}

This proposition is proved in Section \ref{sec:QC-beth-acyclic} below. 
Let us now use it to prove that the quotient complex  $\QCo$ is acyclic. 

Indeed, it is easy to see that 
$$
\QCo ~ = ~ \bigcup_m  \cF_m \QCo.
$$
Moreover, $\cF_m (\QCo)^d = \bfzero$ for every $m < -d$. 

Thus, applying \cite[Claim A.2]{notes} and Proposition \ref{prop:QC-beth-acyclic}, we deduce that
$\QCo$ is acyclic. 

This concludes a proof of the second part of Proposition \ref{prop:GCo-dfGCo3} modulo 
Proposition \ref{prop:QC-beth-acyclic}.

\subsection{Assembling a connected graph from a ``separating'' forest and islands}
\label{sec:disassem-G}

Let us consider a connected graph $\G$ with at least one cut vertex (and 
with all vertices having valency $\ge 3$). It is clear that all separating edges and 
cut vertices of $\G$ form a forest $F$ and we call $F$ {\it the separating forest} of $\G$.   
We will now explain how any connected 
graph $\G$ with at least one cut vertex can be reconstructed from its 
separating forest and 1-vertex irreducible graphs which we call {\it islands.} 
For this purpose, we need a more precise terminology. 

An {\it island} is a $1$-vertex irreducible graph $\Ups$ with a non-zero number of 
marked vertices.  Each marked vertex of $\Ups$ must have valency $\ge 2$
and all the remaining vertices of $\Ups$ must have valency $\ge 3$.  

Let $S= \{ c_1, c_2, \dots, c_m\}$ be an auxiliary set with $\ge 2$ elements. 
An {\it $S$-decorated forest} $F$  is a forest with a possibly empty set of {\it internal} 
vertices of valencies $\ge 3$, 
a non-empty set of {\it external} vertices $V_{ext}(F)$ and
a \und{surjective} map 
\begin{equation}
\label{mmp}
\mmp : S  \to V_{ext}(F)
\end{equation}    
satisfying these properties: 
\begin{itemize}

\item {\it each connected component of $F$ has at least one external vertex;}

\item {\it if a connected component of $F$ has exactly one external vertex $v$
then $\mmp^{-1}(v)$ has at least $2$ elements.}

\end{itemize}
Note that, since each internal vertex has valency $\ge 3$, all univalent and bivalent vertices 
of any $S$-decorated forest $F$ are external. So if a connected component of $F$ has 
exactly one external vertex $v$ then this connected component consists of only one vertex 
and it does not have any edges.

Figure \ref{fig:dec-forest} shows an example of a $\{c_1, \dots, c_{7}\}$-decorated forest $F$ 
and figure \ref{fig:islands} shows a collection of islands with  $\{c_1, \dots, c_{7}\}$ being 
the set of its marked vertices. In pictures, internal vertices of forests are depicted by gray bullets, 
edges of forests are depicted by dashed lines, and external vertices are depicted by white vertices with 
inscribed pre-images of $\mmp$.
Using this forest $F$ and the collection of islands from  figure \ref{fig:islands}, we assemble 
the graph shown in figure \ref{fig:assem} via merging every external vertex  $v \in V_{ext}(F)$ with 
the vertices of the islands marked by elements in $\mmp^{-1}(v)$. For example, the bivalent vertex of 
$F$ is merged with the marked vertex $c_1$ of the tetrahedron and the marked vertex $c_4$ of the 
``square''. 
\begin{figure}[htp] 
\centering 
\begin{minipage}[t]{0.4\linewidth}
\centering 
\begin{tikzpicture}[scale=0.5, >=stealth']
\tikzstyle{ext}=[draw, circle,  minimum size=3, inner sep=2]
\tikzstyle{b}=[circle, draw, fill, minimum size=4, inner sep=1]
\tikzstyle{gr}=[circle, draw, fill, color =gray, minimum size=4, inner sep=1]
\node [ext] (c7) at (0,0) {$c_7$};
\node [ext] (c1c4) at (3,0) {$c_1, c_4$};
\node [gr] (b1) at (5,0) {};
\node [ext] (c5) at (6.5,1) {$c_5$};
\node [ext] (c2) at (6.5,-1) {$c_2$};
\node [ext] (c3c6) at (3,-3) {$c_3, c_6$};
\draw [dashed] (c7) edge (c1c4) (c1c4) edge (b1)
(b1) edge (c5) edge (c2);
\end{tikzpicture}
\caption{A $\{c_1, \dots, c_7\}$-decorated forest $F$. 
The only internal vertex is shown as the gray bullet} \label{fig:dec-forest}
\end{minipage}
\hspace{0.2cm}
\begin{minipage}[t]{0.4\linewidth}
\centering 
\begin{tikzpicture}[scale=0.5, >=stealth']
\tikzstyle{ext}=[draw, circle,  minimum size=3, inner sep=2]
\tikzstyle{b}=[circle, draw, fill, minimum size=4, inner sep=1]
\begin{scope}[shift={(-1,-4)}]
\node [ext] (c7) at (-3.5,2) {\small $c_7$};
\draw (c7) ..controls (-5,1) and (-5,3) .. (c7);
\node [ext] (c3) at (-3.5,-1) {\small $c_3$};
\draw (c3) ..controls (-5,-2) and (-5,0) .. (c3);
\end{scope}
\node [b] (b1) at (0,0) {};
\node [b] (b2) at (-1,1) {};
\node [b] (b3) at (-2,0) {};
\node [ext] (c1) at (-1,-1) {\small $c_1$};
\draw  (b1) edge (b2) edge (b3) edge (c1) (b2) edge (b3) edge (c1) (b3) edge (c1);
\begin{scope}[shift={(4,0)}]
\node [b] (b1) at (0,0) {};
\node [b] (b2) at (-1,1) {};
\node [b] (b3) at (-2,0) {};
\node [ext] (c5) at (-1,-1) {\small $c_5$};
\draw  (b1) edge (b2) edge (b3) edge (c5) (b2) edge (b3) edge (c5) (b3) edge (c5);
\end{scope}
\begin{scope}[shift={(4,-5)}]
\node [b] (p1) at (0,0) {};
\node [b] (p2) at (1.5,0) {};
\node [b] (p3) at (0.46,1.43) {};
\node [ext] (p4) at (-1.21,0.88) {\small $c_2$};
\node [b] (p5) at (-1.21,-0.88) {};
\node [b] (p6) at (0.46,-1.43) {};
\draw (p1) edge (p2) edge (p3) edge (p4) edge (p5) edge (p6)
(p2) edge (p3) (p3) edge (p4) (p4) edge (p5) (p5) edge (p6) (p6) edge (p2); 
\end{scope}
\begin{scope}[shift={(-1,-5)}]
\node [ext] (c4) at (0,1.5) {\small $c_4$};
\node [b] (s1) at (1.5,0) {};
\node [b] (s2) at (0,-1.5) {};
\node [b] (s0) at (0,0) {};
\node [ext] (c6) at (-1.5,0) {\small $c_6$};
\draw (s0) edge (s1) edge (s2) edge (c4) edge (c6) (c4) edge (s1) edge (c6) (s2) edge (s1) edge (c6);
\end{scope}
\end{tikzpicture}
\caption{A collection of islands} \label{fig:islands}
\end{minipage}
\end{figure}

%
%
\begin{figure}[htp] 
\centering 
\begin{tikzpicture}[scale=0.5, >=stealth']
\tikzstyle{ext}=[draw, circle,  minimum size=3, inner sep=2]
\tikzstyle{b}=[circle, draw, fill, minimum size=4, inner sep=1]
\tikzstyle{gr}=[circle, draw, fill, color =gray, minimum size=4, inner sep=1]
\node [gr] (c7) at (0,0) {};
\node [gr] (c1c4) at (2,0) {};
\node [gr] (gr1) at (5,0) {};
\node [gr] (c5) at (6.5,1) {};
\node [gr] (c2) at (6.5,-1) {};
\draw [dashed] (c7) edge (c1c4) (c1c4) edge (gr1)
(gr1) edge (c5) edge (c2);
\draw (c7) ..controls (-1,-1) and (-1,1) .. (c7);
\node [b] (b1) at (2,2) {};
\node [b] (b2) at (3,1) {};
\node [b] (b3) at (1,1) {};
\draw (b1) edge (b2) edge (b3) edge (c1c4) (b2) edge (b3) edge (c1c4)  (b3) edge (c1c4);
\node [b] (b4) at (2,-3) {};
\node [b] (b5) at (3.5,-1.5) {};
\node [gr] (c3c6) at (0.5,-1.5) {};
\node [b] (b0) at (2,-1.5) {};
\draw (b0) edge (b4) edge (b5) edge (c3c6) edge (c1c4) (c1c4) edge (b5) edge (c3c6) (b4) edge (b5) edge (c3c6);
\draw (c3c6) ..controls (-0.5,-2.5) and (-0.5,-0.5) .. (c3c6);
\node [b] (b7) at (8,1) {};
\node [b] (b8) at (8,2.5) {};
\node [b] (b9) at (6.5,2.5) {};
\draw (c5) edge (b7) edge (b8) edge (b9) (b7) edge (b8) edge (b9) (b8) edge (b9);
\begin{scope}[shift={(7.71,-1.88)}]
\node [b] (p1) at (0,0) {};
\node [b] (p2) at (1.5,0) {};
\node [b] (p3) at (0.46,1.43) {};
\node [b] (p5) at (-1.21,-0.88) {};
\node [b] (p6) at (0.46,-1.43) {};
\draw (p1) edge (p2) edge (p3) edge (c2) edge (p5) edge (p6)
(p2) edge (p3) (p3) edge (c2) (c2) edge (p5) (p5) edge (p6) (p6) edge (p2); 
\end{scope}
\end{tikzpicture}
\caption{The graph assembled from the forest in figure \ref{fig:dec-forest}
and the islands in figure \ref{fig:islands}. The cut vertices are shown as 
gray bullets and separating edges are shown as dashed lines} \label{fig:assem}
\end{figure}
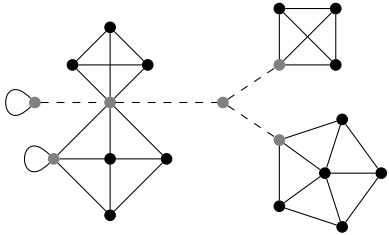

It is clear that every connected graph $\G$ with at least one cut vertex (and all vertices having valency $\ge 3$)
can be disassembled into a collection of islands (with at least 2 elements) and 
its separating forest decorated by the set of marked vertices of the islands. 

An example of this process is shown in figures \ref{fig:G}, \ref{fig:forest} and 
\ref{fig:G-disassem}. More precisely, we start with a graph $\G$ shown in 
figure \ref{fig:G}. The cut vertices of $\G$ are shown as 
gray bullets and separating edges of $\G$ are shown as dashed lines.

By detaching the separating forest from $\G$ we get a 
disconnected graph with some edges having ``free'' ends shown in 
figure \ref{fig:forest}.  To get the islands, we attach the ``free'' ends to appropriate 
marked vertices $c_1, c_2, \dots, c_7$ and decorate the separating forest by 
the set $\{c_1, c_2, \dots, c_7\}$. The resulting collection of islands 
and the  $\{c_1, c_2, \dots, c_7\}$-decorated forest are shown in figure \ref{fig:G-disassem}. 
\begin{figure}[htp]
\begin{minipage}[t]{0.45\linewidth}
\centering 
\begin{tikzpicture}[scale=0.5, >=stealth']
\tikzstyle{w}=[circle, draw, minimum size=4, inner sep=1]
\tikzstyle{b}=[circle, draw, fill, minimum size=4, inner sep=1]
\tikzstyle{gr}=[circle, draw, fill, color =gray, minimum size=4, inner sep=1]
\node [gr] (b1) at (0,0) {};
\node [b] (b2) at (-1,1) {};
\node [b] (b3) at (-2,0) {};
\node [b] (b4) at (-1,-1) {};
\draw  (b1) edge (b2) edge (b3) edge (b4) (b2) edge (b3) edge (b4) (b3) edge (b4);
\node [b] (b5) at (1,1) {};
\node [gr] (b6) at (2,0) {};
\node [b] (b7) at (1,-1) {};
\draw  (b1) edge (b5) edge (b7) (b6) edge (b5) edge (b7) (b5) edge (b7); 
\node [gr] (b8) at (3,0) {};
\node [gr] (b9) at (3,1) {};
\draw [dashed] (b8) edge (b6) edge (b9); 
\node [gr] (b11) at (3,2.5) {};
\node [b] (b12) at (4.5,2.5) {};
\node [b] (b13) at (4.5,1) {};
\node [gr] (b14) at (1.5,2.5) {};
\draw  (b9) edge (b11) edge (b12) edge (b13) (b11) edge (b12) edge (b13)  
(b12) edge (b13);
\draw [dashed] (b11) edge (b14);
\draw (b14) ..controls (0.5,1.5) and (0.5,3.5) .. (b14);
\begin{scope}[shift={(5,-2)}]
\node [b] (p1) at (0,0) {};
\node [b] (p2) at (1.5,0) {};
\node [b] (p3) at (0.46,1.43) {};
\node [gr] (p4) at (-1.21,0.88) {};
\node [b] (p5) at (-1.21,-0.88) {};
\node [b] (p6) at (0.46,-1.43) {};
\draw (p1) edge (p2) edge (p3) edge (p4) edge (p5) edge (p6)
(p2) edge (p3) (p3) edge (p4) (p4) edge (p5) (p5) edge (p6) (p6) edge (p2); 
\end{scope}
\draw [dashed] (b8) edge (p4);
\end{tikzpicture}
~\\[0.3cm]
\caption{A graph $\G$.  Cut vertices are shown as gray bullets and separating edges 
are shown as dashed lines} \label{fig:G}
\end{minipage}
\begin{minipage}[t]{0.45\linewidth}
\centering 
\begin{tikzpicture}[scale=0.5, >=stealth']
\tikzstyle{w}=[circle, draw, minimum size=4, inner sep=1]
\tikzstyle{b}=[circle, draw, fill, minimum size=4, inner sep=1]
\tikzstyle{empty}=[circle, minimum size=4, inner sep=3]
\tikzstyle{gr}=[circle, draw, fill, color =gray, minimum size=4, inner sep=1]
\begin{scope}[shift={(7,0)}]
\node [gr] (bb1) at (1,0) {};
\node [gr] (bb6) at (2,0) {};
\node [gr] (bb8) at (3,0) {};
\node [gr] (bb9) at (3,1) {};
\node [gr] (bb10) at (4,-1) {};
\draw [dashed] (bb8) edge (bb6) edge (bb9) edge (bb10); 
\node [gr] (bb11) at (3,2) {};
\node [gr] (bb14) at (1.5,2) {};
\draw [dashed] (bb11) edge (bb14); 
\end{scope}
\node [empty] (b1) at (0,0) {};
\node [b] (b2) at (-1,1) {};
\node [b] (b3) at (-2,0) {};
\node [b] (b4) at (-1,-1) {};
\draw  (b1) edge (b2) edge (b3) edge (b4) (b2) edge (b3) edge (b4) (b3) edge (b4);
\node [b] (b5) at (1,1) {};
\node [empty] (b6) at (2,0) {};
\node [b] (b7) at (1,-1) {};
\draw  (b1) edge (b5) edge (b7) (b6) edge (b5) edge (b7) (b5) edge (b7); 
\node [empty] (b9) at (3,1) {};
\node [empty] (b11) at (3,2.5) {};
\node [b] (b12) at (4.5,2.5) {};
\node [b] (b13) at (4.5,1) {};
\node [empty] (b14) at (1.5,2.5) {};
\draw  (b9) edge (b11) edge (b12) edge (b13) (b11) edge (b12) edge (b13)  
(b12) edge (b13);
\draw (b14) ..controls (0.5,1.5) and (0.5,3.5) .. (b14);
\begin{scope}[shift={(5,-2)}]
\node [b] (p1) at (0,0) {};
\node [b] (p2) at (1.5,0) {};
\node [b] (p3) at (0.46,1.43) {};
\node [empty] (p4) at (-1.21,0.88) {}; 
\node [b] (p5) at (-1.21,-0.88) {};
\node [b] (p6) at (0.46,-1.43) {};
\draw (p1) edge (p2) edge (p3) edge (p4) edge (p5) edge (p6)
(p2) edge (p3) (p3) edge (p4) (p4) edge (p5) (p5) edge (p6) (p6) edge (p2); 
\end{scope}
\end{tikzpicture}
~\\[0.3cm]
\caption{The forest is detached from $\G$} \label{fig:forest}
\end{minipage}
\begin{minipage}[t]{\linewidth}
\centering 
\begin{tikzpicture}[scale=0.5, >=stealth']
\tikzstyle{b}=[circle, draw, fill, minimum size=4, inner sep=1]
\tikzstyle{ext}=[draw, circle,  minimum size=3, inner sep=2]
\tikzstyle{empty}=[circle, minimum size=4, inner sep=3]
\tikzstyle{gr}=[circle, draw, fill, color =gray, minimum size=4, inner sep=1]
\begin{scope}[shift={(10,0)}]
\node [ext] (c1c2) at (0,0) {\small $c_1, c_2$};
\node [ext] (c3) at (2.5,0) {\small $c_3$};
\node [gr] (gr) at (4,0) {};
\node [ext] (c4) at (4,1.5) {\small $c_4$};
\node [ext] (c7) at (5,-1.5) {\small $c_7$};
\draw [dashed] (gr) edge (c3) edge (c4) edge (c7); 
\node [ext] (c6) at (0,2.5) {\small $c_6$};
\node [ext] (c5) at (2,2.5) {\small $c_5$};
\draw [dashed] (c6) edge (c5); 
\end{scope}
\node [ext] (b1) at (-2,0) {\small $c_1$};
\node [b] (b2) at (-3,1) {};
\node [b] (b3) at (-4,0) {};
\node [b] (b4) at (-3,-1) {};
\draw  (b1) edge (b2) edge (b3) edge (b4) (b2) edge (b3) edge (b4) (b3) edge (b4);
\node [ext] (bb1) at (0,0) {\small $c_2$};
\node [b] (b5) at (1,1) {};
\node [ext] (b6) at (2,0) {\small $c_3$};
\node [b] (b7) at (1,-1) {};
\draw  (bb1) edge (b5) edge (b7) (b6) edge (b5) edge (b7) (b5) edge (b7); 
\node [ext] (b9) at (4,1) {\small $c_4$};

\node [ext] (b11) at (4,2.5) {\small $c_5$};
\node [b] (b12) at (5.5,2.5) {};
\node [b] (b13) at (5.5,1) {};
\node [ext] (b14) at (1.5,2.5) {\small $c_6$};
\draw  (b9) edge (b11) edge (b12) edge (b13) (b11) edge (b12) edge (b13)  
(b12) edge (b13);
\draw (b14) ..controls (-0.5,1.5) and (-0.5,3.5) .. (b14);
\begin{scope}[shift={(5,-2)}]
\node [b] (p1) at (0,0) {};
\node [b] (p2) at (1.5,0) {};
\node [b] (p3) at (0.46,1.43) {};
\node [ext] (p4) at (-1.21,0.88) {\small $c_7$};
\node [b] (p5) at (-1.21,-0.88) {};
\node [b] (p6) at (0.46,-1.43) {};
\draw (p1) edge (p2) edge (p3) edge (p4) edge (p5) edge (p6)
(p2) edge (p3) (p3) edge (p4) (p4) edge (p5) (p5) edge (p6) (p6) edge (p2); 
\end{scope}
\end{tikzpicture}
~\\[0.3cm]
\caption{$\G$ is disassembled into its separating forest and the islands} \label{fig:G-disassem}
\end{minipage}
\end{figure} 

\subsection{The proof of Proposition \ref{prop:QC-beth-acyclic}}
\label{sec:QC-beth-acyclic}

Let $\beth$ be a graph satisfying Condition \ref{cond:beth} with $n_0$ non-cut vertices, $n_1$ cut vertices
and $r_0$ edges.  

Let us fix integers $r \ge 0$, $n \ge n_1$ and denote by 
$\gra(\beth)_{n_0 + n}^{r_0 + r}$ the set of connected graphs $\G \in \gra_{n_0 + n}^{r_0 + r}$ which 
satisfy the following conditions:
 
\begin{itemize}

\item every vertex of $\G$ has valency $\ge 3$, 

\item $\G$ has at least one cut vertex and $\G^{\red}$ is isomorphic to $\beth$, 

\item $\G$ has exactly $r$ separating edges and they are labeled by 
$\{\und{r_0+ 1}, \und{r_0 + 2}, \dots, \und{r_0 + r}\}$, 

\item $\G$ has exactly $n$ cut vertices and they are labeled by $\{\gray{n_0 +1, n_0 + 2, \dots, n_0 +n}\}$.

\end{itemize}
The last two conditions imply that $n_0$ non-cut vertices (resp. $r_0$ non-separating edges)
of $\G$ are labeled by  $\{\gray{1, 2, \dots, n_0}\}$ (resp. by $\{\und{1}, \und{2}, \dots, \und{r_0}\}$).

For example, the graph $\G$ shown in figure \ref{fig:G-beth} belongs to $\gra(\beth)^{22+1}_{11+2}$, 
where $\beth$ is depicted in figure \ref{fig:beth}.
\begin{figure}[htp]
\centering 
\begin{tikzpicture}[scale=1, >=stealth']
\tikzstyle{w}=[circle, draw, minimum size=3, inner sep=1]
\tikzstyle{b}=[circle, draw, fill, minimum size=3, inner sep=1]
\tikzstyle{gr}=[circle, draw, fill, color =gray, minimum size=4, inner sep=1]

\node [b] (b0) at (0,0) {};
\draw (0.3,-0.3) node[anchor=center, color =gray] {{\small $1$}};
\node [gr] (b1) at (4,0) {};
\draw (4,-0.4) node[anchor=center, color =gray] {{\small $12$}};
\node [b] (b2) at (0.62, 1.90) {};
\draw (0.62,2.15) node[anchor=center, color =gray] {{\small $2$}};
\node [b] (b3) at (-1.62, 1.18) {};
\draw (-1.82,1.18) node[anchor=center, color =gray] {{\small $3$}};
\node [b] (b4) at (-1.62, -1.18) {};
\draw (-1.82,-1.28) node[anchor=center, color =gray] {{\small $4$}};
\node [b] (b5) at (0.62, -1.90) {};
\draw (0.72,-2.1) node[anchor=center, color =gray] {{\small $5$}};
\draw (b1) edge (b2) (b2) edge (b3) 
(b3) edge (b4) (b4) edge (b5) (b5) edge (b1);
\draw (b0) edge (b1)  edge (b2) edge (b3) edge (b4)  edge (b5);
\node [b] (t1) at (7,1) {};
\draw (7,0.75) node[anchor=center, color =gray] {{\small $6$}};
\node [b] (t2) at (4,3) {};
\draw (3.8,3) node[anchor=center, color =gray] {{\small $8$}};
\node [b] (t3) at (7,3.5) {};
\draw (7,3.75) node[anchor=center, color =gray] {{\small $7$}};
\draw (b1) edge (t1) edge (t2) edge (t3) (t1) edge (t2) edge (t3)  (t2) edge (t3);
\node [gr] (tt0) at (5,-1) {};
\draw (5.1,-0.7) node[anchor=center, color =gray] {{\small $13$}};
\draw [dashed] (b1) edge (tt0);
\begin{scope}[shift={(1,-1)}]
\node [b] (tt1) at (7,-1) {};
\draw (7.3,-1) node[anchor=center, color =gray] {{\small $11$}};

\node [b] (tt2) at (4,-3) {};
\draw (4,-3.25) node[anchor=center, color =gray] {{\small $9$}};

\node [b] (tt3) at (7,-3.5) {};
\draw (7.1,-3.7) node[anchor=center, color =gray] {{\small $10$}};

\draw (tt0) edge (tt1) edge (tt2) edge (tt3) (tt1) edge (tt2) edge (tt3)  (tt2) edge (tt3);
\end{scope}
\end{tikzpicture}
\caption{An example of a graph $\G \in \gra(\beth)^{22+1}_{11+2}$. 
The (only separating) edge connecting $\gray{12}$ to  $\gray{13}$ is labeled by $\und{23}$. 
The labels $\und{1}, \und{2}, \dots, \und{22}$ for non-separating edges are not shown} 
\label{fig:G-beth}
\end{figure}

The group $\bbS_n \times \bbS_r$ acts on $\gra(\beth)_{n_0 + n}^{r_0 + r}$ in the obvious way by rearranging the labels 
$$
\gray{n_0 +1}, \gray{n_0 +2},\dots, \gray{n_0 + n} \quad 
\textrm{ and } \quad \und{r_0 +1},  \und{r_0 +2}, \dots, \und{r_0 + r}\,.
$$
Using this action, we define the graded vector space 
\begin{equation}
\label{C-beth}
C_{\beth} ~ : =~ \bigoplus_{n \ge n_1, ~r \ge 0} ~  \big(\, \bs^{2 (n_0  + n) - 2 - (r_0 +r) } \,  
\span_{\bbK}\big(  \gra(\beth)_{n_0 + n}^{r_0 + r}  \big) \otimes  \sgn_{r}  \,\big)_{\bbS_n \times \bbS_r } \,.
\end{equation}

Let us denote by $\de$ the linear map 
$$
\de : \span_{\bbK}\big(  \gra(\beth)_{n_0 + n}^{r_0 + r} \big) \to  \span_{\bbK}\big(  \gra(\beth)_{n_0 + n+1}^{r_0 + r+1} \big)
$$
defined by the formula
\begin{equation}
\label{delta-dfn}
\de(\G) =  -(-1)^{|\G|} \, \sum_{i = n_0+1}^{n_0+n} \de_i(\G),
\end{equation}
where $\G \in  \gra(\beth)_{n_0 + n}^{r_0 + r}$ and $\de_i(\G)$ is obtained from $\G \circ_i \G_{\ed}$ by 
retaining only graphs with vertices of valencies $\ge 3$ for which the additional edge is separating.  

It is easy to see that $\de$ descends to a linear map of degree $1$ on $C_{\beth}$. 
For example, if $\G$ is the graph shown in figure \ref{fig:G-beth} then 
$$
\de(\G) = \de_{12}(\G)
$$
because $\G \circ_{13} \G_{\ed}$ involves only ``unwanted'' graphs. 
Moreover, $\de_{12}(\G)$ is the sum of the graphs depicted in 
figures \ref{fig:deG-beth} and \ref{fig:deG-beth1}.
\begin{figure}[htp]
\centering 
\begin{tikzpicture}[scale=1, >=stealth']
\tikzstyle{w}=[circle, draw, minimum size=3, inner sep=1]
\tikzstyle{b}=[circle, draw, fill, minimum size=3, inner sep=1]
\tikzstyle{gr}=[circle, draw, fill, color =gray, minimum size=4, inner sep=1]
\node [b] (b0) at (0,0) {};
\draw (0.3,-0.3) node[anchor=center, color =gray] {{\small $1$}};
\node [gr] (b1) at (4,0) {};
\draw (4,-0.4) node[anchor=center, color =gray] {{\small $12$}};
\node [b] (b2) at (0.62, 1.90) {};
\draw (0.62,2.15) node[anchor=center, color =gray] {{\small $2$}};
\node [b] (b3) at (-1.62, 1.18) {};
\draw (-1.82,1.18) node[anchor=center, color =gray] {{\small $3$}};
\node [b] (b4) at (-1.62, -1.18) {};
\draw (-1.82,-1.28) node[anchor=center, color =gray] {{\small $4$}};
\node [b] (b5) at (0.62, -1.90) {};
\draw (0.72,-2.1) node[anchor=center, color =gray] {{\small $5$}};
\draw (b1) edge (b2) (b2) edge (b3) 
(b3) edge (b4) (b4) edge (b5) (b5) edge (b1);
\draw (b0) edge (b1)  edge (b2) edge (b3) edge (b4)  edge (b5);
\begin{scope}[shift={(1,1)}]
\node [gr] (t0) at (4,0) {};
\draw (4.2,-0.2) node[anchor=center, color =gray] {{\small $13$}};
\node [b] (t1) at (7,1) {};
\draw (7,0.75) node[anchor=center, color =gray] {{\small $6$}};
\node [b] (t2) at (4,3) {};
\draw (3.8,3) node[anchor=center, color =gray] {{\small $8$}};
\node [b] (t3) at (7,3.5) {};
\draw (7,3.75) node[anchor=center, color =gray] {{\small $7$}};
\draw (t0) edge (t1) edge (t2) edge (t3) (t1) edge (t2) edge (t3)  (t2) edge (t3);
\end{scope}
\draw [dashed] (b1) edge (t0);
\node [gr] (tt0) at (5,-1) {};
\draw (5.1,-0.7) node[anchor=center, color =gray] {{\small $14$}};
\draw [dashed] (b1) edge (tt0);
\begin{scope}[shift={(1,-1)}]
\node [b] (tt1) at (7,-1) {};
\draw (7.3,-1) node[anchor=center, color =gray] {{\small $11$}};
\node [b] (tt2) at (4,-3) {};
\draw (4,-3.25) node[anchor=center, color =gray] {{\small $9$}};
\node [b] (tt3) at (7,-3.5) {};
\draw (7.1,-3.7) node[anchor=center, color =gray] {{\small $10$}};
\draw (tt0) edge (tt1) edge (tt2) edge (tt3) (tt1) edge (tt2) edge (tt3)  (tt2) edge (tt3);
\end{scope}
\end{tikzpicture}
\caption{The edge $(\gray{12}, \gray{14})$ is labeled by $\und{23}$
and the edge  $(\gray{12}, \gray{13})$ is labeled by $\und{24}$ 
The labels $\und{1}, \und{2}, \dots, \und{22}$ for non-separating edges are not shown} 
\label{fig:deG-beth}
\end{figure}
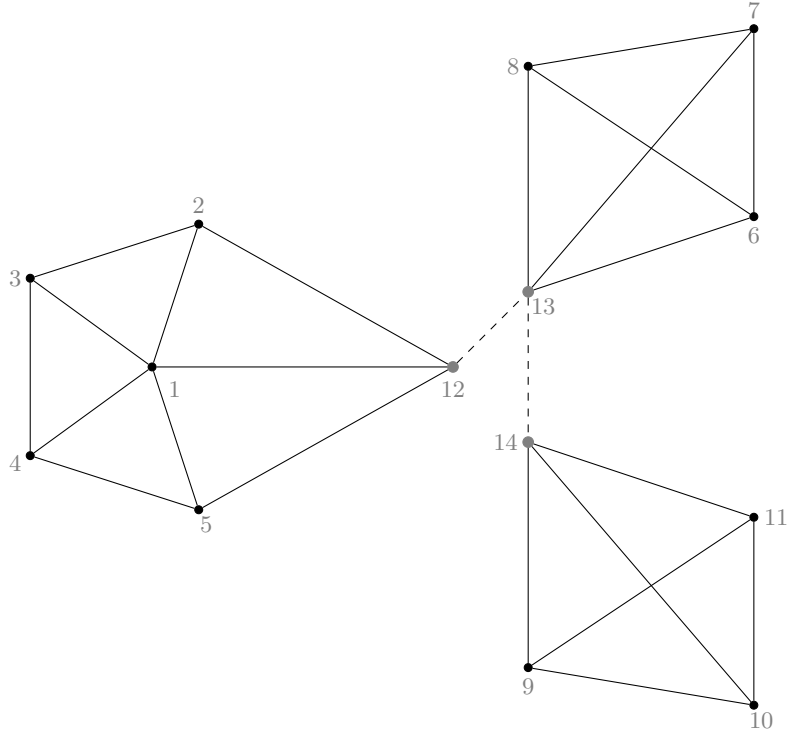
%
%
\begin{figure}[htp]
\centering 
\begin{tikzpicture}[scale=1, >=stealth']
\tikzstyle{w}=[circle, draw, minimum size=3, inner sep=1]
\tikzstyle{b}=[circle, draw, fill, minimum size=3, inner sep=1]
\tikzstyle{gr}=[circle, draw, fill, color =gray, minimum size=4, inner sep=1]
\node [b] (b0) at (0,0) {};
\draw (0.3,-0.3) node[anchor=center, color =gray] {{\small $1$}};
\node [gr] (b1) at (4,0) {};
\draw (4,-0.3) node[anchor=center, color =gray] {{\small $12$}};
\node [b] (b2) at (0.62, 1.90) {};
\draw (0.62,2.15) node[anchor=center, color =gray] {{\small $2$}};
\node [b] (b3) at (-1.62, 1.18) {};
\draw (-1.82,1.18) node[anchor=center, color =gray] {{\small $3$}};
\node [b] (b4) at (-1.62, -1.18) {};
\draw (-1.82,-1.28) node[anchor=center, color =gray] {{\small $4$}};
\node [b] (b5) at (0.62, -1.90) {};
\draw (0.72,-2.1) node[anchor=center, color =gray] {{\small $5$}};
\draw (b1) edge (b2) (b2) edge (b3) 
(b3) edge (b4) (b4) edge (b5) (b5) edge (b1);
\draw (b0) edge (b1)  edge (b2) edge (b3) edge (b4)  edge (b5);
\begin{scope}[shift={(1,1)}]
\node [gr] (t0) at (4,0) {};
\draw (4.2,-0.2) node[anchor=center, color =gray] {{\small $13$}};
\node [b] (t1) at (7,1) {};
\draw (7,0.75) node[anchor=center, color =gray] {{\small $6$}};
\node [b] (t2) at (4,3) {};
\draw (3.8,3) node[anchor=center, color =gray] {{\small $8$}};
\node [b] (t3) at (7,3.5) {};
\draw (7,3.75) node[anchor=center, color =gray] {{\small $7$}};
\draw (t0) edge (t1) edge (t2) edge (t3) (t1) edge (t2) edge (t3)  (t2) edge (t3);
\end{scope}
\draw [dashed] (b1) edge (t0);
\node [gr] (tt0) at (5,-1) {};
\draw (4.7,-1) node[anchor=center, color =gray] {{\small $14$}};
\draw [dashed] (t0) edge (tt0);
\begin{scope}[shift={(1,-1)}]
\node [b] (tt1) at (7,-1) {};
\draw (7.3,-1) node[anchor=center, color =gray] {{\small $11$}};
\node [b] (tt2) at (4,-3) {};
\draw (4,-3.25) node[anchor=center, color =gray] {{\small $9$}};
\node [b] (tt3) at (7,-3.5) {};
\draw (7.1,-3.7) node[anchor=center, color =gray] {{\small $10$}};
\draw (tt0) edge (tt1) edge (tt2) edge (tt3) (tt1) edge (tt2) edge (tt3)  (tt2) edge (tt3);
\end{scope}
\end{tikzpicture}
\caption{The edge $(\gray{13}, \gray{14})$ is labeled by $\und{23}$
and the edge  $(\gray{12}, \gray{13})$ is labeled by $\und{24}$. 
The labels $\und{1}, \und{2}, \dots, \und{22}$ for non-separating edges are not shown} 
\label{fig:deG-beth1}
\end{figure}
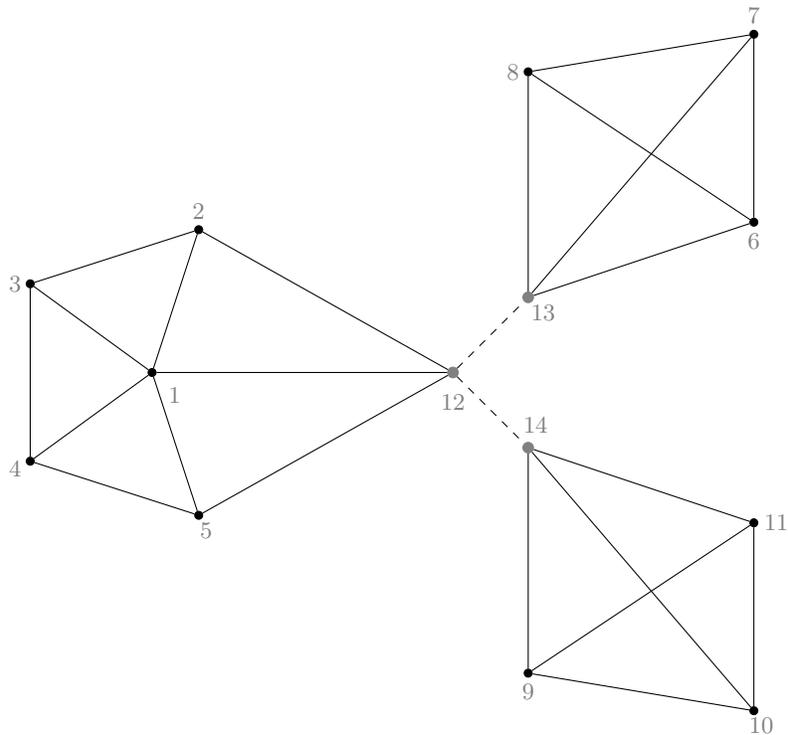

Let us prove that
\begin{claim}
\label{cl:delta2is0}
The map $\de : C_{\beth} \to C_{\beth}$ satisfies the identity
$$
\de^2 = 0. 
$$
\end{claim}
\begin{proof}
Let $\G \in \gra(\beth)^{r_0 + r}_{n_0 + n}$.

It is easy to see that the terms coming from linear combinations $(\G \circ_i \G_{\ed}) \circ_j \G_{\ed}$  for $j \neq i$ and
$j \neq i+1$ appear twice with opposite signs. So they cancel each other and it remains to 
consider contributions coming from $(\G \circ_i \G_{\ed}) \circ_{i} \G_{\ed}$ and $(\G \circ_i \G_{\ed}) \circ_{i+1} \G_{\ed}$.

Let us denote by $X = \{x_1, \dots, x_m\}$ the union of $\mmp^{-1}(i)$ and $\{\und{k_1}, \und{k_2}, \dots\}$, 
where $\und{k_1}, \und{k_2}, \dots$  are all the separating edges (if any) incident to $i$ in $\G$.
It is clear that, if $m \le 2$, then we do not have contributions coming from  
$(\G \circ_i \G_{\ed}) \circ_{i} \G_{\ed}$ and $(\G \circ_i \G_{\ed}) \circ_{i+1} \G_{\ed}$. 

If $m \ge 3$, then for every partition of $X$ into three (non-empty) subsets  $X_1$, $X_2$, $X_3$,  
we get $6$ terms coming 
from $(\G \circ_i \G_{\ed}) \circ_{i} \G_{\ed}$ and $(\G \circ_i \G_{\ed}) \circ_{i+1} \G_{\ed}$. 
These terms are all shown schematically in figure \ref{fig:six}. It is easy to see that the terms 
in the two columns cancel each other when we pass to coinvariants. 
\begin{figure}[htp]
\centering 
\begin{tikzpicture}[scale=0.8, >=stealth']
\tikzstyle{w}=[circle, draw, minimum size=3, inner sep=1.8]
\tikzstyle{b}=[circle, draw, fill, minimum size=3, inner sep=1]
\tikzstyle{gr}=[circle, draw, fill, color =gray, minimum size=4, inner sep=1]
%
%
\node [w] (v1) at (0,0) {{\small $X_1$}};
\node [w] (v2) at (-2,2.3) {{\small $X_2$}};
\node [w] (v3) at (2,2.3) {{\small $X_3$}};
\draw [dashed] (v2) edge (v1) edge (v3);
\draw (-2,1) node[anchor=center] {{\small $\und{r_0+r+1}$}};
\draw (0,2.8) node[anchor=center] {{\small $\und{r_0+r+2}$}};
\begin{scope}[shift={(0,-4.5)}]
\node [w] (v1) at (0,0) {{\small $X_1$}};
\node [w] (v2) at (-2,2.3) {{\small $X_2$}};
\node [w] (v3) at (2,2.3) {{\small $X_3$}};
\draw [dashed] (v3) edge (v1) edge (v2);
\draw (2,1) node[anchor=center] {{\small $\und{r_0+r+1}$}};
\draw (0,2.8) node[anchor=center] {{\small $\und{r_0+r+2}$}};
\end{scope}
\begin{scope}[shift={(0,-9)}]
\node [w] (v1) at (0,0) {{\small $X_1$}};
\node [w] (v2) at (-2,2.3) {{\small $X_2$}};
\node [w] (v3) at (2,2.3) {{\small $X_3$}};
\draw [dashed] (v1) edge (v2) edge (v3);
\draw (-2,1) node[anchor=center] {{\small $\und{r_0+r+1}$}};
\draw (2,1) node[anchor=center] {{\small $\und{r_0+r+2}$}};
\end{scope}
%
%
\begin{scope}[shift={(8,0)}]
\node [w] (v1) at (0,0) {{\small $X_1$}};
\node [w] (v2) at (-2,2.3) {{\small $X_2$}};
\node [w] (v3) at (2,2.3) {{\small $X_3$}};
\draw [dashed] (v2) edge (v1) edge (v3);
\draw (-2,1) node[anchor=center] {{\small $\und{r_0+r+2}$}};
\draw (0,2.8) node[anchor=center] {{\small $\und{r_0+r+1}$}};
\end{scope}
\begin{scope}[shift={(8,-4.5)}]
\node [w] (v1) at (0,0) {{\small $X_1$}};
\node [w] (v2) at (-2,2.3) {{\small $X_2$}};
\node [w] (v3) at (2,2.3) {{\small $X_3$}};
\draw [dashed] (v3) edge (v1) edge (v2);
\draw (2,1) node[anchor=center] {{\small $\und{r_0+r+2}$}};
\draw (0,2.8) node[anchor=center] {{\small $\und{r_0+r+1}$}};
\end{scope}
\begin{scope}[shift={(8,-9)}]
\node [w] (v1) at (0,0) {{\small $X_1$}};
\node [w] (v2) at (-2,2.3) {{\small $X_2$}};
\node [w] (v3) at (2,2.3) {{\small $X_3$}};
\draw [dashed] (v1) edge (v2) edge (v3);
\draw (-2,1) node[anchor=center] {{\small $\und{r_0+r+2}$}};
\draw (2,1) node[anchor=center] {{\small $\und{r_0+r+1}$}};
\end{scope}
\end{tikzpicture}
\caption{The decoration of a cut vertex $v$ by $X_t$ means that 
$X_t$ is the union of $\mmp^{-1}(v)$ and the set of {\bf non-separating}
edges incident to $v$} 
\label{fig:six}
\end{figure}
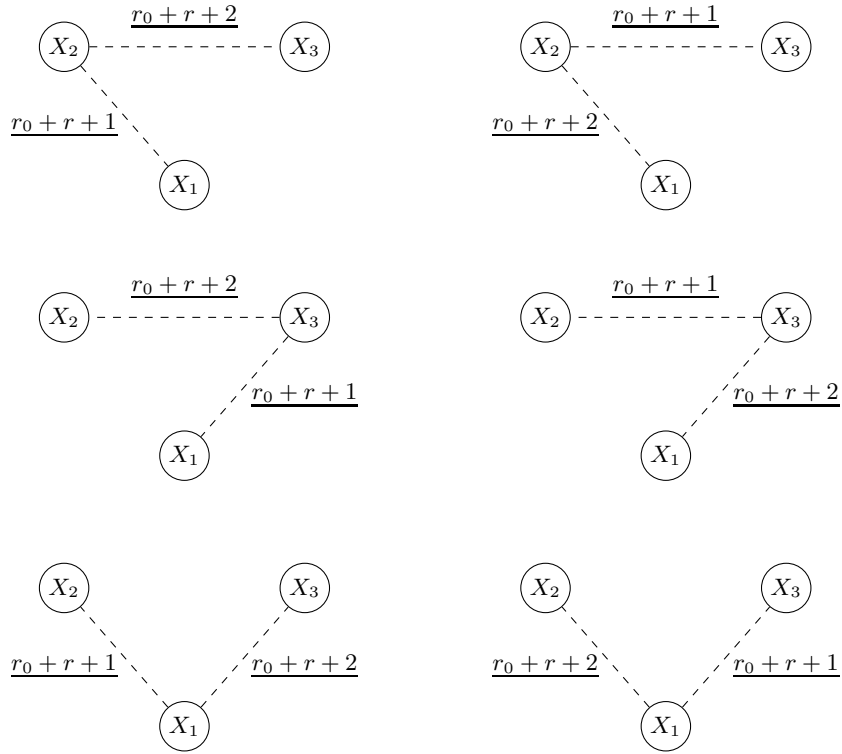

Thus the claim is proved. 
\end{proof}

Let $\beth$ be the graph satisfying Condition \ref{cond:beth} and let $n_0$ 
(resp. $r_0$) be the number of non-cut vertices (resp. edges) of $\beth$. 
The group $\bbS_{n_0} \times \bbS_{r_0}$ acts on 
the graded vector space $C_{\beth}$ by rearranging labels on 
non-cut vertices and non-separating edges. Moreover, the 
differential $\de$ commutes with this action. 

Let us now observe that both the space of coinvariants 
\begin{equation}
\label{C-beth-coinv}
\big( C_{\beth} \big)_{\bbS_{n_0} \times \bbS_{r_0}}
\end{equation}
and the space $\QCo_{\beth}$ from \eqref{GCo-beth} is the span of 
isomorphism classes of even connected graphs $\G$ satisfying 
these conditions 
\begin{itemize}

\item every vertex of $\G$ has valency $\ge 3$, 

\item $\G$ has at least one cut vertex, and 

\item $\G^{\red}$ is isomorphic to $\beth$. 

\end{itemize}
Moreover, such $\G$ has the same degree in \eqref{C-beth-coinv} and in $\QCo_{\beth}$.

Therefore the space of coinvariants \eqref{C-beth-coinv} is isomorphic to $\QCo_{\beth}$ as the 
graded vector space. 

It is not hard to see that the differential induced on \eqref{C-beth-coinv} by $\de$ corresponds 
to the differential $\pa_{\Gr}$ on $\QCo_{\beth}$. 
Thus, since the base field has characteristic zero and the group $\bbS_{n_0} \times \bbS_{r_0}$
is finite, Proposition \ref{prop:QC-beth-acyclic} follows directly from 
\begin{prop}
\label{prop:C-beth}
For every graph $\beth$ satisfying Condition \ref{cond:beth},  
the cochain complex $(C_{\beth}, \de)$ is acyclic. 
\end{prop}
\begin{proof}
Let, as above, $n_0$ (resp. $r_0$) be the number of non-cut vertices (resp. non-separating edges) of $\beth$
and $n_1$ be the number of cut vertices of $\beth$.  
Clearly, the separating forest of every 
graph $\G \in \gra(\beth)^{r_0 + r}_{n_0 + n}$ (for some $r \ge 0$ and $n \ge n_1$) has 
exactly $n_1$ connected components. Every connected component of the separating forest 
of $\G$ is incident to $\ge 2$ islands. These islands may have common (necessarily cut) vertices 
but they do not share edges.

Let us order the islands of $\beth$ according to this rule: 
$$
\Ups_1 < \Ups_2 ~ \Leftrightarrow ~
\textrm{the smallest edge of } \Ups_1 \textrm{ precedes the smallest edge of } \Ups_2\,.
$$

Let $\Ups_0$ be the smallest island of $\beth$. If the island $\Ups_0$ has exactly one marked vertex then 
we call this marked vertex {\it special} and denote it by $c_0$. 

Let us now consider the case when $\Ups_0$ has several marked vertices.
For a marked vertex $c$ of $\Ups_0$, we denote  
by $X(c)$ the set of islands of $\beth$ which are different from $\Ups_0$ and 
attached to $c$ in $\beth$. Let us also denote by $\Ups(c)$ the smallest island 
from the set $X(c)$.
Clearly, for two distinct marked vertices $c_1$ and $c_2$ of $\Ups_0$ the sets 
$X(c_1)$ and $X(c_2)$ are disjoint (otherwise, the vertex of $\beth$ corresponding to $c_1$ 
is not a cut vertex). So $\Ups(c_1) \neq \Ups(c_2)$ if $c_1 \neq c_2$.

We call the marked vertex $c_0$ of $\Ups_0$ {\it special} if $\Ups(c_0)$ is the smallest island of the set 
$$
\{\Ups(c) ~|~ c ~ \textrm{is a marked vertex of } \Ups_0 \}. 
$$

Thus we conclude that, for every graph $\beth$ satisfying Condition \ref{cond:beth},  
the (unique) smallest island $\Ups_0$ of $\beth$ has exactly one special marked vertex. 
For example, if $\beth$ is the graph shown in figure \ref{fig:beth} (with the lexicographic order 
on the set of edges), then the pentagon is the smallest island and the unique marked vertex 
of this island is special. 

Let $c_0$ be the special marked vertex of the smallest island of $\beth$, 
$\G \in \gra(\beth)^{r_0 + r}_{n_0 + n}$, and $F$ be the separating forest of 
$\G$ decorated by marked vertices of the islands of $\beth$. If 
\begin{equation}
\label{c0-alone}
\mmp^{-1} (\mmp(c_0)) = \{ c_0 \}
\end{equation}
and the vertex $\mmp(c_0)$ has valency $1$ in $F$ then 
we denote by $\ti{\G}$ the graph in 
$\gra(\beth)^{r_0 + r-1}_{n_0 + n-1}$ which is obtained by contracting the 
(only) separating edge incident to $\mmp(c_0)$. If this separating edge is labeled by $\und{j}$
and it connects the (cut) vertices with labels $i_1 < i_2$ then we 
label vertices and edges of $\ti{\G}$ in the following way: 
\begin{itemize}

\item we shift the labels  $\und{j+1}, \dots, \und{r_0 + r}$ down by $1$,

\item the vertex obtained by contracting edge $\und{j}$ is labeled by $\gray{i_1}$, 

\item we shift the labels $\gray{i_2 +1}, \dots, \gray{n_0 + n}$ down by $1$. 

\end{itemize}

We set
\begin{equation}
\label{h-dfn}
h(\G)  : = (-1)^{j} \ti{\G},
\end{equation}
if condition \eqref{c0-alone} is satisfied and $\mmp(c_0)$ has valency $1$ in $F$. 
Otherwise, we set 
\begin{equation}
\label{h-dfn-0}
h(\G) : = 0. 
\end{equation}
It is easy to see that equations \eqref{h-dfn} and \eqref{h-dfn-0} define a degree $-1$ linear map 
\begin{equation}
\label{homotopy}
h : C_{\beth} \to C_{\beth}\,.
\end{equation}

We will conclude the proof of the proposition by showing that 
\begin{equation}
\label{de-h-h-de}
\de \circ h + h \circ \de = \id_{C_{\beth}}\,.
\end{equation}

The set $\gra(\beth)^{r_0 + r}_{n_0 + n}$ splits into the disjoint union of two subsets: 
\begin{itemize}

\item The first subset consists of graphs for which condition \eqref{c0-alone} is satisfied 
and the vertex $\mmp(c_0)$ is incident to exactly one separating edge. 

\item The second subset consists of graphs for which condition \eqref{c0-alone} is not satisfied or 
it is satisfied but the vertex  $\mmp(c_0)$ in the separating forest of $\G$ is not univalent. 

\end{itemize}
For example the graphs shown in figures \ref{fig:beth}, \ref{fig:G-beth} and \ref{fig:deG-beth}
belongs to the second subset while the graph shown in figure \ref{fig:deG-beth1} belongs to the 
first subset.

Let us assume that $\G$ belongs to the first subset of $\gra(\beth)^{r_0 + r}_{n_0 + n}$.
It is easy to see that $\de (h (\G))$ has exactly one term whose the underlying graph 
belongs to the first subset of $\gra(\beth)^{r_0 + r}_{n_0 + n}$. Moreover, this term 
coincides with $\G$. 

Let $i$ be the label of the cut vertex $\mmp(c_0)$ in $\G$. Since the linear 
combination $\G \circ_i \G_{\ed}$ does not involve graphs in which the additional 
edge is separating, $\de_i(\G)$ in \eqref{delta-dfn} is zero.  
Hence, 
$$
-h (\de(\G)) 
$$
coincides with the sum of terms in $\de (h (\G))$ whose underlying graphs 
belong to the second subset of $\gra(\beth)^{r_0 + r}_{n_0 + n}$. 
Thus $\de \circ h (\G) + h \circ \de(\G) = \G$.

Let us now assume that $\G$ belongs to the second subset of $\gra(\beth)^{r_0 + r}_{n_0 + n}$.
For such $\G$, we have $h (\G) =0$. 
It is easy to see that $\de(\G)$ has exactly one term $v$ whose underlying graph belongs to 
the first subset of $\gra(\beth)^{r_0 + r+1}_{n_0 + n+1}$. Moreover, $h(v) = \G$. 
Since all the remaining terms of $\de(\G)$ are annihilated by $h$,
$$
\de \circ h (\G) + h \circ \de(\G) = \G. 
$$

Proposition \ref{prop:C-beth} is proved. 
\end{proof}

As we explained above, Proposition \ref{prop:C-beth} implies Proposition \ref{prop:QC-beth-acyclic}. 
Thus the second part of Proposition \ref{prop:GCo-dfGCo3} is proved. 

\appendix

\section{Cohomology of auxiliary complexes}
\label{app:cohomol-T-V2}

Let us denote by $\cV_2$ the two dimensional vector space 
spanned by the degree 1 symbols $a$ and $b$ and denote by
$$
U : = (\und{T}(\cV_2),\mb_{1}),
\qquad
R : = (\und{T}(\cV_2),\mb_{2})
\quad
\textrm{and} 
\quad
P : = (T(\cV_2),\mb_{3})
$$
the cochain complexes with the following differentials:
\begin{equation}
\label{mb1-app}
\mb_{1}(v_1 v_2 \dots v_n) ~:=~ \sum_{i =1}^{n} (-1)^{i+1} \, v_1 \dots v_{i} (a+b) v_{i+1} \dots v_n\,,
\end{equation}
\begin{equation}
\label{mb2-app}
\mb_{2} (v_1 v_2 \dots v_n) ~: =~ 
- (a+b) v_1 \dots v_n + \mb_{1}(v_1 v_2 \dots v_n),
\end{equation}
\vspace{0.07cm}
$$
\mb_{3}(1) : = (a+b)/2, \qquad
\mb_3(a) = \mb_3(b) : = 0, 
$$
\begin{equation}
\label{mb3-app}
\mb_3(v_1 v_2 \dots v_n) ~:=~ \sum_{i =1}^{n-1} (-1)^{i+1} \, v_1 \dots v_{i} (a+b) v_{i+1} \dots v_n\,.
\end{equation}

To analyze these complexes, we consider the truncated tensor algebra $\und{T}(\cW_2)$, 
where $\cW_2$ is the vector space spanned by two symbols $a$ and $x$ of degree $1$. 
We identify $\und{T}(\cW_2)$ (resp. $T(\cW_2)$) with $\und{T}(\cV_2)$  (resp. $T(\cV_2)$)
via the obvious isomorphism which sends $a$ to $a$ and $x$ to $(a+b)$. 
The differentials on (the underlying vector spaces of)
$\und{T}(\cW_2)$ and $T(\cW_2)$ corresponding to 
\eqref{mb1-app},  \eqref{mb2-app}, \eqref{mb3-app} take the 
form\footnote{By abuse of notation, we use the same symbols for the corresponding 
differentials on $\und{T}(\cW_2)$ and  $T(\cW_2)$, respectively.}
$$
\mb_{1}(v_1 v_2 \dots v_n) := \sum_{i =1}^{n} (-1)^{i+1} \, v_1 \dots v_{i} \,x\, v_{i+1} \dots v_n\,,
$$
$$
\mb_{2} (v_1 v_2 \dots v_n) : =
- x \, v_1 \dots v_n + \mb_{1}(v_1 v_2 \dots v_n),
$$
$$
\mb_{3}(1) : = x/2, \qquad
\mb_3(a) = \mb_3(x) : = 0, 
\qquad 
\mb_3(v_1 v_2 \dots v_n) ~:=~ \sum_{i =1}^{n-1} (-1)^{i+1} \, v_1 \dots v_{i} \, x\, v_{i+1} \dots v_n\,.
$$

Let us prove the following statements about the 
cochain complexes $U$ and $R$:
\begin{claim} 
\label{cl:U-and-R}
The cochain complex $U$ is acyclic. As for $R$, we have
$$
H^{n} ( R ) ~ \cong ~ 
\begin{cases}
 \bbK \qquad {\rm if} ~~ n=1,  \\
  \bfzero \qquad {\rm otherwise}\,.
\end{cases}
$$
Moreover, $H^1(R)$ is spanned by the cohomology class of the cocycle $a+b$. 
\end{claim}
\begin{proof} The cochain complex $(\und{T}(\cW_2), \mb_1)$ splits into the direct sum 
of subcomplexes
\begin{equation}
\label{sum-deg-a}
\bigoplus_{m \ge 0} K_m\,,
\end{equation}
where $K_m$ is the subspace of  $\und{T}(\cW_2)$ spanned by monomials of degree $m$ in the symbol $a$.

Since $\mb_1(x^{2 t-1}) = x^{2t}$ and $\mb_1(x^{2t}) = 0$, the cochain 
complex $(K_0, \mb_1 )$ is acyclic. 

To analyze $(K_m, \mb_1)$ for $m \ge 1$, we denote by $K^{\hs}$ the following cochain complex 
\begin{equation}
\label{K-heart}
\dots \, \to \bfzero \to \bfzero  \to \bbK \stackrel{- x \cdot}{ \longrightarrow} \bbK x \stackrel{0}{ \longrightarrow}
\bbK x^2  \stackrel{-x \cdot}{ \longrightarrow} \bbK x^3  \stackrel{0}{ \longrightarrow} 
\bbK x^4  \stackrel{- x \cdot}{ \longrightarrow}  \dots
\end{equation}

It is easy to see that (for every $m \ge 1$), $K_m$ splits into the direct sum 
$$
K_{m, x} \oplus K_{m, a}\,, 
$$
where $K_{m, x}$ (resp. $K_{m, a}$) is the subspace of $K_m$ spanned by monomials 
of the form $x \dots $ (resp. $a \dots $). Moreover, the cochain complex $K_{m,x}$ is isomorphic to 
$K_0 \otimes (\bs K^{\hs})^{\otimes \, m}$ and the cochain complex $K_{m,a}$ is isomorphic to 
$(\bs K^{\hs})^{\otimes \, m}$. 

Since the cochain complexes $K_0$ and $K^{\hs}$ are acyclic, $U$ is also acyclic. 

Let us now consider the cochain complex $R \cong \big(\und{T}(\cW_2), \mb_2 \big)$. 

Just as $U$, the cochain complex  $\big(\und{T}(\cW_2), \mb_2 \big)$ splits into the direct sum
\eqref{sum-deg-a} of subcomplexes $(K_m, \mb_2)$ for $m \ge 0$. 

Since $\mb_2(x^{2n-1}) = 0$ and $\mb_2(x^{2n}) = - x^{2n+1}$ for every positive integer $n$, 
$H^{m}(K_0, \mb_2) = \bfzero$ for every $m \neq 1$, 
$H^1(K_0, \mb_2) \cong \bbK$ and it is spanned by the cohomology class of $x$. 

It is easy to see that, for every $m \ge 1$, the cochain complex $K_m$ is isomorphic to the tensor product  
$$
K^{\hs} \otimes (\bs\, K^{\hs})^{\otimes \, m}\,,
$$
where $K^{\hs}$ is the cochain complex defined in \eqref{K-heart}. 

Since the cochain complex $K^{\hs}$ is acyclic, so is $K_m$ for every $m \ge 1$. 
 
The claim is proved. 
\end{proof}

The following claim takes care of the cochain complex $(T(\cV_2),\mb_{3})$:
\begin{claim}  
\label{cl:P}
The cochain complex $P \cong  \big(T(\cW_2), \mb_3 \big)$ has the following 
cohomology
$$
H^{n} ( P) ~ \cong ~ 
\begin{cases}
\bbK \qquad {\rm if} ~~ n=1,  \\
\bfzero \qquad {\rm otherwise}.
\end{cases}
$$
The cocycles $a$ and $-b$ are cohomologous and $H^1(P)$ is spanned by the cohomology class of $a$. 
\end{claim}  
\begin{proof} Just as for $U$ and $R$ the cochain complex $P$ splits into the 
direct sum of subcomplexes
\begin{equation}
\label{P-sum}
\bigoplus_{m \ge 0} K^{-}_{m},
\end{equation}
where $K^{-}_m$ is the subspace of $T(\cW_2)$ spanned by  monomials of degree $m$ in the symbol $a$.

Since $\mb_3 (1) = x/2$, $\mb_3(x^{2 t-1}) =0$ and $\mb_3(x^{2 t}) = x^{2t+1}$ for every $t \ge 1$, the 
complex $(K^{-}_{0}, \mb_3)$ is acyclic.  

As for $m \ge 2$, it is easy to see that, up to the shift of degree, $(K^{-}_m, \mb_3)$ is isomorphic to the cochain complex 
$$
(\bs\, K^{\hs})^{\otimes \, m-1}
$$
where $K^{\hs}$ is the cochain complex defined in \eqref{K-heart}. 

Finally, $K^{-}_1$ is the following cochain complex with the zero differential 
$$
 \dots \to \bfzero \to \bbK a \to \bfzero \to \bfzero \to \dots
$$

The desired statements follow.
\end{proof}

\section{The subcomplexes of polygons and path graphs}
\label{app:polygons-paths}
Let us recall that $\dfGCo_{\conn, \dia}$ is the subcomplex of $\dfGCo_{\conn}$ spanned by (even) connected graphs with 
all vertices having valency $2$ (i.e. $\G_{\lp}$ and various polygons) and $\dfGCo_{\conn, -}$ is the subcomplex of $\dfGCo_{\conn}$ 
spanned by path graphs and the graph $\G_{\bul} \in \dgra_1^0$. 
 
In this appendix, we prove Propositions \ref{prop:dfGCo-diamond} and \ref{prop:dfGCo-path} which 
take care of the complexes  $\dfGCo_{\conn, \dia}$ and $\dfGCo_{\conn, -}$.

First, we observe that the degree $n$ term $U^n$ of the cochain complex $U \cong \big( \und{T}(\cW_2), \mb_1 \big)$ is equipped 
with the obvious action of the cyclic group $G_n$ of order $n$. The generator $g_n$ of this group acts as
\begin{equation}
\label{cyclic-acts}
g_n (v_1, v_2,\dots, v_n) = (-1)^{n-1} (v_2, v_3,\dots, v_n, v_1). 
\end{equation}

Moreover, for every $X \in U^n$, the vector $\mb_1 (X - g_n (X))$ belongs 
to the subspace spanned by vectors of the form 
$$
Y - g_{n+1} (Y), \qquad Y \in U^{n+1}\,. 
$$
Hence $\mb_1$ induces a differential on the cochain complex 
\begin{equation}
\label{coinvar-cyclic}
\bfzero \to U^1_{G_1} \to U^2_{G_2} \to \dots \to U^n_{G_n} \to \dots\,. 
\end{equation}
By abuse of notation, we will use the same symbol $\mb_1$ for the differential on the cochain 
complex \eqref{coinvar-cyclic}. We denote by the resulting cochain complex by $K_{\dia}$. 

Just as $U$, the cochain complex $K_{\dia}$ splits into the direct 
sum of subcomplexes 
$$
K_{\dia}  \cong \bigoplus_{r \ge 0} K_{\dia, r}\,,
$$ 
where $K_{\dia, r}$ is spanned by (images of) monomials in which $a$ appears exactly $r$ times. 

It is easy to see that, for every $r \ge 1$, $K_{\dia, r}$ is obtained from the cochain complex 
$\big( \bs\, K^{\hs} \big)^{\otimes \, r}$
by taking coinvariants with respect to the action of the cyclic group $G_r$. 
Here, $K^{\hs}$ is the cochain complex defined in \eqref{K-heart}. 

Thus, since our base field has characteristic zero and $K^{\hs}$ is acyclic, 
the K\"unneth theorem implies that  $K_{\dia, r}$ is acyclic 
for every $r \ge 1$. 

It is clear that $x^n = 0$ in $\big( T^{n}(\cW_2) \big)_{G_n}$ for every even (positive) integer $n$ and 
$x^n \neq 0$ in  $\big( T^{n}(\cW_2) \big)_{G_n}$ for every odd (positive) integer $n$. Therefore 
$K_{\dia, 0}$ is the following cochain complex with the zero differential 
$$
\cdots  \to \bfzero \to \bfzero \to \bfzero \to \bbK x \to \bfzero \to \bbK x^3 \to \bfzero \to  \bbK x^5 \to \bfzero  \to \cdots\,.
$$

Thus the cochain complex \eqref{coinvar-cyclic} has the following cohomology: 
\begin{equation}
\label{H-K-dia}
H^{m} ( K_{\dia} )  \cong 
\begin{cases}
\bbK  \qquad \textrm{if} ~~ m \textrm{~~is positive and odd}, \\
\bfzero \qquad \textrm{otherwise}.
\end{cases}
\end{equation}
Moreover, $H^{2n+1} ( K_{\dia} )$ (for $n \ge 0$) is spanned by the cohomology class represented by $(a+b)^{2n+1}$. 

\subsection{The end of the proof of Proposition \ref{prop:dfGCo-diamond}}
Let us recall (see \eqref{S2-action}) that the cochain complex  $(U, \mb_1)$ 
is equipped with an action of $\bbS_2$. It is easy to see that this action descends to 
the cochain complex $K_{\dia}$ (see \eqref{coinvar-cyclic}). 

To every monomial $X = v_1 v_2 \dots v_n$ in $U$ we assign the cycle graph $\G_X \in \dgra_n^n$ which 
is obtained as follows: if $v_i = a$ and $i \le n-1$ then the $i$-th edge originates from vertex $i$ and terminates at 
vertex $i+1$; if $v_i = b$ and $i \le n-1$ then the $i$-th edge originates from vertex $i+1$ and terminates at vertex $i$; 
finally, the $n$-th edges connects vertex $n$ to vertex $1$; it originates from vertex $n$ (resp. $1$) if $v_n = a$ (resp. $v_n = b$). 

It is clear that the assignment $X \mapsto \G_X$ extends to the surjective map of cochain complexes 
\begin{equation}
\label{from-K-dia}
K_{\dia}  \to  \dfGCo_{\conn, \dia}\,.
\end{equation}
Moreover a vector $Y \in K_{\dia}$ belongs to the kernel of this map if and only if 
$Y$ belongs to the span of vectors of the form $X - \si(X), $
where $X \in K_{\dia}$ and $\si$ is the only non-trivial element of $\bbS_2$. 

Thus \eqref{from-K-dia} induces an isomorphism of cochain complexes 
$\big( K_{\dia} \big)_{\bbS_2} \cong  \dfGCo_{\conn, \dia}$ and 
Proposition \ref{prop:dfGCo-diamond} follows from \eqref{H-K-dia} and 
the fact that the image of $(a+b)^{2n+1}$ in $\big( K_{\dia} \big)_{\bbS_2}$ is
non-zero if and only if $n$ is divisible by $2$.  \qed

\subsection{The subcomplex $\dfGCo_{\conn, -}$ is indeed acyclic}
\label{app:paths}

To every monomial $X  = v_1v_2 \dots v_n \in \cV_2^{\otimes\, n}$\,, we assign the path graph 
$\G^{-}_X \in \dgra_{n+1}^n$ by declaring that, if $v_i = a$ (resp. $v_i = b$), then edge $\und{i}$ 
originates at vertex $\gray{i}$ (resp. $\gray{i+1}$) and terminates at vertex $\gray{i+1}$ 
(resp. $\gray{i}$). Setting, 
$$
\psi(X) : = \G^{-}_X\,, \qquad \psi(1) : = \G_{\bul}
$$  
we get a surjective map of cochain complexes $\psi : P \to \dfGCo_{\conn, -}$.  

It is easy to see that a vector $Y \in P$ belongs to the kernel of $\psi$ if and only if 
$Y$ belongs to the span of vectors of the form $X - \si (X),$
where $X$ is a monomial in $P$, $\si = (1,2) \in \bbS_2$ and the action 
of $\bbS_2$ on $P$ is defined by  \eqref{S2-action}. 

Due to Claim \ref{cl:P}, $H^{\bul}(P)$ is spanned by the cohomology class of $a$. 
Since, in the space of coinvariants $P_{\bbS_2}$, we have $a = (a+b)/2$, 
the cochain complex $P_{\bbS_2}$ is acyclic. Thus the cochain complex 
$\dfGCo_{\conn, -}$ is also acyclic. \qed

~\\

\noindent\textsc{Department of Mathematics,
Temple University, \\
Wachman Hall Rm. 638\\
1805 N. Broad St.,\\
Philadelphia PA, 19122 USA \\
\emph{E-mail address:} {\bf vald@temple.edu}}

~\\

\noindent\textsc{Department of Mathematics and Statistics\\
University of Nevada, Reno\\
1664 N. Virginia Street\\
Reno, NV 89557-0084 USA\\
\emph{E-mail address:} {\bf chrisrogers@unr.edu, chris.rogers.math@gmail.com}}

~\\


\begin{thebibliography}{99}

\bibitem{AT} A. Alekseev and C. Torossian,  The Kashiwara-Vergne conjecture and Drinfeld's associators, 
Ann. of Math. (2) {\bf 175}, 2 (2012) 415--463;
arXiv:0802.4300.

\bibitem{A-Turchin} G. Arone and V. Turchin, Graph-complexes computing the rational homotopy of 
high dimensional analogues of spaces of long knots, Ann. Inst. Fourier (Grenoble) {\bf 65}, 1 (2015) 1--62; 
arXiv:1108.1001.  

\bibitem{Bar-Natan-GC} D. Bar-Natan and B. McKay,
Graph Cohomology -- An Overview and Some Computations,
available at \url{https://www.math.toronto.edu/~drorbn/papers/GCOC/GCOC.ps}

\bibitem{Cattaneo} A.S. Cattaneo, P. Cotta-Ramusino, and R. Longoni, 
Configuration spaces and Vassiliev classes in any dimension, 
Algebr. Geom. Topol. {\bf 2} (2002) 949--1000. 

\bibitem{C-Vogtmann} J. Conant, F. Gerlits and K. Vogtmann, Cut vertices in commutative graphs, 
Q. J. Math. {\bf 56}, 3 (2005) 321--336. 

\bibitem{hairy-stuff} J. Conant, M. Kassabov and K. Vogtmann, 
Higher hairy graph homology, Geom. Dedicata {\bf 176} (2015) 345--374.

%

\bibitem{stable} V.A. Dolgushev, Stable formality quasi-isomorphisms for Hochschild cochains,  arXiv:1109.6031.

\bibitem{slides} V.A. Dolgushev, The Intricate Maze of Graph Complexes, slides are available at 
\url{https://math.temple.edu/~vald/GCtalk.pdf}

\bibitem{notes} V. A. Dolgushev and C. L. Rogers, 
Notes on Algebraic Operads, Graph Complexes, and Willwacher's 
Construction, in {\it Mathematical aspects of quantization,} 25--145, 
Contemp. Math., {\bf 583}, Amer. Math. Soc., Providence, RI, 2012;
arXiv:1202.2937.

\bibitem{DeligneTw} V.A. Dolgushev and T. H. Willwacher, 
Operadic Twisting -- with an application to Deligne's conjecture,
J. Pure Appl. Algebra, {\bf 219}, 5 (2015) 1349--1428; 
arXiv:1207.2180.

\bibitem{Drinfeld} V.G. Drinfeld, On quasitriangular quasi-Hopf algebras and
on a group that is closely connected with ${\rm Gal}(\overline{\bbQ}/\bbQ)$.
(Russian) Algebra i Analiz {\bf 2}, 4 (1990) 149--181;
translation in Leningrad Math. J. {\bf 2}, 4 (1991) 829--860.

\bibitem{Ham} A. Hamilton, 
Classes on the moduli space of Riemann surfaces through a noncommutative Batalin-Vilkovisky 
formalism, Adv. Math. {\bf 243} (2013) 67--101. 

\bibitem{Ham-Lazar} A. Hamilton and A. Lazarev,
Characteristic classes of $A_{\infty}$-algebras, 
J. Homotopy Relat. Struct. {\bf 3}, 1 (2008) 65--111; 
arXiv:0801.0904. 

\bibitem{Ham-Lazar1} A. Hamilton and A. Lazarev,
Graph cohomology classes in the Batalin-Vilkovisky formalism,
J. Geom. Phys. {\bf 59}, 5 (2009) 555--575;  arXiv:math/0701825.



\bibitem{Igusa} K. Igusa, Graph cohomology and Kontsevich cycles, 
Topology {\bf 43}, 6 (2004) 1469--1510. 

\bibitem{KWZ} A. Khoroshkin, T. Willwacher, and M. Zivkovic,
Differentials on graph complexes, Adv. Math. {\bf 307} (2017) 1184--1214; 
arXiv:1411.2369. 

\bibitem{KWZ11}  A. Khoroshkin, T. Willwacher, and M. Zivkovic,
Differentials on graph complexes II: hairy graphs, 
Lett. Math. Phys. {\bf 107}, 10 (2017) 1781--1797; arXiv:1508.01281. 


\bibitem{K} M. Kontsevich, Deformation quantization of
Poisson manifolds, Lett. Math. Phys., {\bf 66} (2003) 157-216; 
q-alg/9709040.

\bibitem{K-conj} M. Kontsevich, Formality conjecture,
{\it Deformation theory and symplectic geometry} (Ascona, 1996),
139--156, Math. Phys. Stud., 20, Kluwer Acad. Publ.,
Dordrecht, 1997.

\bibitem{K-noncom-symp} M. Kontsevich, 
Formal (non)commutative symplectic geometry, 
{\it The Gel'fand Mathematical Seminars,} 1990--1992, 
173--187, Birkh\"auser Boston, Boston, MA, 1993. 

\bibitem{K-M-Volic} R. Koytcheff, B.A. Munson, I. Voli\'c,
Configuration space integrals and the cohomology of the space of homotopy string links,
J. Knot Theory Ramifications {\bf 22}, 11 (2013) 73 pp.


\bibitem{LV-book} J.-L. Loday and B. Vallette,
Algebraic operads, {\it Grundlehren der Mathematischen Wissenschaften}, 
{\bf 346}, Springer, Heidelberg, 2012.


%


\bibitem{MV-nado} S. Merkulov and B. Vallette, 
Deformation theory of representations of prop(erad)s. I, 
J. Reine Angew. Math. {\bf 634} (2009) 51--106;
arXiv:0707.0889. 

\bibitem{Sergei-Thomas} S. Merkulov and T. Willwacher, 
Props of ribbon graphs, involutive Lie bialgebras and moduli spaces of curves, 
arXiv:1511.07808.

\bibitem{Sinha} D.P. Sinha, Operads and knot spaces, J. Amer. Math. Soc. {\bf 19}, 2 (2006) 461--486. 



\bibitem{Thomas-Ger-knots} T. Willwacher, Deformation quantization and the Gerstenhaber structure on 
the homology of knot spaces, arXiv:1506.07078.

\bibitem{Thomas-slides} T. Willwacher, Graph complexes, slides are available at
\url{https://people.math.ethz.ch/~wilthoma/docs/jobslides.pdf}

\bibitem{Thomas} T. Willwacher, M. Kontsevich's graph complex and the Grothendieck-Teichm\"uller Lie algebra, 
Invent. Math. {\bf 200}, 3 (2015) 671--760; arXiv:1009.1654.

 
\bibitem{WZ} T. Willwacher and M. \v{Z}ivkovi\'c,
Multiple edges in M. Kontsevich's graph complexes and computations of the dimensions and Euler characteristics, 
Adv. Math. {\bf 272} (2015) 553--578; arXiv:1401.4974. 
    
\end{thebibliography}
\end{document}